\documentclass[oneside,english]{amsart}
\usepackage[T1]{fontenc}
\usepackage[latin9]{inputenc}
\usepackage{verbatim}
\usepackage{prettyref}
\usepackage{amstext}
\usepackage{amsthm}
\usepackage{amssymb}
\usepackage{graphicx}
\usepackage{esint}
\usepackage[normalem]{ulem}

\usepackage{color}

\makeatletter
\numberwithin{equation}{section}
\numberwithin{figure}{section}
  \theoremstyle{plain}
  \newtheorem{thm}{\protect\theoremname}[section]
  \theoremstyle{remark}
  \newtheorem{rem}{\protect\remarkname}[section]
  \theoremstyle{plain}
  \newtheorem{prop}{\protect\propositionname}[section]
 \ifx\proof\undefined\
   \newenvironment{proof}[1][\proofname]{\par
     \normalfont\topsep6\p@\@plus6\p@\relax
     \trivlist
     \itemindent\parindent
     \item[\hskip\labelsep
           \scshape
       #1]\ignorespaces
   }{%
     \endtrivlist\@endpefalse
   }
   \providecommand{\proofname}{Proof}
 \fi
  \theoremstyle{plain}
  \newtheorem{cor}{\protect\corollaryname}[section]
  \theoremstyle{plain}
  \newtheorem{lem}{\protect\lemmaname}[section]
  \theoremstyle{plain}
  
  \theoremstyle{definition}
  
  \theoremstyle{definition}

\usepackage{mathrsfs,amsthm,amssymb,bbm,fullpage}
\usepackage{hyperref}

\newcommand{\R}{\mathbb{R}}

\newcommand{\indicator}[1]{\mathbbm{1}_{#1}}

\newrefformat{prob}{Problem \ref{#1}}
\newrefformat{prop}{Proposition \ref{#1}}
\newrefformat{cor}{Corollary \ref{#1}}
\newrefformat{conj}{Conjecture \ref{#1}}
\newrefformat{rem}{Remark \ref{#1}}
\newrefformat{fig}{Figure \ref{#1}}
\newrefformat{subsec}{Section \ref{#1}}

\@ifundefined{showcaptionsetup}{}{%
 \PassOptionsToPackage{caption=false}{subfig}}
\usepackage{subfig}
\makeatother

\usepackage{babel}
\providecommand{\conjecturename}{Conjecture}
\providecommand{\corollaryname}{Corollary}
\providecommand{\definitionname}{Definition}
\providecommand{\lemmaname}{Lemma}
\providecommand{\problemname}{Problem}
\providecommand{\propositionname}{Proposition}
\providecommand{\remarkname}{Remark}
\providecommand{\theoremname}{Theorem}

\begin{document}

\title{On the optimal design of wall-to-wall heat transport}

\author{Charles R.\ Doering and Ian Tobasco}

\date{\today}
\begin{abstract}
We consider the problem of optimizing heat transport 
through an incompressible fluid layer. Modeling passive scalar transport
by advection-diffusion, we maximize the mean rate of
total transport by a divergence-free velocity field. Subject
to various boundary conditions and intensity constraints, we prove
that the maximal rate of transport scales linearly in the r.m.s.\
kinetic energy and, up to possible logarithmic corrections, as the
$1/3$rd power of the mean enstrophy in the advective regime. This
makes rigorous a previous prediction on the near optimality of convection
rolls for energy-constrained transport. Optimal designs for enstrophy-constrained
transport are significantly more difficult to describe:
we introduce a ``branching'' flow design with an unbounded number
of degrees of freedom and prove it achieves nearly optimal transport.
The main technical tool behind these results is a variational principle
for evaluating the transport of candidate designs. The principle
admits dual formulations for bounding transport from above and below.
While the upper bound is closely related to the ``background method'',
the lower bound reveals a connection between the optimal design problems
considered herein and other apparently related model problems from
mathematical materials science. These connections serve to motivate
 designs.

\end{abstract}

\maketitle
\tableofcontents{}

\section{Introduction}

\subsection{The wall-to-wall optimal transport problem}

This paper concerns a class of optimal design problems from fluid
dynamics that asks to maximize the overall transport of heat through
an incompressible fluid layer. Passive scalar transport by an incompressible
fluid is governed by the advection-diffusion equation
\begin{equation}
\partial_{t}T+\mathbf{u}\cdot\nabla T=\kappa\Delta T\label{eq:advection-diffusion-kappa}
\end{equation}
where $T(\mathbf{x},t)$ is the scalar field undergoing transport,
referred to as temperature throughout, $\mathbf{u}(\mathbf{x},t)$
is the velocity vector field of the fluid, and $\kappa$ is the coefficient
of molecular diffusivity. In general, the velocity field $\mathbf{u}$
and temperature $T$ may depend on both space $\mathbf{x}=(x,y,z)$
and time $t$. Due to incompressibility, $\mathbf{u}$ must remain
divergence-free. Thinking of $\mathbf{u}$ as being in our control,
we set ourselves the task of choosing it to maximize the overall transport
of heat determined by \prettyref{eq:advection-diffusion-kappa}.

This is a rich class of optimal design problems and we are interested
in the dependence of any solutions, i.e., \emph{optimal designs},
on various constraints that may be imposed. We discuss specific constraints
for the velocities later on, but let us handle the temperature field
first. Supposing the fluid is contained between two impenetrable parallel
planar walls at a distance $h$, we fix the temperature at the walls
by imposing the constant Dirichlet boundary conditions
\begin{equation}
T|_{z=0}=T_{\text{hot}}\quad\text{and}\quad T|_{z=h}=T_{\text{cold}}.\label{eq:Dirichlet-dimesional}
\end{equation}
If the velocity field $\mathbf{u}$ is regular enough \textemdash eventually
our constraints on it will ensure this \textemdash the advection-diffusion
equation \prettyref{eq:advection-diffusion-kappa} admits a unique
solution $T$ satisfying \prettyref{eq:Dirichlet-dimesional} for
every essentially bounded initial temperature field $T|_{t=0}=T_{0}(\mathbf{x})$.
We see, therefore, that the overall heat transport specified by \prettyref{eq:advection-diffusion-kappa}
should depend in general on $\mathbf{u}$ and $T_{0}$. However, as
the partial differential equation (PDE) \prettyref{eq:advection-diffusion-kappa}
is dissipative, any dependence on the initial temperature $T_{0}$
is eventually lost as $t\to\infty$, and the resulting heat transport
can be thought of as being set by $\mathbf{u}$ alone.

In this paper, we study the optimal design of wall-to-wall heat transport
in the long-time limit, subject to various boundary conditions and
intensity constraints on the velocity field $\mathbf{u}$. To simplify
matters, we consider all fields to be periodic in the wall-parallel
variables $x$ and $y$ with periods $l_{x}$ and $l_{y}$. That is,
we take $\mathbf{x}$ to belong to the domain $\Omega=\mathbb{T}_{xy}^{2}\times[0,h]_{z}$,
identifying $\mathbb{T}_{xy}^{2}$ with $[0,l_{x}]\times[0,l_{y}]$
in the usual way. We turn now to discuss the precise measure of overall
heat transport that will be optimized throughout. 

\subsubsection{Finite-time wall-to-wall optimal transport}

According to the advection-diffusion equation \prettyref{eq:advection-diffusion-kappa}
and the boundary conditions \prettyref{eq:Dirichlet-dimesional},
the vertically averaged rate of heat transport per unit area up to
time $t=\tau$ is given by
\[
J_{\tau}=\frac{1}{\tau}\frac{1}{l_{x}l_{y}h}\int_{0}^{\tau}\int_{\Omega}\hat{k}\cdot(\mathbf{u}T-\kappa\nabla T)\,d\mathbf{x}dt=\frac{\kappa}{h}(T_{hot}-T_{cold})+\fint_{0}^{\tau}\fint_{\Omega}wT\,d\mathbf{x}dt.
\]
Here, $\mathbf{u}=u\hat{i}+v\hat{j}+w\hat{k}$ and $\fint$ denotes
an average over the integration domain. We are interested in determining
those velocity fields which maximize the overall heat transport $J_{\tau}$.
Of course, unless $\mathbf{u}$ is suitably constrained the optimal
transport $\sup\,J_{\tau}$ will be infinitely large. It is natural to prescribe the overall
magnitude of $\mathbf{u}$, and to enforce suitable boundary
conditions at the walls $\partial\Omega$. The resulting optimal
design problems take the form
\begin{equation}
\sup_{\substack{\mathbf{u}(\mathbf{x},t)\\
||\mathbf{u}||=U\\
+b.c.
}
}\,J_{\tau}\label{eq:finite-timepblm}
\end{equation}
where the parameter $U$ sets the advective intensity of the admissible
velocity fields.

We consider two classes of admissible velocity fields which we refer to as being ``energy-''
or ``enstrophy-constrained''. In the energy-constrained class, we take
\[
||\mathbf{u}||^{2}=\fint_{0}^{\tau}\fint_{\Omega}|\mathbf{u}|^{2}\,d\mathbf{x}dt
\]
in \prettyref{eq:finite-timepblm} so that the constraint
$||\mathbf{u}||=U$ sets the average kinetic energy available for
advection. As for boundary conditions, the no-penetration ones
\[
w|_{\partial\Omega}=0
\]
are well-suited to this class. We call the problem that results the 
finite-time energy-constrained wall-to-wall optimal transport
problem. The finite-time enstrophy-constrained\footnote{For various boundary conditions including the ones considered here,
the mean square rate of strain $||\nabla\mathbf{u}||_{L^{2}(\Omega)}^{2}$
and enstrophy $||\nabla\times\mathbf{u}||_{L^{2}(\Omega)}^{2}$ are
the same.} problem arises from taking
\[
||\mathbf{u}||^{2}=h^{2}\fint_{0}^{\tau}\fint_{\Omega}|\nabla\mathbf{u}|^{2}\,d\mathbf{x}dt
\]
and enforcing the no-slip boundary conditions
\[
\mathbf{u}|_{\partial\Omega}=\mathbf{0}
\]
in \prettyref{eq:finite-timepblm}. The essential results of this
paper hold as well for the stress-free boundary conditions
\[
w|_{\partial\Omega}=0\quad\text{and}\quad\partial_{z}u|_{\partial\Omega}=\partial_{z}v|_{\partial\Omega}=0,
\]
although our focus is mostly on the no-slip ones. 


\subsubsection{Infinite-time wall-to-wall optimal transport}

Having introduced the finite-time energy- and enstrophy-constrained
wall-to-wall optimal transport problems, we turn to discuss their
infinite-time analogs which are the focus of this paper. Let $\left\langle \cdot\right\rangle $ denote
the (limit superior) space and long-time average
\[
\left\langle f\right\rangle =\limsup_{\tau\to\infty}\,\fint_{0}^{\tau}\fint_{\Omega}f(\mathbf{x},t)\,d\mathbf{x}dt.
\]
As an integration by parts shows, the space and long-time averaged
heat transport determined by \prettyref{eq:advection-diffusion-kappa}
satisfies
\[
\limsup_{\tau\to\infty}\,J_{\tau}=\kappa\left\langle |\nabla T|^{2}\right\rangle .
\]
Note this depends on $\mathbf{u}$ but not on the initial temperature
$T_{0}$ so long as it is bounded. 
In direct analogy with the finite-time optimal transport
problems, we define the infinite-time energy-constrained wall-to-wall
optimal transport problem by
\begin{equation}
\sup_{\substack{\mathbf{u}(\mathbf{x},t)\\
\left\langle |\mathbf{u}|^{2}\right\rangle =U^{2}\\
w|_{\partial\Omega}=0
}
}\,\kappa\left\langle |\nabla T|^{2}\right\rangle \label{eq:energy-constrained-in-units}
\end{equation}
and the infinite-time enstrophy-constrained problem by
\begin{equation}
\sup_{\substack{\mathbf{u}(\mathbf{x},t)\\
\left\langle |\nabla\mathbf{u}|^{2}\right\rangle =\frac{U^{2}}{h^{2}}\\
\mathbf{u}|_{\partial\Omega}=\mathbf{0}
}
}\,\kappa\left\langle |\nabla T|^{2}\right\rangle .\label{eq:enstrophy-constrained-in-units}
\end{equation}
It is these infinite-time optimal design problems that we study in
the remainder of this paper. As we never return to the finite-time
problems, we discontinue the use of the distinguishing phrases from
now on.

A word is in order regarding the sense in which we consider \prettyref{eq:energy-constrained-in-units}
and \prettyref{eq:enstrophy-constrained-in-units} to be solved. We
do not claim that there must exist maximizers for either problem. 
Although this certainly merits investigation, and is related
to questions of $\Gamma$-convergence \cite{braides2002gamma} of
the finite-time problems to the infinite-time ones, we choose in this
paper to focus instead on the maximum \emph{value} of transport which
is always well-defined. To the maximum value is associated maximizing
sequences, i.e., near optimizers which we may seek to describe. Even
in the steady versions of \prettyref{eq:energy-constrained-in-units}
and \prettyref{eq:enstrophy-constrained-in-units} \textemdash where
all fields are assumed to be independent of time and optimal designs
are guaranteed to exist \textemdash determining the maximal transport
is a non-trivial task. 

The energy- and enstrophy-constrained wall-to-wall optimal transport
problems \prettyref{eq:energy-constrained-in-units} and \prettyref{eq:enstrophy-constrained-in-units}
were introduced in \cite{hassanzadeh2014wall} and studied further
in \cite{SD} by a combination of asymptotic and numerical methods.
Similar methods have since been applied to study other related optimal
transport problems \cite{alben2017improved,marcotte2018optimal,motoki2018optimal}.
A key question left unresolved by these works is whether the local
maximizers constructed therein actually achieve heat transport comparable
to that of global optimizers. In this paper, we present a new mathematically
rigorous approach to answering this question. Our methods do not rely
on the use of Euler-Lagrange equations; as these are non-concave maximization
problems with many local maximizers, critical point conditions do
not suffice to identify global optimizers. Rather, our starting point
is a new variational formula for evaluating wall-to-wall heat transport,
which is useful both for proving \emph{a priori} upper bounds on optimal
transport as well as lower bounds on the transport of candidate designs.
For the energy-constrained problem, we prove that the convection roll
designs from \cite{hassanzadeh2014wall} achieve globally optimal
heat transport up to a universal prefactor in the advection-dominated
regime. For the enstrophy-constrained problem, we construct a new
class of ``branching'' designs featuring a large and potentially
unbounded number of degrees of freedom. A well-chosen branching design
achieves optimal transport up to possible logarithmic corrections. 

The wall-to-wall optimal transport problem is naturally related to
the study of transport in turbulent fluids. One consequence of our
results is a proof that any flows arising in Rayleigh's original two-dimensional
model of buoyancy-driven convection between stress-free walls
\cite{Rayleigh1916} must achieve significantly sub-optimal rates
of heat transport in the large Rayleigh number regime $Ra\gg1$. Indeed,
while our results imply the existence of incompressible flows achieving transport 
consistent with the proposed ``ultimate scaling'' law $Nu\sim Ra^{1/2}$ (up to logarithmic corrections), 
such transport is impossible in Rayleigh's original model \cite{whitehead2011ultimate}.
In fact, our analysis leads us to wonder whether such logarithmic corrections to
scaling should always hold, independent of dimension or boundary
conditions. Behind these claims is a more or less explicit connection
between the fluid dynamical optimal design problems considered herein
and other apparently related model problems from the study of ``energy-driven
pattern formation'' in materials science \cite{kohn2007energy}.
We discuss these considerations in detail at the end. A preliminary version of our methods and results was announced in \cite{tobasco2017optimal}.

\subsection{Main results and methods}

\subsubsection{Non-dimensionalization}

We are concerned with the dependence of energy- and enstrophy-constrained
wall-to-wall optimal transport  \prettyref{eq:energy-constrained-in-units}
and \prettyref{eq:enstrophy-constrained-in-units} in their parameters. 
We make use of two standard non-dimensional quantities.
The Pecl\'et number
\[
Pe=\frac{Uh}{\kappa}
\]
is a dimensionless measure of the intensity of advection relative
to that of diffusion. Transport by \prettyref{eq:advection-diffusion-kappa}
is dominated by advection when $Pe\gg1$ and by diffusion when $Pe\ll1$.
The Nusselt number $Nu$ is a dimensionless measure of the enhancement
of heat transport by convection over that of pure conduction. In the
fluid layer geometry, 
\[
Nu(\mathbf{u})=\frac{\kappa\left\langle |\nabla T|^{2}\right\rangle }{\frac{\kappa}{h^{2}}\left(T_{hot}-T_{cold}\right)^{2}}=1+\frac{h}{\kappa}\frac{1}{T_{hot}-T_{cold}}\left\langle wT\right\rangle .
\]
Note this does not depend on the initial temperature $T_{0}$. 

Such non-dimensionalization reduces the number of free parameters
in \prettyref{eq:energy-constrained-in-units}
and \prettyref{eq:enstrophy-constrained-in-units} to three: the dimensionless group $Pe$ and the aspect
ratios of the domain $\frac{h}{l_{x}}$ and $\frac{h}{l_{y}}$. That
is, it suffices to take
\[
h=\kappa=T_{hot}=1\quad\text{and}\quad T_{cold}=0
\]
and study the dependence of the resulting non-dimensionalized optimal
transport problems
\[
\sup_{\substack{\mathbf{u}(\mathbf{x},t)\\
\left\langle |\mathbf{u}|^{2}\right\rangle =Pe^{2}\\
w|_{\partial\Omega}=0
}
}\,Nu(\mathbf{u})\qquad\text{and}\qquad\sup_{\substack{\mathbf{u}(\mathbf{x},t)\\
\left\langle |\nabla\mathbf{u}|^{2}\right\rangle =Pe^{2}\\
\mathbf{u}|_{\partial\Omega}=\mathbf{0}
}
}\,Nu(\mathbf{u})
\]
on $Pe$, $l_{x}$, and $l_{y}$. Henceforth, we understand the Nusselt number to be given by
\begin{equation}
Nu(\mathbf{u})=\left\langle |\nabla T|^{2}\right\rangle =1+\left\langle wT\right\rangle \label{eq:Nusseltdefn}
\end{equation}
where $T$ is determined from $\mathbf{u}$ by solving the advection-diffusion
equation
\begin{equation}
\partial_{t}T+\mathbf{u}\cdot\nabla T=\Delta T\label{eq:advection-diffusion}
\end{equation}
with Dirichlet boundary conditions 
\[
T|_{z=0}=1\quad\text{and}\quad T|_{z=1}=0
\]
and any essentially bounded initial data $T|_{t=0}=T_{0}$ (the choice
of which is immaterial to our results). The domain $\Omega=\mathbb{T}_{xy}^{2}\times I_{z}$
where $\mathbb{T}_{xy}^{2}$ is identified with $[0,l_{x}]\times[0,l_{y}]$
and $I_{z}=[0,1]$. As always, $\mathbf{u}$ is understood to be divergence-free. 

\subsubsection{Summary of main results}

Our results concern the asymptotic dependence of optimal transport
in the advective regime $Pe\gg1$. Concerning energy-constrained transport,
we find that the maximal transport rate scales linearly in the
r.m.s.\ kinetic energy as $Pe\to\infty$. More precisely, we prove
the following result:
\begin{thm}
\label{thm:mainbounds_w2w_eng} There exist positive constants $C$
and $C'$ so that
\[
CPe\leq\sup_{\substack{\mathbf{u}(\mathbf{x},t)\\
\left\langle |\mathbf{u}|^{2}\right\rangle =Pe^{2}\\
w|_{\partial\Omega}=0
}
}\,Nu(\mathbf{u})\leq\frac{1}{2}Pe
\]
for all $Pe\geq C'$. The constant $C$ is independent of all parameters
and $C'$ depends only on the aspect ratios of the domain. 
\end{thm}
As noted in \cite{hassanzadeh2014wall}, the \emph{a priori} upper bound $Nu\lesssim Pe$
can be proved by a quick application of the maximum principle and
the Cauchy-Schwarz inequality. On the other hand, to prove the lower
bound one must construct a certain family of admissible velocity fields
$\{\mathbf{u}_{Pe}\}$ and prove that their Nusselt numbers scale
linearly in $Pe$ in the advective regime. Such a construction was
described using methods of matched asymptotic analysis in \cite{hassanzadeh2014wall}
(albeit with no attempt to control the errors in the ensuing estimates).
Our construction is inspired by that one: we consider a convection
roll system as in \prettyref{fig:rollsrough} and choose the number
of rolls to scale optimally in $Pe$. Our approach to evaluating $Nu$ allows to rigorously justify the predictions from \cite{hassanzadeh2014wall}
regarding the (near) optimality of such flows.

The enstrophy-constrained problem turns out to be much more
difficult to resolve. We prove that the maximal enstrophy-constrained transport rate scales, up to possible
logarithmic corrections, as the $2/3$rd power of the r.m.s.\ rate-of-strain
as $Pe\to\infty$. Furthermore, we obtain a bound on
the size of any corrections to this scaling:
\begin{thm}
\label{thm:mainbounds_w2w_ens} There exist positive constants $C$,
$C'$, and $C''$ so that

\[
C\frac{1}{\log^{4/3}Pe}Pe^{2/3}\leq\sup_{\substack{\mathbf{u}(\mathbf{x},t)\\
\left\langle |\nabla\mathbf{u}|^{2}\right\rangle =Pe^{2}\\
\mathbf{u}|_{\partial\Omega}=0
}
}\,Nu(\mathbf{u})\leq C'Pe^{2/3}
\]
for all $Pe\geq C''$. The constants $C$ and $C'$ are independent
of all parameters and $C''$ depends only on the aspect ratios of
the domain.
\end{thm}
\begin{rem} \label{rem:bcremark}
The same bounds apply to enstrophy-constrained optimal transport between no-penetration or stress-free walls.
Indeed, by a simple inclusion argument, maximal transport between impenetrable walls is never less than for stress-free walls, and both are bounded below by maximal transport between no-slip walls.
Since the \emph{a priori} upper bound $Nu\lesssim Pe^{2/3}$ applies so long as $w|_{\partial\Omega}=0$ (this is what is proved in \prettyref{sec:aprioribds}), the result follows.
\end{rem}
This result concerning the $2/3$-scaling law of enstrophy-constrained
wall-to-wall optimal transport \textemdash modulo logarithms \textemdash was
first announced in our paper \cite{tobasco2017optimal}. The present
paper provides all the mathematical details of the analysis outlined
there, as well as a much more complete discussion of our general approach
to the optimal design of heat transport. The bulk of it is
devoted to motivating and evaluating the branching flows depicted
in \prettyref{fig:branchingrough}, which are the key to proving the
logarithmically corrected lower bound from \prettyref{thm:mainbounds_w2w_ens}.

After this work was completed, a computational study of the Euler-Lagrange equations for the enstrophy-constrained wall-to-wall optimal transport problem reported convincing numerical evidence for velocity fields that produce $Nu \sim Pe^{2/3}$ in three dimensions \cite{motoki2018maximal}. Interestingly, numerical studies of the two-dimensional problem have thus far failed to produce heat transport scaling of this sort \cite{hassanzadeh2014wall,motoki2018optimal,SD}. Whereas the velocity fields produced in these two-dimensional studies feature near-wall ``recirculation zones'', which serve to enhance heat transport at moderate $Pe$, they come nowhere near the complexity of our branching flows. The three-dimensional computations, however, do exhibit branching of a fully three-dimensional character. Whether or not such three-dimensional branching flows can be constructed so as to eliminate the logarithmic gap in \prettyref{thm:mainbounds_w2w_ens} as $Pe\to \infty$ remains to be seen.

\begin{figure}
\centering

\subfloat[]{\includegraphics[width=0.33\paperwidth,height=0.25\paperheight]{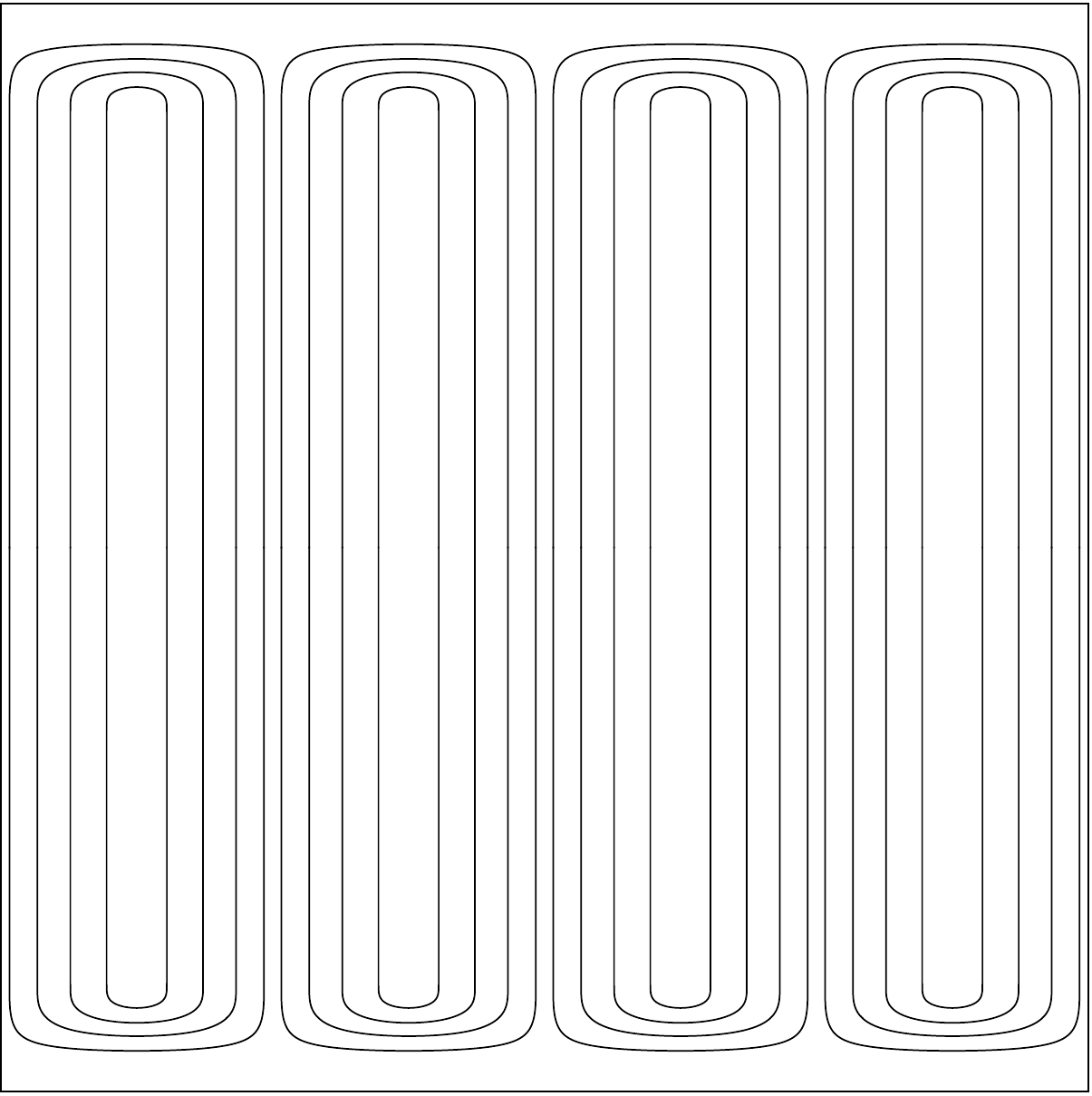}\label{fig:rollsrough}
}\hspace{4em}\subfloat[]{\includegraphics[width=0.33\paperwidth,height=0.25\paperheight]{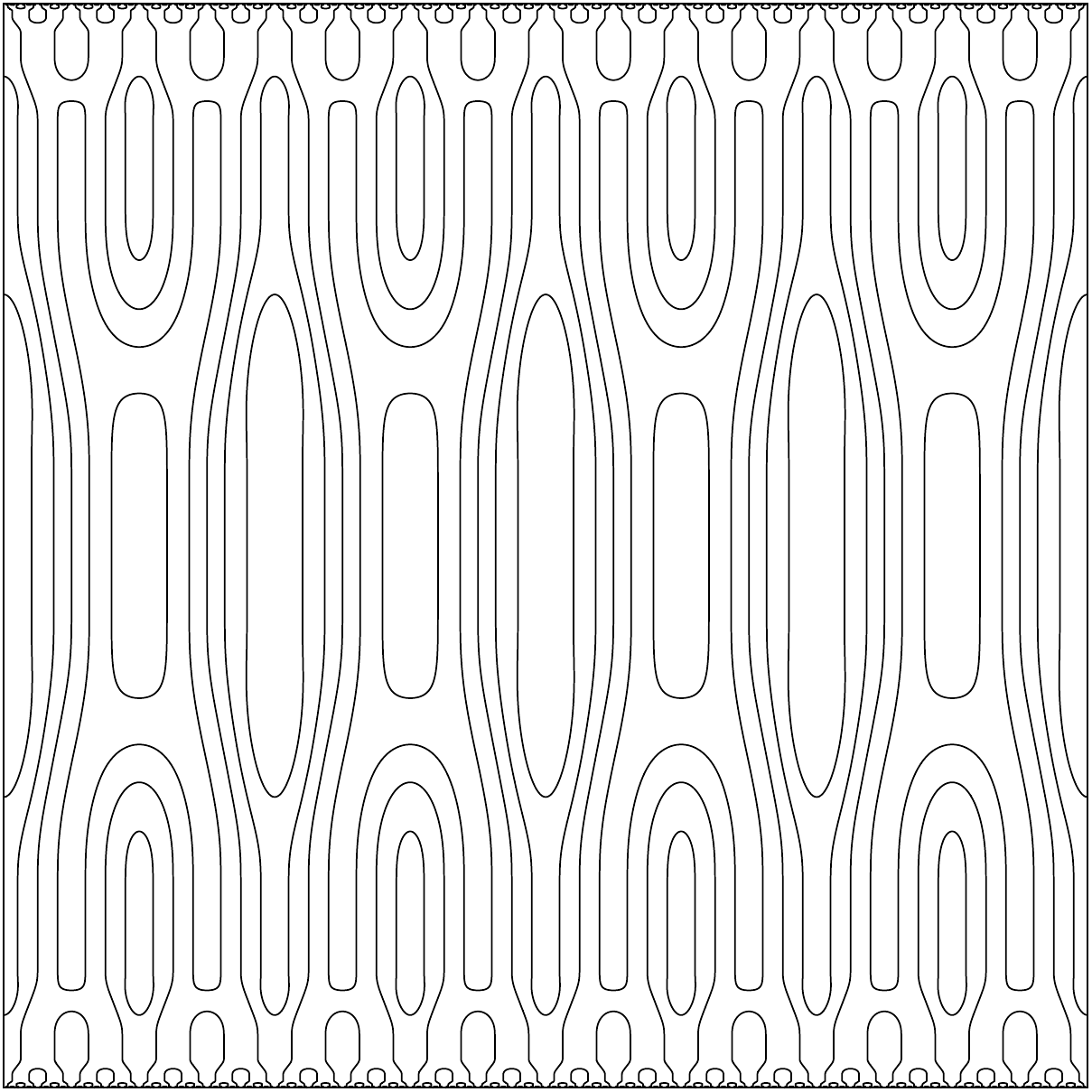}\label{fig:branchingrough}
}
%

\caption{Streamlines from two families of velocity fields considered in this
paper: (A) the convection roll construction and (B) the branching
construction. The former involves a single horizontal wavenumber while
the latter involves multiple horizontal wavenumbers, the total number
of which is allowed to diverge in the advective limit $Pe\to\infty$.
Such constructions are useful for establishing (nearly) sharp lower
bounds on wall-to-wall optimal transport. \label{fig:designs} }
\end{figure}

\subsubsection{Outline of the approach}

\prettyref{thm:mainbounds_w2w_eng} and \prettyref{thm:mainbounds_w2w_ens}
contain two types of statements: \emph{a priori }upper bounds on the
Nusselt number $Nu$ that hold for all velocity fields, and matching
lower bounds on $Nu$ for suitable designs. Methods to establish rigorous
upper bounds on convective transport go back at least to Howard in
the context of turbulent buoyancy-driven convection \cite{howard1963heat},
and Constantin and one of the authors who developed the ``background method''
to prove upper bounds on $Nu$ (albeit absent Howard's hypothesis
of statistical stationarity) \cite{constantin1995variational,doering1994variational,doering1996variational}.
Although a suitably adapted background method can be applied here
\cite{SD} we do not proceed in this way. Instead, we present a new
method for establishing upper bounds based on the fact that, for steady
velocity fields, there exists a variational principle for evaluating
heat transport. In the time-dependent case, this leads to new variational
bounds on $Nu$ that imply the background method. The bound we obtain
is as follows:
\begin{equation}
Nu(\mathbf{u})\leq\inf_{\eta}\,\left\langle |\nabla\eta|^{2}+|\nabla\Delta^{-1}\left[\partial_{t}\eta+\text{div}(\mathbf{u}\eta)\right]|^{2}\right\rangle \label{eq:unsteady_ub}
\end{equation}
where $\eta$ must satisfy
\[
\eta|_{z=0}=1\quad\text{and}\quad\eta|_{z=1}=0.
\]
Here and throughout $\Delta^{-1}$ denotes the inverse Laplacian operator
with vanishing Dirichlet boundary conditions. The bound \prettyref{eq:unsteady_ub}
is sharp for steady flows; in that case \prettyref{eq:unsteady_ub}
becomes an equality and $\eta$ need not depend on time.

In contrast, methods to establish rigorous lower bounds on $Nu$ are
far and few between. The righthand side of \prettyref{eq:unsteady_ub}
is a convex minimization. Therefore, on general grounds, there should
exist a concave maximization that is its dual. We find that
\begin{equation}
Nu(\mathbf{u})-1\geq\sup_{\xi}\,\left\langle 2w\xi-|\nabla\xi|^{2}-|\nabla\Delta^{-1}\left[\partial_{t}\xi+\text{div}(\mathbf{u}\xi)\right]|^{2}\right\rangle \label{eq:unsteady_lb}
\end{equation}
where $\xi$ must satisfy
\[
\xi|_{z=0}=0\quad\text{and}\quad\xi|_{z=1}=0.
\]
As with \prettyref{eq:unsteady_ub}, the bound \prettyref{eq:unsteady_lb}
becomes sharp for $\mathbf{u}$ that do not depend on time. 

Armed with these observations, we describe a new duality-based approach
to producing candidate designs. Consider the steady enstrophy-constrained
wall-to-wall optimal transport problem
\begin{equation}
\max_{\substack{\mathbf{u}(\mathbf{x})\\
\fint_{\Omega}|\nabla\mathbf{u}|^{2}=Pe^{2}\\
\mathbf{u}|_{\partial\Omega}=0
}
}Nu(\mathbf{u}),\label{eq:enstrophy-constrainedpblm-steady}
\end{equation}
whose optimal value bounds the unsteady maximum from below. Appealing
to the steady version of \prettyref{eq:unsteady_lb}, we find that
\prettyref{eq:enstrophy-constrainedpblm-steady} can be rewritten
as
\begin{equation}
\min_{\substack{\mathbf{u}(\mathbf{x}),\xi(\mathbf{x})\\
\fint_{\Omega}w\xi=1\\
\mathbf{u}|_{\partial\Omega}=\mathbf{0},\xi|_{\partial\Omega}=0
}
}\,\fint_{\Omega}\left|\nabla\Delta^{-1}\text{div}(\mathbf{u}\xi)\right|^{2}+\epsilon\fint_{\Omega}|\nabla\mathbf{u}|^{2}\cdot\fint_{\Omega}|\nabla\xi|^{2}\label{eq:optl_design_pblm}
\end{equation}
where $\epsilon=Pe^{-2}$. Indeed, the optimal values of \prettyref{eq:enstrophy-constrainedpblm-steady}
and \prettyref{eq:optl_design_pblm} are reciprocals and their optimizers
are in correspondence. Thus, solving the steady enstrophy-constrained
problem \prettyref{eq:enstrophy-constrainedpblm-steady} for $Pe\gg1$
is equivalent to solving \prettyref{eq:optl_design_pblm} for $\epsilon\ll1$.
We refer to \prettyref{eq:optl_design_pblm} as an ``integral'' formulation
of wall-to-wall optimal transport.

The family of variational problems \prettyref{eq:optl_design_pblm}
is non-convex and singularly perturbed. The situation shares important
similarities with other model problems from the field of ``energy-driven
pattern formation'' in materials science \cite{kohn2007energy}.
These include the study of branching patterns in micromagnetics \cite{choksi1998bounds,choksi1999domain}
and wrinkling cascades in thin elastic sheets \cite{belgacem2000rigorous,jin2001energy,ortiz1994morphology}.
For such problems, it is known that certain patterns which, at a glance,
look like \prettyref{fig:branchingrough} provide nearly optimal ways
of matching low energy states that are geometrically incompatible
but forced to coexist. We discuss such connections further in \prettyref{sec:energy-driven-pattern-formation}.

Of course, \prettyref{eq:optl_design_pblm} does not derive from materials
science but instead from fluid dynamics. We note the striking similarities
between it and Howard's variational problem, the latter of which gave
birth to the field of variational\emph{ }bounds on turbulent transport
\cite{howard1963heat}. It was recognized by Busse \cite{busse1969howards}
that Howard's problem should admit multiply-scaled optimizers. The
resulting construction is known as Busse's ``multi-$\alpha$'' technique.
After suitable modifications (wall-to-wall optimal transport and Howard's
problem are in the end quite distinct) Busse's techniques can also be used to
study \prettyref{eq:optl_design_pblm}. We consider these connections
further in \prettyref{sec:Implications-for-RBC}.

By either analogy, we are led to construct self-similar branching
flows as candidates for \prettyref{eq:optl_design_pblm}. The streamlines
depicted in \prettyref{fig:branchingrough} are symmetric about $z=1/2$;
each half of the domain is made up of $n$ convection roll systems
coupled through $n-1$ transition layers. In the bulk there are large
anisotropic convection rolls at some horizontal length-scale $l_{\text{bulk}}$.
Streamlines refine away from the bulk until there results an isotropic
convection roll system at some much smaller length-scale $l_{\text{bl}}$.
The entire construction can be modeled by a single length-scale function
$\ell(z)$ that interpolates through the layers. In terms of $\ell$,
we find the optimal branching construction to be picked out by the
solution of 
\begin{equation}
\min_{\substack{\ell(z)\\
\ell(z_{\text{bulk}})=l_{\text{bulk}}\\
\ell(z_{\text{bl}})=l_{\text{bl}}
}
}\,l_{\text{bl}}+\int_{z_{\text{bulk}}}^{z_{\text{bl}}}(\ell')^{2}\,dz+\epsilon\left(\frac{1}{l_{\text{bulk}}^{2}}+\int_{z_{\text{bulk}}}^{z_{\text{bl}}}\frac{1}{\ell^{2}}\,dz+\frac{1}{l_{\text{bl}}}\right)^{2}, \label{eq:1d_w2w}
\end{equation}
which satisfies
\begin{align*}
\ell(z) & \sim\epsilon^{1/6}\log^{1/6}\,\frac{1}{\epsilon}\,\sqrt{1-z}\quad z\in[z_{\text{bulk}},z_{\text{bl}}],\\
l_{\text{bulk}} & \sim\epsilon^{1/6}\log^{1/6}\,\frac{1}{\epsilon},\ l_{\text{bl}}\sim\epsilon^{1/3}\log^{1/3}\,\frac{1}{\epsilon}.
\end{align*}
Although this analysis does not prove that optimal designs must exhibit
fluctuations according to these rules, it does yield designs sufficient
to obtain the asserted lower bounds from \prettyref{thm:mainbounds_w2w_ens}.
The lower bounds from \prettyref{thm:mainbounds_w2w_eng} on energy-constrained
optimal transport are much simpler to obtain, and serve as a test
case for our approach.

\subsection{Outline of the paper}

\prettyref{sec:aprioribds} proves the \emph{a priori} upper bounds
from \prettyref{thm:mainbounds_w2w_eng} and \prettyref{thm:mainbounds_w2w_ens}
and establishes the variational principles and bounds on $Nu$ alluded
to above. The proof of the lower bounds from \prettyref{thm:mainbounds_w2w_eng}
and \prettyref{thm:mainbounds_w2w_ens} is spread across \prettyref{sec:OptlDesignProb_w2w},
\prettyref{sec:energy-constraineddesign}, and \prettyref{sec:enstrophy_constrained_design}.
In \prettyref{sec:OptlDesignProb_w2w} we describe our general approach
to the optimal design of heat transport. In \prettyref{sec:energy-constraineddesign}
we test our methods on the steady energy-constrained problem and obtain
a proof of the lower bound part of \prettyref{thm:mainbounds_w2w_eng}.
In \prettyref{sec:enstrophy_constrained_design} we consider the steady
enstrophy-constrained problem and prove the lower bound part of \prettyref{thm:mainbounds_w2w_ens}.
We conclude in \prettyref{sec:Implications-for-RBC} and \prettyref{sec:energy-driven-pattern-formation}
with a discussion of bounds on turbulent heat transport, and
a discussion of wall-to-wall optimal transport as a problem of energy-driven
pattern formation. 

\subsection{Notation}

Having non-dimensionalized, we employ the domain $\Omega=\mathbb{T}_{xy}^{2}\times I_{z}$
where $\mathbb{T}_{xy}^{2}$ is identified with $[0,l_{x}]\times[0,l_{y}]$
and $I_{z}=[0,1]$. The spatial average of an integrable function
$f$ on $\Omega$ is denoted by
\[
\fint_{\Omega}f=\frac{1}{|\Omega|}\int_{\Omega}f(\mathbf{x})\,d\mathbf{x}
\]
where $|\Omega|=|\mathbb{T}_{xy}^{2}|=l_{x}l_{y}$. Generally speaking,
$\fint$ indicates a well-defined average over the indicated domain
of integration. Some distinguished averages used in this paper include
\[
\overline{f}=\fint_{\mathbb{T}_{xy}^{2}}f=\frac{1}{|\mathbb{T}_{xy}^{2}|}\int_{\mathbb{T}_{xy}^{2}}f(x,y,\cdot)\,dxdy
\]
which averages over the periodic variables $x$ and $y$, the (limit
superior) space and long-time average
\[
\left\langle f\right\rangle =\limsup_{\tau\to\infty}\,\fint_{0}^{\tau}\fint_{\Omega}f=\limsup_{\tau\to\infty}\,\frac{1}{\tau}\frac{1}{|\Omega|}\int_{0}^{\tau}\int_{\Omega}f(\mathbf{x},t)\,d\mathbf{x}dt
\]
and the truncated space and time average
\[
\left\langle f\right\rangle _{\tau}=\fint_{0}^{\tau}\fint_{\Omega}f=\frac{1}{\tau}\frac{1}{|\Omega|}\int_{0}^{\tau}\int_{\Omega}f(\mathbf{x},t)\,d\mathbf{x}dt.
\]

We use the standard $L^{2}$- and $\dot{H}^{1}$-norms for functions
on $\Omega$,
\[
||f||_{L^{2}(\Omega)}=\sqrt{\int_{\Omega}|f|^{2}}\quad\text{and}\quad||f||_{\dot{H}^{1}(\Omega)}=\sqrt{\int_{\Omega}|\nabla f|^{2}}.
\]
The set of smooth and compactly supported functions on $\Omega$ is
$C_{c}^{\infty}(\Omega)$. The Sobolev space $H_{0}^{1}(\Omega)$
is its completion in the norm $||\cdot||_{\dot{H}^{1}(\Omega)}$.
We use $\left(\cdot,\cdot\right)$ to denote the duality pairing of
$H^{-1}$ with $H_{0}^{1}$. We denote by $\Delta^{-1}$ the inverse
Laplacian operator with vanishing Dirichlet boundary conditions, which
is well-defined from $H^{-1}(\Omega)\to H_{0}^{1}(\Omega)$. 

The notation $X\lesssim Y$ means that there exists a positive constant
$C$ not depending on any parameters such that $X\leq CY$. We use
the notations $X\wedge Y=\min\{X,Y\}$ and $X\vee Y=\max\{X,Y\}$.

\subsection{Acknowledgements}

We thank R.~V.\ Kohn and A.~N.\ Souza for helpful discussions.
This work was supported by NSF Awards DGE-0813964 and DMS-1812831 (IT), DMS-1515161 and DMS-1813003
(CRD), a Van Loo Postdoctoral Fellowship (IT) and a Guggenheim Foundation
Fellowship (CRD).

\section{\emph{A priori} bounds on wall-to-wall optimal transport\label{sec:aprioribds}}

We begin our analysis of wall-to-wall optimal transport by proving
the \emph{a priori} upper bounds from \prettyref{thm:mainbounds_w2w_eng}
and \prettyref{thm:mainbounds_w2w_ens}. Unless otherwise explicitly
stated, we consider throughout that $\left\langle |\mathbf{u}|^{2}\right\rangle <\infty$
so that \prettyref{eq:advection-diffusion} is well-posed.

The upper bound from \prettyref{thm:mainbounds_w2w_eng} on energy-constrained
transport is straightforward to prove, and we dispatch with it first.
\begin{prop}
We have that
\[
Nu(\mathbf{u})\leq1+\frac{1}{2}\left\langle |w|^{2}\right\rangle ^{1/2}
\]
whenever $w|_{\partial\Omega}=0$.
\end{prop}
\begin{proof}
Let us recall the argument from \cite{hassanzadeh2014wall}. First,
note that $Nu$ does not depend on the initial temperature $T_{0}$.
Thus, we can take $T_{0}=1-z$ and conclude by the maximum principle
that the associated solution of \prettyref{eq:advection-diffusion}
satisfies
\[
0\leq T\leq1\quad\text{a.e.}
\]
Note also that because $w$ vanishes at $\partial\Omega$,
\[
\left\langle w\right\rangle =0.
\]
Combining this with Jensen's inequality and the definition of the
Nusselt number \prettyref{eq:Nusseltdefn}, we have that
\begin{align*}
Nu-1 & =\left\langle wT\right\rangle =\left\langle w(T-\frac{1}{2})\right\rangle \\
 & \leq\left\langle |w|\right\rangle ||T-\frac{1}{2}||_{L_{tx}^{\infty}}\leq\frac{1}{2}\left\langle |w|^{2}\right\rangle ^{1/2}.
\end{align*}
\end{proof}
The remainder of this section is on upper bounds for enstrophy-constrained
transport. We prove the following bound:
\begin{prop}
\label{prop:ensbound_w2w} There exists a positive constant $C$ such
that
\[
Nu(\mathbf{u})\leq1+C\left\langle |\nabla w|^{2}\right\rangle \wedge\left\langle |\nabla w|^{2}\right\rangle ^{1/3}.
\]
whenever $w|_{\partial\Omega}=0$. This constant is independent of
all parameters.
\end{prop}
To the authors' knowledge, there are at least three proofs of \prettyref{prop:ensbound_w2w}.
For one, it can be obtained via an application of the background method
\cite{SD}. It can also be seen as a consequence of Seis' arguments
from \cite{seis2015}. Our proof of \prettyref{prop:ensbound_w2w}
is different from either of these: we obtain it via a new approach
using a Dirichlet-type variational principle for the functional $Nu(\mathbf{u})$. 

It should be mentioned that we are not the first to notice the variational
structure of the advection-diffusion equation. The existence of a
variational principle for advection-diffusion in bounded domains appears
to have been first reported in \cite{ortiz1985variational}, where
it was used to systematically derive ``best approximation'' finite
element schemes. Around the same time, as described in \cite{milton1990characterizing},
variational principles for computing effective complex conductivities
in periodic homogenization were discovered by Gibiansky and Cherkaev
(the relevant corrector equation is again divergence-form but not
self-adjoint). We learned about the existence of such principles from
the papers \cite{avellaneda1991integral,fannjiang1994convection},
whose formulas for computing effective diffusivities in periodic homogenization
inspired the formulas for $Nu$ obtained below. Let us also mention
the related work \cite{ghoussoub2009selfdual} which discusses non-standard
variational principles for PDEs at large. It was a pleasant surprise
to learn that the seemingly \emph{ad hoc} change of variables introduced
in \cite{hassanzadeh2014wall} for handling the Euler-Lagrange equations
of wall-to-wall optimal transport turn out to be similar to those
employed in previous works, and that behind it all is a variational
principle for $Nu$. 

The remainder of this section is organized as follows. First we establish
a variational principle for $Nu$ in the steady case where the reasoning
is most transparent. We then extend the arguments to general unsteady
flows, where the variational principle turns into a variational bound
as anticipated in \prettyref{eq:unsteady_ub}. \prettyref{prop:ensbound_w2w}
follows immediately thereafter. Later on in \prettyref{sec:OptlDesignProb_w2w},
we obtain the dual formula to bound $Nu$ from below. In order to
highlight the key step in the proof \textemdash a certain symmetrizing
change of variables for the advection-diffusion equation \textemdash we
refer to this as the ``symmetrization method''.

\subsection{The symmetrization method for steady velocity fields}

We start with the case where $\mathbf{u}$ is an arbitrary divergence-free
vector field belonging to $L^{2}(\Omega;\mathbb{R}^{3})$. In this
case,
\[
Nu(\mathbf{u})=1+\fint_{\Omega}w\theta=1+\fint_{\Omega}|\nabla\theta|^{2}
\]
where $\theta=T-(1-z)$ is the deviation of the temperature field
from the conductive state. That is, $\theta$ is the unique (essentially
bounded) weak solution of 
\[
\mathbf{u}\cdot\nabla\theta=\Delta\theta+w
\]
with zero Dirichlet boundary data $\theta|_{\partial\Omega}=0$. To
change variables, we let $\theta_{\pm}$ be the unique weak solutions
of the pair of formally adjoint PDEs
\begin{equation}
\pm\mathbf{u}\cdot\nabla\theta_{\pm}=\Delta\theta_{\pm}+w\label{eq:adjointsteadyPDEs}
\end{equation}
with $\theta_{\pm}|_{\partial\Omega}=0$, and observe that $\theta=\theta_{+}$.
Then, we define $\eta,\xi\in H_{0}^{1}(\Omega)$ by
\[
\eta=\frac{1}{2}(\theta_{+}-\theta_{-})\quad\text{and}\quad\xi=\frac{1}{2}(\theta_{+}+\theta_{-})
\]
and observe they satisfy the equivalent system of PDEs
\begin{equation}
\begin{cases}
\mathbf{u}\cdot\nabla\eta=\Delta\xi+w\\
\mathbf{u}\cdot\nabla\xi=\Delta\eta
\end{cases}.\label{eq:steadyPDEsystem}
\end{equation}
We claim the change of variables $(\theta_{+},\theta_{-})\leftrightarrow(\eta,\xi)$
yields a variational formula for $Nu$.

Testing the second equation in \prettyref{eq:steadyPDEsystem} against
$\xi$ and integrating by parts shows that $\nabla\xi\perp\nabla\eta$
in $L^{2}(\Omega)$, since
\[
\int_{\Omega}\nabla\eta\cdot\nabla\xi=-\int_{\Omega}\mathbf{u}\cdot\nabla\xi\xi=0.
\]
Therefore, 
\[
Nu-1=\fint_{\Omega}|\nabla\theta_{+}|^{2}=\fint_{\Omega}|\nabla\eta|^{2}+|\nabla\xi|^{2}
\]
or, using the first PDE in \prettyref{eq:steadyPDEsystem},
\begin{equation}
Nu-1=\fint_{\Omega}|\nabla\eta|^{2}+|\nabla\Delta^{-1}\left[\mathbf{u}\cdot\nabla\eta-w\right]|^{2}.\label{eq:Nu-non-local-formula}
\end{equation}
Consider the righthand side of \prettyref{eq:Nu-non-local-formula}
as it depends on $\eta$. Since $\mathbf{u}\in L^{2}$ and is divergence-free,
the righthand side is well-defined for $\eta\in H_{0}^{1}(\Omega)\cap L^{\infty}(\Omega)$.
Now consider the variational problem
\begin{equation}
\inf_{\eta\in H_{0}^{1}(\Omega)\cap L^{\infty}(\Omega)}\,\fint_{\Omega}|\nabla\eta|^{2}+|\nabla\Delta^{-1}\left[\mathbf{u}\cdot\nabla\eta-w\right]|^{2},\label{eq:steady_min-pblm_for_Nu}
\end{equation}
which is strictly convex so that any minimizer must be unique. By
a first variation argument, we see that $\eta$ is a minimizer of
\prettyref{eq:steady_min-pblm_for_Nu} if and only if it satisfies
the Euler-Lagrange equation
\[
\Delta\eta=\mathbf{u}\cdot\nabla\Delta^{-1}\left[\mathbf{u}\cdot\nabla\eta-w\right].
\]
This is a rewrite of the system \prettyref{eq:steadyPDEsystem} with
$\xi$ defined by 
\[
\xi=\Delta^{-1}\left[\mathbf{u}\cdot\nabla\eta-w\right].
\]
Since that system possesses solutions in the class $H_{0}^{1}\cap L^{\infty}$,
it immediately follows that \prettyref{eq:steady_min-pblm_for_Nu}
has a minimizer. It also follows that the minimal value in \prettyref{eq:steady_min-pblm_for_Nu}
is $Nu-1$. 

Relabeling $\eta$ as $\eta+1-z$ and using that $\mathbf{u}$ is
divergence-free yields the following variational principle for heat
transport:
\begin{thm}
\label{thm:steady-primal-varprin} Let $\mathbf{u}(\mathbf{x})$ be
a divergence-free vector field in $L^{2}(\Omega;\mathbb{R}^{3})$.
Then,
\[
Nu(\mathbf{u})=\min_{\substack{\eta\in H^{1}(\Omega)\cap L^{\infty}(\Omega)\\
\eta|_{z=0}=1,\,\eta|_{z=1}=0
}
}\,\fint_{\Omega}\left|\nabla\eta\right|^{2}+\left|\nabla\Delta^{-1}\emph{div}(\mathbf{u}\eta)\right|^{2}.
\]
\end{thm}

\subsection{Upper bounds on unsteady transport by symmetrization}

With some care, the previous argument can be adapted to the time-dependent
case. Here, the useful change of variables arises from the pair of
PDEs
\[
\pm(\partial_{t}+\mathbf{u}\cdot\nabla)\theta_{\pm}=\Delta\theta_{\pm}+w
\]
which are formally adjoint in space and time. These are in obvious
analogy with \prettyref{eq:adjointsteadyPDEs} from the steady case.
However, since parabolic PDEs are generically only well-posed forward
in time, making sense of the ``$-$'' equation presents an added
difficulty. To deal with this, we will reverse the sense of time between
the equations, performing the change of variables $t\to\tau-t$ for
appropriately chosen $\tau\gg1$. In the limit $\tau\to\infty$, we
recover the unsteady variational bound.

Define the admissible set of test functions
\begin{multline*}
\mathcal{A}=\{\eta:\eta\in L^{\infty}([0,\infty);L^{2}(\Omega))\} \\ \cap\{\eta(t)\in H^{1}(\Omega)\cap L^{\infty}(\Omega),\partial_{t}\eta(t)\in H^{-1}(\Omega)\ \text{a.e.}\ t\}.
\end{multline*}

\begin{thm}
\label{thm:primal_Nubound_w2w} Let $\mathbf{u}(\mathbf{x},t)$ be
a divergence-free vector field with bounded mean energy $\left\langle |\mathbf{u}|^{2}\right\rangle <\infty$.
Then,
\[
Nu(\mathbf{u})\leq\inf_{\substack{\eta\in\mathcal{A}\\
\eta|_{z=1}=0,\,\eta|_{z=0}=1
}
}\,\left\langle \left|\nabla\eta\right|^{2}+\left|\nabla\Delta^{-1}\left[\partial_{t}\eta+\emph{div}(\mathbf{u}\eta)\right]\right|^{2}\right\rangle .
\]
\end{thm}
\begin{proof}
We begin by introducing the (approximately) symmetrized variables.
Let $\theta_{+}=T-(1-z)$ and note it solves
\[
\partial_{t}\theta_{+}+\mathbf{u}\cdot\nabla\theta_{+}=\Delta\theta_{+}+w
\]
on $[0,\infty)\times\Omega$ and vanishes at $\partial\Omega$. Let
$\theta_{-}^{\tau}$ be the unique essentially bounded weak solution
of
\[
-\partial_{t}\theta_{-}-\mathbf{u}\cdot\nabla\theta_{-}=\Delta\theta_{-}+w
\]
on $\Omega_{\tau}=[0,\tau]\times\Omega$ that vanishes at $\partial\Omega$
and has final time data $\theta_{-}^{\tau}(\tau)=0$. 
Next, define the symmetrized variables $\eta^{\tau}$
and $\xi^{\tau}$ by
\[
\eta^{\tau}=\frac{1}{2}\left(\theta_{+}-\theta_{-}^{\tau}\right)\quad\text{and}\quad\xi^{\tau}=\frac{1}{2}(\theta_{+}+\theta_{-}^{\tau}).
\]
Observe $\eta^{\tau}$ and $\xi^{\tau}$ solve
\begin{equation}
\begin{cases}
\partial_{t}\eta^{\tau}+\mathbf{u}\cdot\nabla\eta^{\tau}=\Delta\xi^{\tau}+w\\
\partial_{t}\xi^{\tau}+\mathbf{u}\cdot\nabla\xi^{\tau}=\Delta\eta^{\tau}
\end{cases}\label{eq:symmetrized_timedept}
\end{equation}
on $\Omega_{\tau}$, vanish at $\partial\Omega$, and remain bounded
in $L_{tx}^{\infty}$ uniformly in time. In particular, by the maximum
principle,
\[
||\eta^{\tau}||_{L_{tx}^{\infty}}\vee||\xi^{\tau}||_{L_{tx}^{\infty}}\lesssim||\theta_{+}||_{L_{tx}^{\infty}}\vee||\theta_{-}^{\tau}||_{L_{tx}^{\infty}}\leq C(T_{0}).
\]

We proceed as in the steady case, accumulating errors that vanish
as $\tau\to\infty$. From the second PDE in \prettyref{eq:symmetrized_timedept}
we find that
\begin{align*}
\left\langle \nabla\eta^{\tau}\cdot\nabla\xi^{\tau}\right\rangle _{\tau} & =\fint_{0}^{\tau}\frac{1}{|\Omega|}\left(-\Delta\eta^{\tau},\xi^{\tau}\right)\,dt=\fint_{0}^{\tau}\frac{1}{|\Omega|}\left(-\partial_{t}\xi^{\tau}-\text{div}(\mathbf{u}\xi^{\tau}),\xi^{\tau}\right)\,dt\\
 & =-\fint_{0}^{\tau}\frac{1}{|\Omega|}\frac{d}{dt}\frac{1}{2}||\xi^{\tau}||_{L^{2}(\Omega)}^{2}+\fint_{0}^{\tau}\fint_{\Omega}\mathbf{u}\xi^{\tau}\cdot\nabla\xi^{\tau} \\
 & =\frac{1}{2\tau}\frac{1}{|\Omega|}(||\xi^{\tau}(0)||_{L^{2}(\Omega)}^{2}-||\xi^{\tau}(\tau)||_{L^{2}(\Omega)}^{2}).
\end{align*}
Therefore,
\[
\left|\left\langle \nabla\eta^{\tau}\cdot\nabla\xi^{\tau}\right\rangle _{\tau}\right|\leq C(T_{0})\frac{1}{\tau}.
\]
Since 
\[
Nu-1=\left\langle |\nabla\theta|^{2}\right\rangle =\left\langle |\nabla\theta_{+}|^{2}\right\rangle 
\]
and
\[
\left\langle |\nabla\theta_{+}|^{2}\right\rangle _{\tau}=\left\langle |\nabla\eta^{\tau}|^{2}+\left|\nabla\xi^{\tau}\right|^{2}\right\rangle _{\tau}+2\left\langle \nabla\eta^{\tau}\cdot\nabla\xi^{\tau}\right\rangle _{\tau},
\]
we conclude that
\begin{equation}
Nu-1=\left\langle |\nabla\eta^{\tau}|^{2}+\left|\nabla\Delta^{-1}\left[\partial_{t}\eta^{\tau}+\mathbf{u}\cdot\nabla\eta^{\tau}-w\right]\right|^{2}\right\rangle _{\tau}+O(\frac{1}{\tau}).\label{eq:Nu_unsteady_uptoerror}
\end{equation}

Next, we prove that $\eta^{\tau}$ is approximately minimal, with
an error that vanishes as $\tau\to\infty$. Let $\eta\in\mathcal{A}$
vanish at $\partial\Omega$ but be otherwise arbitrary, and consider
the difference
\[
A_{\tau}=\left\langle |\nabla\eta|^{2}+\left|\nabla\Delta^{-1}\left[\partial_{t}\eta+\mathbf{u}\cdot\nabla\eta-w\right]\right|^{2}\right\rangle _{\tau} -\left\langle |\nabla\eta^{\tau}|^{2}+\left|\nabla\Delta^{-1}\left[\partial_{t}\eta^{\tau}+\mathbf{u}\cdot\nabla\eta^{\tau}-w\right]\right|^{2}\right\rangle _{\tau}.
\]
Using the convexity of $|\cdot|^{2}$, we can expand around $\eta^{\tau}$
and use \prettyref{eq:symmetrized_timedept} to arrive at the lower
bound
\begin{align*}
\frac{A_{\tau}}{2} & \geq\left\langle \nabla(\eta-\eta^{\tau})\cdot\nabla\eta^{\tau}+\nabla\Delta^{-1}(\partial_{t}+\mathbf{u}\cdot\nabla)(\eta-\eta^{\tau})\cdot\nabla\Delta^{-1}\left[(\partial_{t}+\mathbf{u}\cdot\nabla)\eta^{\tau}-w\right]\right\rangle _{\tau}\\
 & =-\frac{1}{|\Omega|}\fint_{0}^{\tau}\left((\partial_{t}+\mathbf{u}\cdot\nabla)\xi^{\tau},\eta-\eta^{\tau}\right)+\left((\partial_{t}+\mathbf{u}\cdot\nabla)(\eta-\eta^{\tau}),\xi^{\tau}\right)\,dt \\ 
 & =-\frac{1}{|\Omega|}\fint_{0}^{\tau}\frac{d}{dt}\left[\int_{\Omega}(\eta-\eta^{\tau})\xi^{\tau}\right]\,dt\\
 & \gtrsim-\frac{1}{\tau}\frac{1}{|\Omega|}(||\eta||_{L_{t}^{\infty}L_{\mathbf{x}}^{2}}\vee||\eta^{\tau}||_{L_{t}^{\infty}L_{\mathbf{x}}^{2}})\cdot||\xi^{\tau}||_{L_{t}^{\infty}L_{\mathbf{x}}^{2}}\geq-C(T_{0},\eta,\Omega)\frac{1}{\tau}.
\end{align*}
Combining this with \prettyref{eq:Nu_unsteady_uptoerror}, we find
that
\[
Nu-1\leq\left\langle |\nabla\eta|^{2}+\left|\nabla\Delta^{-1}\left[\partial_{t}\eta+\mathbf{u}\cdot\nabla\eta-w\right]\right|^{2}\right\rangle _{\tau}+C(T_{0},\eta,\Omega)\frac{1}{\tau}.
\]
Taking $\tau\to\infty$ yields the inequality
\[
Nu-1\leq\left\langle |\nabla\eta|^{2}+\left|\nabla\Delta^{-1}\left[\partial_{t}\eta+\mathbf{u}\cdot\nabla\eta-w\right]\right|^{2}\right\rangle .
\]
This holds for all $\eta\in\mathcal{A}$ that vanish at $\partial\Omega$.
Changing variables by $\eta\to\eta+1-z$ and optimizing yields the
result.
\end{proof}
Even if $\mathbf{u}$ depends on time, $\eta$ can be taken to be
independent of time and still used to bound $Nu$. The simplified
version of \prettyref{thm:primal_Nubound_w2w} that results is analogous
to \prettyref{thm:steady-primal-varprin}, but for unsteady heat transport.
\begin{cor}
\label{cor:symmsteadybound_w2w} Let $\mathbf{u}(\mathbf{x},t)$ be
a divergence-free vector field with bounded mean energy $\left\langle |\mathbf{u}|^{2}\right\rangle <\infty$.
Then,
\begin{equation}
Nu(\mathbf{u})\leq\inf_{\substack{\eta\in H^{1}(\Omega)\cap L^{\infty}(\Omega)\\
\eta|_{z=1}=0,\,\eta|_{z=0}=1
}
}\,\fint_{\Omega}\left|\nabla\eta\right|^{2}+\left\langle \left|\nabla\Delta^{-1}\emph{div}(\mathbf{u}\eta)\right|^{2}\right\rangle .\label{eq:symmsteadybound_w2w}
\end{equation}
\end{cor}

\subsection{Proof of \prettyref{prop:ensbound_w2w}}

We prove \prettyref{prop:ensbound_w2w} by deducing it from \prettyref{cor:symmsteadybound_w2w}
with a good choice of test function $\eta$. Evidently, we must find
a convenient way to bound the non-local term appearing there. Since
$\nabla\Delta^{-1}\text{div}$ is an $L^{2}$-orthogonal projection,
one has the bound
\[
\left\langle \left|\nabla\Delta^{-1}\text{div}(\mathbf{u}\eta)\right|^{2}\right\rangle \leq\left\langle \left|\mathbf{u}(\eta-c)\right|^{2}\right\rangle 
\]
for an arbitrary constant $c$. In the case that $\mathbf{u}$ satisfies
no-slip boundary conditions, one can deduce \prettyref{prop:ensbound_w2w}
by choosing $\eta\approx c$ thereby localizing the righthand side
to a small neighborhood of $\partial\Omega$. Then, a straightforward
application of Poincar\'e's inequality yields the result.

The final step in the preceding argument requires all components of
$\mathbf{u}$ to vanish at $\partial\Omega$. This is not useful for
dealing with no-penetration boundary conditions. Nevertheless, \prettyref{prop:ensbound_w2w}
holds in this more general case. The key is to approach the non-local
term from \prettyref{cor:symmsteadybound_w2w} by duality. Observe
that
\[
\int_{\Omega}|\nabla\Delta^{-1}\text{div}\,\mathbf{m}|^{2}=\sup_{\theta\in H_{0}^{1}(\Omega)}\,\int_{\Omega}2\nabla\theta\cdot\mathbf{m}-|\nabla\theta|^{2}
\]
whenever $\mathbf{m}\in L^{2}(\Omega)$. Thus, the inequality
\[
\int_{\Omega}|\nabla\Delta^{-1}\text{div}\,\mathbf{m}|^{2}\leq C
\]
and the statement that
\[
\int_{\Omega}2\nabla\theta\cdot\mathbf{m}\leq C+\int_{\Omega}|\nabla\mathbf{\theta}|^{2}\quad\forall\,\theta\in H_{0}^{1}
\]
are one and the same. Taking $\mathbf{m}=\mathbf{u}\eta$ where $\eta$
depends only on $z$, we conclude it will be useful to have bounds
of the form
\[
\int_{\Omega}2\overline{w\theta}\eta'(z)\leq C+\int_{\Omega}|\nabla\mathbf{\theta}|^{2}\quad\forall\,\theta\in H_{0}^{1}.
\]
The following preliminary result allows us to establish bounds of
this type.
\begin{lem}
\label{lem:lingrowth_wtheta}Let $w,\theta\in H_{0}^{1}(\Omega)$.
Then,
\[
\left|\overline{w\theta}(z)\right|\lesssim\frac{|z\wedge(1-z)|}{|\mathbb{T}_{xy}^{2}|}||\partial_{z}\theta||_{L^{2}(\Omega)}||\partial_{z}w||_{L^{2}(\Omega)}\quad\text{a.e.}
\]
 
\end{lem}
\begin{rem}
The reader familiar with the background method may recognize that
this inequality also plays a key role in carrying out that approach
to \emph{a priori} bounds. In particular, it is useful for verifying
the spectral constraint. See \prettyref{sec:onbackgroundmethod} for
more on the connection between the symmetrization method and the background
method.
\end{rem}
\begin{proof}
By the usual approximation arguments, we can take $\mathbf{u}$ and
$\theta$ to be smooth. Differentiating and applying the Cauchy-Schwarz
inequality, we find that
\[
\left|\frac{d}{dz}\overline{w^{2}}\right|\leq\frac{2}{|\mathbb{T}_{xy}^{2}|}||\partial_{z}w||_{L^{2}(\mathbb{T}_{xy}^{2})}||w||_{L^{2}(\mathbb{T}_{xy}^{2})}.
\]
Integrating this from $0$ to $z$ yields 
\begin{align*}
||w||_{L^{\infty}([0,z];L_{xy}^{2})}^{2} & \leq\frac{2}{|\mathbb{T}_{xy}^{2}|}\int_{0}^{z}||\partial_{z}w(z')||_{L_{xy}^{2}}||w(z')||_{L_{xy}^{2}}\,dz' \\
& \leq\frac{2}{|\mathbb{T}_{xy}^{2}|}||\partial_{z}w||_{L^{2}(\Omega)}||w||_{L^{2}([0,z];L_{xy}^{2})}\\
 & \leq\frac{2}{|\mathbb{T}_{xy}^{2}|}||\partial_{z}w||_{L^{2}(\Omega)}||w||_{L^{\infty}([0,z];L_{xy}^{2})}|z|^{1/2}.
\end{align*}
Therefore,
\begin{equation}
||w||_{L^{\infty}([0,z];L_{xy}^{2})}\leq\frac{2}{|\mathbb{T}_{xy}^{2}|}|z|^{1/2}||\partial_{z}w||_{L^{2}(\Omega)}.\label{eq:rootgrowth_w}
\end{equation}
Similarly, we find that
\begin{equation}
||\theta||_{L^{\infty}([0,z];L_{xy}^{2})}\leq\frac{2}{|\mathbb{T}_{xy}^{2}|}|z|^{1/2}||\partial_{z}\theta||_{L^{2}(\Omega)}.\label{eq:rootgrowth_theta}
\end{equation}

Now consider the product $w\theta$. We have that
\[
|\frac{d}{dz}\overline{w\theta}|\leq|\overline{\partial_{z}w\theta}|+|\overline{w\partial_{z}\theta}|\leq\frac{1}{|\mathbb{T}_{xy}^{2}|}\left(||\partial_{z}w||_{L_{xy}^{2}}||\theta||_{L_{xy}^{2}}+||w||_{L_{xy}^{2}}||\partial_{z}\theta||_{L_{xy}^{2}}\right)
\]
for all $z$. Therefore,
\begin{align*}
||\overline{w\theta}||_{L^{\infty}([0,z])} & \leq\frac{1}{|\mathbb{T}_{xy}^{2}|}\int_{0}^{z}||\partial_{z}w(z')||_{L_{xy}^{2}}||\theta(z')||_{L_{xy}^{2}}+||w(z')||_{L_{xy}^{2}}||\partial_{z}\theta(z')||_{L_{xy}^{2}}\,dz'\\
 & \leq\frac{1}{|\mathbb{T}_{xy}^{2}|}\left(||\partial_{z}w||_{L^{2}(\Omega)}||\theta||_{L^{2}([0,z];L_{xy}^{2})}+||\partial_{z}\theta||_{L^{2}(\Omega)}||w||_{L^{2}([0,z];L_{xy}^{2})}\right)\\
 & \leq\frac{|z|^{1/2}}{|\mathbb{T}_{xy}^{2}|}\left(||\partial_{z}w||_{L^{2}(\Omega)}||\theta||_{L^{\infty}([0,z];L_{xy}^{2})}+||\partial_{z}\theta||_{L^{2}(\Omega)}||w||_{L^{\infty}([0,z];L_{xy}^{2})}\right).
\end{align*}
Applying \prettyref{eq:rootgrowth_w} and \prettyref{eq:rootgrowth_theta}
we find that 
\[
||\overline{w\theta}||_{L^{\infty}([0,z])}\leq\frac{4|z|}{|\mathbb{T}_{xy}^{2}|}||\partial_{z}w||_{L^{2}(\Omega)}||\partial_{z}\theta||_{L^{2}(\Omega)}.
\]
The argument above is symmetric under $z\to1-z$, so we immediately
obtain the inequality
\[
||\overline{w\theta}||_{L^{\infty}([1-z,1])}\leq\frac{4|1-z|}{|\mathbb{T}_{xy}^{2}|}||\partial_{z}w||_{L^{2}(\Omega)}||\partial_{z}\theta||_{L^{2}(\Omega)}.
\]
These two combine to prove the result.
\end{proof}
Now we are ready to prove \prettyref{prop:ensbound_w2w}, which immediately
implies the upper bound part of \prettyref{thm:mainbounds_w2w_ens}.
We follow the plan laid out above.

\begin{proof}[Proof of \prettyref{prop:ensbound_w2w}]

We apply \prettyref{cor:symmsteadybound_w2w} with an appropriate
class of test functions $\{\eta_{\delta}\}$. Given $\delta\in(0,1/2]$,
we define $\eta_{\delta}$ by
\begin{equation}
\eta_{\delta}(z)=\begin{cases}
1-\frac{1}{2\delta}z & 0\leq z\leq\delta\\
\frac{1}{2} & \delta\leq z\leq1-\delta\\
\frac{1}{2\delta}(1-z) & 1-\delta\leq z\leq1
\end{cases}.\label{eq:testfunctions}
\end{equation}
Note these are admissible in \prettyref{eq:symmsteadybound_w2w}.
Thus, 
\begin{equation}
Nu(\mathbf{u})\leq\inf_{\delta\in(0,\frac{1}{2}]}\left\{ \fint_{\Omega}\left|\nabla\eta_{\delta}\right|^{2}+\left\langle \left|\nabla\Delta^{-1}\text{div}(\mathbf{u}\eta_{\delta})\right|^{2}\right\rangle \right\} .\label{eq:intermediate_Nu_fromabove_0}
\end{equation}
The first integral appearing above is simple to estimate and it satisfies
\[
\fint_{\Omega}\left|\nabla\eta_{\delta}\right|^{2}\sim\frac{1}{\delta}.
\]
So,
\begin{equation}
Nu\lesssim\inf_{\delta\in(0,\frac{1}{2}]}\left\{ \frac{1}{\delta}+\left\langle \left|\nabla\Delta^{-1}\text{div}(\mathbf{u}\eta_{\delta})\right|^{2}\right\rangle \right\} .\label{eq:intermediate_Nu_fromabove}
\end{equation}

Now we must estimate the non-local term. We claim that
\begin{equation}
\left\langle \left|\nabla\Delta^{-1}\text{div}(\mathbf{u}\eta_{\delta})\right|^{2}\right\rangle \lesssim\delta^{2}\left\langle |\nabla w|^{2}\right\rangle \label{eq:claim_for_nopenetration}
\end{equation}
for all $\delta\in(0,\frac{1}{2}]$. We argue at a.e.\ time $t$.
By duality, 
\[
\int_{\Omega}|\nabla\Delta^{-1}\text{div}(\mathbf{u}\eta_{\delta})|^{2}=\sup_{\theta\in H_{0}^{1}(\Omega)}\,2\left(\text{div}(\mathbf{u}\eta_{\delta}),\theta\right)-\int_{\Omega}|\nabla\theta|^{2}.
\]
Since  $\eta_{\delta}$ only depends on $z$, $\text{div}(\mathbf{u}\eta_{\delta})=w\eta_{\delta}'$
and we have by Fubini that
\[
\fint_{\Omega}|\nabla\Delta^{-1}\text{div}(\mathbf{u}\eta_{\delta})|^{2}=\sup_{\theta\in H_{0}^{1}(\Omega)}\,\fint_{I_{z}}2\overline{\theta w}\eta_{\delta}'-\fint_{\Omega}|\nabla\theta|^{2}.
\]
Therefore to show \prettyref{eq:claim_for_nopenetration} it suffices
to prove the inequality
\[
\fint_{I_{z}}2\overline{\theta w}\eta_{\delta}'\leq C\delta^{2}\fint_{\Omega}|\nabla w|^{2}+\fint_{\Omega}|\nabla\theta|^{2}\quad\forall\,\mathbf{u}\in H_{0}^{1}(\Omega;\mathbb{R}^{d}),\theta\in H_{0}^{1}(\Omega)
\]
Recalling the formula for $\eta_{\delta}$ from \prettyref{eq:testfunctions},
we see we must prove that
\[
\fint_{0}^{\delta}|\overline{\theta w}(z)|\,dz\lesssim\delta\left(\fint_{\Omega}|\nabla w|^{2}\right)^{1/2}\left(\fint_{\Omega}|\nabla\theta|^{2}\right)^{1/2}\quad\forall\,\mathbf{u}\in H_{0}^{1}(\Omega;\mathbb{R}^{d}),\theta\in H_{0}^{1}(\Omega).
\]
Applying \prettyref{lem:lingrowth_wtheta} proves this result and
hence the desired estimate \prettyref{eq:claim_for_nopenetration}. 

Assembling \prettyref{eq:intermediate_Nu_fromabove} and \prettyref{eq:claim_for_nopenetration},
we conclude that 
\[
Nu\lesssim\inf_{\delta\in(0,\frac{1}{2}]}\left\{ \frac{1}{\delta}+\delta^{2}\left\langle |\nabla w|^{2}\right\rangle \right\} .
\]
If $\left\langle |\nabla w|^{2}\right\rangle \geq1$ we may choose
$\delta\sim\left\langle |\nabla w|^{2}\right\rangle ^{-1/3}$ to conclude
that

\begin{equation}
Nu\lesssim\left\langle |\nabla w|^{2}\right\rangle ^{1/3}.\label{eq:Nu-bd-1}
\end{equation}
To handle the case $\left\langle |\nabla w|^{2}\right\rangle \leq1$
we treat the choice $\delta=\frac{1}{2}$ more carefully in the above.
Since $\eta_{\frac{1}{2}}=1-z$, \prettyref{eq:intermediate_Nu_fromabove_0}
and \prettyref{eq:claim_for_nopenetration} combine to prove that
\begin{equation}
Nu\leq1+\left\langle \left|\nabla\Delta^{-1}\text{div}(\mathbf{u}\eta_{\frac{1}{2}})\right|^{2}\right\rangle \leq1+C\left\langle |\nabla w|^{2}\right\rangle .\label{eq:Nu-bd-2}
\end{equation}
From \prettyref{eq:Nu-bd-1} and \prettyref{eq:Nu-bd-2} we conclude
the result.

\end{proof}

\section{Optimal design of steady wall-to-wall transport \label{sec:OptlDesignProb_w2w}}

This section begins the proof of the lower bounds from \prettyref{thm:mainbounds_w2w_eng}
and \prettyref{thm:mainbounds_w2w_ens}. While every admissible velocity
field yields a lower bound on the maximal rate of heat transport,
it is not at all clear what sorts of features are required for velocity
fields to achieve maximal (or nearly maximal) transport. It is natural
to wonder how the overall character of optimal designs depends on
the intensity budget. One possibility is captured by the convection
rolls pictured in \prettyref{fig:rollsrough}. This design is a relatively
simple one, as the number of lengthscales required to describe it
remains independent of the intensity budget. A second, much more complicated
possibility is captured by the branching designs from \prettyref{fig:branchingrough}.
There, the total number of lengthscales is allowed to depend on the
intensity budget and can be unbounded as $Pe\to\infty$. In any case,
one requires a general method by which to evaluate $Nu$ to allow
for comparison between candidate designs. 

The best scenario would be to develop an ansatz-free approach to evaluating
heat transport that, by its functional form, suggests optimal designs.
In this section, we achieve this for a general class of steady (i.e.,
time-independent) wall-to-wall optimal transport problems, including
the energy- and enstrophy-constrained ones as special cases. The class
of problems we have in mind are of the form
\begin{equation}
\sup_{\substack{\mathbf{u}(\mathbf{x})\\
||\mathbf{u}||=Pe\\
+\text{b.c.}
}
}\,Nu(\mathbf{u})\label{eq:optimalw2w_generalintensity-1}
\end{equation}
where $||\cdot||$ denotes any norm in which the advective intensity
of $\mathbf{u}$ may be measured. As described in the introduction,
the steady energy-constrained problem arises from employing the (volume-averaged)
$L^{2}$-norm
\[
||\mathbf{u}||=\left(\fint_{\Omega}|\mathbf{u}|^{2}\right)^{1/2}
\]
to measure advective intensity, while the steady enstrophy-constrained
one arises from the (volume-averaged) $\dot{H}^{1}$-norm
\[
||\mathbf{u}||=\left(\fint_{\Omega}|\nabla\mathbf{u}|^{2}\right)^{1/2}
\]
In any case, we require that $\mathbf{u}\in L^{2}$ in order that
its heat transport be well-defined. As our aim in this section is
to present a general approach to intensity-constrained optimal transport,
we leave the boundary conditions unspecified. Of course, we do not
claim that there exist optimizers at this level of generality.

The principal result of this section is that the general wall-to-wall
optimal transport problem \prettyref{eq:optimalw2w_generalintensity-1}
can be reformulated as the double minimization 
\begin{equation}
\inf_{\substack{\mathbf{u}(\mathbf{x}),\xi(\mathbf{x})\\
\fint_{\Omega}w\xi=1\\
+\text{b.c.}
}
}\,\fint_{\Omega}\left|\nabla\Delta^{-1}\text{div}(\mathbf{u}\xi)\right|^{2}+\frac{1}{Pe^{2}}||\mathbf{u}||^{2}\cdot\fint_{\Omega}|\nabla\xi|^{2}\label{eq:optimaldesgnpblm_general-1}
\end{equation}
in the velocity field $\mathbf{u}$ and a new variable $\xi$. The
boundary conditions for $\mathbf{u}$ remain the same as for \prettyref{eq:optimalw2w_generalintensity-1},
while $\xi$ is required to vanish at $\partial\Omega$. As will become
clear, $\xi$ plays a role in the analysis of $Nu$ similar to that
of $\eta$ from the \emph{a priori }bounds of \prettyref{sec:aprioribds} \textemdash in
fact, these variables are dual. The optimal values in \prettyref{eq:optimalw2w_generalintensity-1}
and \prettyref{eq:optimaldesgnpblm_general-1} are reciprocals, and
their optimizers are related through a certain change of variables.
We refer the reader forward to \prettyref{sec:energy-constraineddesign}
and \prettyref{sec:enstrophy_constrained_design} for the application
of these observations to energy- and enstrophy-constrained optimal
transport. Presently, our goal is to establish the connection between
\prettyref{eq:optimalw2w_generalintensity-1} and \prettyref{eq:optimaldesgnpblm_general-1},
and to illustrate how the latter suggests optimal designs. As in \prettyref{sec:aprioribds},
our approach centers on the existence of a variational principle for
$Nu(\mathbf{u})$; it is dual to the one appearing there. After achieving
this duality and using it to obtain \prettyref{eq:optimaldesgnpblm_general-1},
we proceed to make some general remarks on the construction of near
optimal designs. 

\subsection{Dual variational formulations for transport}

Recall from the analysis of \emph{a priori} bounds on transport in
\prettyref{sec:aprioribds} that there is a variational principle
for heat transport in the steady case, and that $Nu$ can be written
as the optimal value of a certain convex minimization problem:
\[
Nu(\mathbf{u})-1=\min_{\eta|_{\partial\Omega}=0}\,\fint_{\Omega}\left|\nabla\eta\right|^{2}+\left|\nabla\Delta^{-1}\text{div}(\mathbf{u}\eta)\right|^{2}.
\]
For the precise statement, see \prettyref{thm:steady-primal-varprin}.
As this is convex it should, in principle, admit a dual formulation. 
\begin{thm}
\label{thm:dual_Nubound_w2w} Let $\mathbf{u}$ be a divergence-free
vector field in $L^{2}(\Omega;\mathbb{R}^{3})$. Then,
\begin{equation}
Nu(\mathbf{u})-1=\max_{\xi\in H_{0}^{1}(\Omega)\cap L^{\infty}(\Omega)}\,\fint_{\Omega}2w\xi-\left|\nabla\xi\right|^{2}-\left|\nabla\Delta^{-1}\emph{div}(\mathbf{u}\xi)\right|^{2}.\label{eq:dual_steady_w2w}
\end{equation}
\end{thm}
\begin{proof}
Recall from the proof of \prettyref{thm:steady-primal-varprin} the
PDE system
\begin{equation}
\begin{cases}
\mathbf{u}\cdot\nabla\eta=\Delta\xi+w,\\
\mathbf{u}\cdot\nabla\xi=\Delta\eta
\end{cases}\label{eq:steadyPDEsystem-1}
\end{equation}
and the formula
\[
Nu-1=\fint_{\Omega}|\nabla\xi|^{2}+|\nabla\eta|^{2}.
\]
Testing the first equation in \prettyref{eq:steadyPDEsystem-1} with
$\xi$, the second with $\eta$, and integrating by parts yields the
string of equalities
\begin{align*}
\int_{\Omega}w\xi & =\int_{\Omega}|\nabla\xi|^{2}+\int_{\Omega}\xi\mathbf{u}\cdot\nabla\eta\\
 & =\int_{\Omega}|\nabla\xi|^{2}-\int_{\Omega}\eta\mathbf{u}\cdot\nabla\xi=\int_{\Omega}|\nabla\xi|^{2}+|\nabla\eta|^{2}.
\end{align*}
Thus,
\[
Nu-1=\fint_{\Omega}2w\xi-|\nabla\xi|^{2}-|\nabla\eta|^{2}=\fint_{\Omega}2w\xi-|\nabla\xi|^{2}-|\nabla\Delta^{-1}\text{div}(\mathbf{u}\xi)|^{2}.
\]
Note in the last step we used the PDE system again.

Now consider the maximization
\[
\sup_{\xi\in H_{0}^{1}(\Omega)\cap L^{\infty}(\Omega)}\,\fint_{\Omega}2w\xi-|\nabla\xi|^{2}-|\nabla\Delta^{-1}\text{div}(\mathbf{u}\xi)|^{2}.
\]
Reasoning with its Euler-Lagrange equation just as in the proof of
\prettyref{thm:steady-primal-varprin}, we deduce that this maximization
problem is well-posed in the given admissible class, with optimal
value $Nu-1$. 
\end{proof}
As claimed in the introduction, there is a corresponding result holding
for unsteady velocities which allows to bound $Nu$ from below, but
not necessarily to evaluate it. This result was described in \prettyref{eq:unsteady_lb},
and its proof is similar to that of \prettyref{thm:primal_Nubound_w2w}.
We remark that although the variational formulas for steady heat transport
from \prettyref{thm:steady-primal-varprin} and \prettyref{thm:dual_Nubound_w2w}
are strongly dual \textemdash they provide convex and concave alternatives
for evaluating $Nu$ whose optimal values agree \textemdash such strong
duality need not hold for unsteady flows. More precisely, we note
that for general velocity fields the bounds \prettyref{eq:unsteady_ub}
and \prettyref{eq:unsteady_lb} need not coincide. In particular,
there will be a duality gap for any velocity field which satisfies
\[
\liminf_{\tau\to\infty}\,\left\langle wT\right\rangle _{\tau}<\limsup_{\tau\to\infty}\,\left\langle wT\right\rangle _{\tau}.
\]
These fields have the peculiar property that the $\limsup$ and $\liminf$
alternatives for defining the space and long-time average heat transport
$Nu$ do not coincide.

\subsection{An integral formulation of wall-to-wall optimal transport\label{sec:An-integral-formulation}}

Having written $Nu$ for steady flows as the optimal value of the
concave maximization problem \prettyref{eq:dual_steady_w2w}, we can
now give a useful reformulation of the entire class of steady wall-to-wall
optimal transport problems from \prettyref{eq:optimalw2w_generalintensity-1}.
This ``integral'' formulation of steady optimal transport will be
used to design and evaluate nearly optimal flows in the subsequent. 

Let $F(Pe)$ denote the optimal value of the steady optimal transport
problem \prettyref{eq:optimalw2w_generalintensity-1},
\[
F(Pe)=\sup_{\substack{\mathbf{u}(\mathbf{x})\\
||\mathbf{u}||=Pe\\
+\text{b.c.}
}
}\,Nu(\mathbf{u}).
\]
 Applying \prettyref{thm:dual_Nubound_w2w}, we find that
\[
F(Pe)=\sup_{\substack{\mathbf{u}(\mathbf{x}),\xi(\mathbf{x})\\
||\mathbf{u}||=Pe\\
+\text{b.c.}
}
}\,\fint_{\Omega}2w\xi-\left|\nabla\xi\right|^{2}-\left|\nabla\Delta^{-1}\text{div}(\mathbf{u}\xi)\right|^{2}.
\]
The boundary conditions for $\mathbf{u}$ remain unchanged, while
according to \prettyref{thm:dual_Nubound_w2w} we must require that
$\xi|_{\partial\Omega}=0$. Performing the substitution
\[
\xi\to\lambda\xi,\quad\lambda\in\mathbb{R}
\]
and maximizing over $\lambda$ yields the equivalent variational problem
\[
\frac{1}{F(Pe)}=\inf_{\substack{\mathbf{u}(\mathbf{x}),\xi(\mathbf{x})\\
||\mathbf{u}||=Pe\\
+\text{b.c.}
}
}\,\frac{\fint_{\Omega}\left|\nabla\Delta^{-1}\text{div}(\mathbf{u}\xi)\right|^{2}+\left|\nabla\xi\right|^{2}}{\left(\fint_{\Omega}w\xi\right)^{2}}.
\]
Changing variables via the substitutions
\[
\mathbf{u}\to Pe\frac{\mathbf{u}}{||\mathbf{u}||}\quad\text{and}\quad\xi\to\frac{||\mathbf{u}||}{Pe}\xi
\]
allows to eliminate the intensity constraint on $\mathbf{u}$ altogether,
so that
\[
\frac{1}{F(Pe)}=\inf_{\substack{\mathbf{u}(\mathbf{x}),\xi(\mathbf{x})\\
+\text{b.c.}
}
}\,\frac{\fint_{\Omega}\left|\nabla\Delta^{-1}\text{div}(\mathbf{u}\xi)\right|^{2}+\frac{1}{Pe^{2}}||\mathbf{u}||^{2}\fint_{\Omega}\left|\nabla\xi\right|^{2}}{\left(\fint_{\Omega}w\xi\right)^{2}}.
\]
Given the scaling symmetries of the above, we may impose the constraint
\[
\fint_{\Omega}w\xi=1
\]
on the minimization without altering the result. This yields the promised
integral reformulation of wall-to-wall optimal transport

\begin{equation}
\frac{1}{F(Pe)}=\inf_{\substack{\mathbf{u}(\mathbf{x}),\xi(\mathbf{x})\\
\fint_{\Omega}w\xi=1\\
+\text{b.c.}
}
}\,\fint_{\Omega}\left|\nabla\Delta^{-1}\text{div}(\mathbf{u}\xi)\right|^{2}+\frac{1}{Pe^{2}}||\mathbf{u}||^{2}\cdot\fint_{\Omega}|\nabla\xi|^{2},\label{eq:optimaldesgnpblm_general}
\end{equation}
and proves the equivalence between \prettyref{eq:optimalw2w_generalintensity-1}
and \prettyref{eq:optimaldesgnpblm_general-1}. 

We turn to discuss the integral formulation of optimal transport just derived.
Observe \prettyref{eq:optimaldesgnpblm_general} consists of two types
of terms, each of which prefers a different kind of design. The first
term
\begin{equation}
\fint_{\Omega}\left|\nabla\Delta^{-1}\text{div}(\mathbf{u}\xi)\right|^{2}\label{eq:advectionterm}
\end{equation}
prefers $\mathbf{u}\xi$ to be divergence-free and we refer to it
as the ``advection term'' throughout. This preference is strong
in the advective regime $Pe\gg1$, as it appears at leading order
in $Pe^{-1}$ in the functional above. The remaining terms
\[
||\mathbf{u}||^{2}\quad\text{and}\quad\fint_{\Omega}|\nabla\xi|^{2}
\]
contribute at higher order in $Pe^{-1}$, and act to regularize designs.
Any admissible design must satisfy the ``net-flux'' constraint
\begin{equation}
\fint_{\Omega}w\xi=1\label{eq:netflux}
\end{equation}
as well as boundary conditions. While patterns such as the convection
roll and branching ones depicted in \prettyref{fig:designs} can be
easily made to satisfy such constraints \textemdash see \prettyref{sec:energy-constraineddesign}
and \prettyref{sec:enstrophy_constrained_design} for details \textemdash determining
the optimal lengthscales for such designs requires performing an optimization
as in \prettyref{eq:optimaldesgnpblm_general}.

Evidently, the most difficult term to evaluate is the advection one
\prettyref{eq:advectionterm}. Before turning to discuss its analysis
in detail, and what it implies for near optimal designs, we make two
general remarks. In order to get a hint as to what designs \prettyref{eq:optimaldesgnpblm_general}
prefers in the advective limit $Pe\to\infty$, one can entertain
the ``limiting'' wall-to-wall optimal transport problem
\[
\inf_{\substack{\mathbf{u}(\mathbf{x}),\xi(\mathbf{x})\\
\fint_{\Omega}w\xi=1\\
+\text{b.c.}
}
}\,\fint_{\Omega}\left|\nabla\Delta^{-1}\text{div}(\mathbf{u}\xi)\right|^{2}
\]
obtained by formally taking $Pe = \infty$ in \eqref{eq:optimaldesgnpblm_general}. This, however, is an ill-posed variational problem. Its optimal value
is zero as there exist admissible sequences $\{(\mathbf{u}_{k},\xi_{k})\}$
satisfying the net-flux constraint \prettyref{eq:netflux} and achieving
\[
\lim_{k\to\infty}\,\fint_{\Omega}|\nabla\Delta^{-1}\text{div}(\mathbf{u}_{k}\xi_{k})|^{2}=0.
\]
Yet, no smooth enough admissible pair $(\mathbf{u},\xi)$ can satisfy
the net-flux constraint and simultaneously achieve 
\begin{equation}
\text{div}(\mathbf{u}\xi)=0.\label{eq:div-free}
\end{equation}
Indeed, if $\mathbf{u}\xi$ were divergence-free then by averaging
\prettyref{eq:div-free} in the periodic variables $x$ and $y$ we
would find that the flux of $\xi$ by $\mathbf{u}$ through each slice
$\{z=\text{const.}\}$ is independent of the slice, i.e.,
\[
\overline{w\xi}(z)=\fint_{\mathbb{T}_{x,y}^{2}}w\xi
\]
is constant in $z$. Applying the boundary conditions which require
at least that $\xi|_{\partial\Omega}=0$, we conclude that $\overline{w\xi}$
must vanish throughout the entire domain. This contradicts the net-flux
constraint \prettyref{eq:netflux}. Therefore, wall-to-wall optimal
transport is a singularly perturbed variational problem: the regularizing
terms from \prettyref{eq:optimaldesgnpblm_general} which at first
glance appear to contribute at higher order in $Pe^{-1}$ are crucial
for determining the character of optimal designs. 

Our second observation is more straightforward: it is regarding the
disappearance of the intensity constraint in the passage from \prettyref{eq:optimalw2w_generalintensity-1}
to its integral formulation \prettyref{eq:optimaldesgnpblm_general}.
Since \prettyref{eq:optimaldesgnpblm_general} is invariant under
the rescaling
\[
\mathbf{u}\to\lambda\mathbf{u}\quad\text{and}\quad\xi\to\frac{1}{\lambda}\xi,\quad\lambda\neq0,
\]
the magnitudes of any of its minimizers are not uniquely determined.
Still, if $(\mathbf{u}_{Pe},\xi_{Pe})$ achieves optimality in \prettyref{eq:optimaldesgnpblm_general},
\[
\mathbf{u}=\frac{Pe}{||\mathbf{u}_{Pe}||}\mathbf{u}_{Pe}
\]
solves the wall-to-wall problem \prettyref{eq:optimalw2w_generalintensity-1}. 

\subsection{Analysis of the advection term}

For a class of designs $\{(\mathbf{u}_{\alpha},\xi_{\alpha})\}_{\alpha\in\mathcal{I}}$
to compete in the minimization \prettyref{eq:optimaldesgnpblm_general},
it must at least achieve
\[
\inf_{\alpha\in\mathcal{I}}\,\fint_{\Omega}\left|\nabla\Delta^{-1}\text{div}(\mathbf{u}_{\alpha}\xi_{\alpha})\right|^{2}=0.
\]
How difficult is it for an admissible pair $(\mathbf{u},\xi)$ to
make this advection term \emph{nearly} zero? First, note that in such
a situation, the vertical flux of $\xi$ by $\mathbf{u}$ through
each slice $\{z=\text{const.}\}$ must be nearly independent of the
slice, 
\[
\frac{d}{dz}\overline{w\xi}\approx0.
\]
By the net-flux constraint \prettyref{eq:netflux}, it follows that
\begin{equation}
\overline{w\xi}\approx1\label{eq:design principle}
\end{equation}
in nearly all of the domain. This is an example of a ``design principle''
for wall-to-wall optimal transport: any nearly optimal design must
achieve \prettyref{eq:design principle} with equality in the limit
$Pe\to\infty$. Although \prettyref{eq:design principle} does not
completely characterize optimal designs, it does give a necessary
condition for constructing competitive ones. This will be particularly
useful later on in \prettyref{sec:enstrophy_constrained_design},
where we devise a functional form for the branching depicted in \prettyref{fig:branchingrough}. 

The advection term \prettyref{eq:advectionterm} contains a wealth
of information for evaluating designs beyond \prettyref{eq:design principle},
but to use it in practice one must confront its non-locality. In
the wall-to-wall domain $\Omega=\mathbb{T}_{xy}^{2}\times[0,1]_{z}\cong[0,l_{x}]\times[0,l_{y}]\times[0,1]$,
we may take advantage of periodicity to represent a function $f\in L^{2}(\Omega)$
via its Fourier series 
\[
f(\mathbf{x})=\sum_{\mathbf{k}\in\mathbb{Z}_{l_{x},l_{y}}^{2}}\hat{f}_{\mathbf{k}}(z)e^{i\mathbf{k}\cdot\mathbf{x}'}
\]
where $\mathbf{x}'=(x,y)$ and 
\[
\mathbb{Z}_{l_{x},l_{y}}^{2}=\left\{ (k_{x},k_{y}):\frac{l_{x}}{2\pi}k_{x},\frac{l_{y}}{2\pi}k_{y}\in\mathbb{Z}\right\} .
\]
We employ the Fourier transform 
\[
\hat{f}_{\mathbf{k}}(z)=\fint_{\mathbb{T}_{xy}^{2}}e^{i\mathbf{k}\cdot\mathbf{x}'}f(\mathbf{x}',z)\,dxdy
\]
and proceed to decompose the advection term mode-by-mode.
\begin{lem}
\label{lem:advection_Fourierests_w2w} Let $\mathbf{u}\in L^{2}(\Omega;\mathbb{R}^{3})$
and $\xi\in L^{\infty}(\Omega)$. The advection term satisfies 
\[
\fint_{\Omega}\left|\nabla\Delta^{-1}\emph{div}(\mathbf{u}\xi)\right|^{2}=\fint_{I_{z}}|\overline{w\xi}-\fint_{\Omega}w\xi|^{2}+\mathcal{Q}(\emph{div}(\mathbf{u}\xi))
\]
where $\mathcal{Q}=\sum_{\mathbf{k}\neq\mathbf{0}}Q_{\mathbf{k}}$
and $\mathcal{Q}_{\mathbf{k}}$ is the positive semi-definite quadratic
form given by
\[
\mathcal{Q}_{\mathbf{k}}(f)=\fint_{I_{z}\times I_{z'}}G_{\mathbf{k}}(z,z')\hat{f}_{\mathbf{k}}(z)\hat{f}_{\mathbf{k}}^{*}(z')\,dzdz'
\]
with kernel
\[
G_{\mathbf{k}}(z,z')=\frac{\emph{csch}(|\mathbf{k}|)}{|\mathbf{k}|}\times\begin{cases}
\sinh(|\mathbf{k}|z)\sinh(|\mathbf{k}|(1-z')) & z\leq z'\\
\sinh(|\mathbf{k}|z')\sinh(|\mathbf{k}|(1-z)) & z\geq z'
\end{cases}\qquad\mathbf{k}\neq0.
\]
\end{lem}
\begin{proof}
This follows from the Green's function representation for $-\Delta^{-1}$
on $\Omega$ with vanishing Dirichlet boundary conditions. Calling
$J=\text{div}(\mathbf{u}\xi)$ and applying Parseval's identity, we
see that the advection term can be written as
\begin{align*}
\fint_{\Omega}\left|\nabla\Delta^{-1}\text{div}(\mathbf{u}\xi)\right|^{2} & =\frac{1}{|\Omega|}\left(J,-\Delta^{-1}J\right)=\sum_{\mathbf{k}}\fint_{I_{z}}(-\frac{d^{2}}{dz^{2}}+|\mathbf{k}|^{2})^{-1}\hat{J}_{\mathbf{k}}\hat{J}_{\mathbf{k}}^{*}\\
 & =\sum_{\mathbf{k}}\fint_{I_{z}\times I_{z'}}G_{\mathbf{k}}(z,z')\hat{J}_{\mathbf{k}}(z)\hat{J}_{\mathbf{k}}^{*}(z')\,dzdz'
\end{align*}
where for $\mathbf{k}\neq\mathbf{0}$ the functions $G_{\mathbf{k}}$
are as defined above. For $\mathbf{k}=\mathbf{0}$, we have 
\[
G_{\mathbf{0}}(z,z')=\begin{cases}
z(1-z') & z\leq z'\\
z'(1-z) & z\geq z'
\end{cases}.
\]
We must only address the form of the $\mathbf{0}$th term now.

We recognize that
\[
\fint_{I_{z}\times I_{z'}}G_{\mathbf{0}}(z,z')\hat{J}_{\mathbf{0}}(z)\hat{J}_{\mathbf{0}}(z')\,dzdz'=\fint_{\Omega}|\nabla\Delta^{-1}\overline{J}(z)|^{2}.
\]
By periodicity,
\[
\overline{J}=\overline{\text{div}(\mathbf{u}\xi)}=\frac{d}{dz}\overline{w\xi}.
\]
Thus,
\begin{align*}
\fint_{\Omega}|\nabla\Delta^{-1}\overline{J}(z)|^{2}& =\fint_{\Omega}|\frac{d}{dz}\left(\frac{d}{dz}\right)^{-2}\frac{d}{dz}\overline{w\xi}(z)|^{2} \\ 
& =\fint_{I_{z}}|\overline{w\xi}-\fint_{I_{z}}\overline{w\xi}|^{2}=\fint_{I_{z}}|\overline{w\xi}-\fint_{\Omega}w\xi|^{2}
\end{align*}
as required.
\end{proof}
As the quadratic form $\mathcal{Q}$ from \prettyref{lem:advection_Fourierests_w2w}
is non-negative, the advection term satisfies the lower bound
\[
\fint_{\Omega}\left|\nabla\Delta^{-1}\text{div}(\mathbf{u}\xi)\right|^{2}\geq\fint_{I_{z}}|\overline{w\xi}-\fint_{\Omega}w\xi|^{2}.
\]
Note this quantifies the design principle \prettyref{eq:design principle}.
The appearance of $\mathcal{Q}$ in \prettyref{lem:advection_Fourierests_w2w}
also makes clear why this principle alone does not suffice to characterize
optimal designs.

We end this section by recording some useful estimates on the kernels
$\{G_{\mathbf{k}}\}$ from the definition of $\mathcal{Q}$. These
will be used later in \prettyref{sec:enstrophy_constrained_design}.
\begin{lem}
\label{lem:Gestimates} Let $A\subset I_{z}$ be Borel measurable.
Then, 
\[
||G_{\mathbf{k}}||_{L^{1}(A\times A)}\lesssim\frac{|A|}{|\mathbf{k}|}(|A|\wedge\frac{1}{|\mathbf{k}|})\quad\mathbf{k}\neq0
\]
and 
\[
||G_{\mathbf{0}}||_{L^{1}(A\times A)}\lesssim|A|^{2}\fint_{A}z\wedge(1-z)\,dz.
\]
\end{lem}
\begin{proof}
To see the first estimate, observe that $G_{\mathbf{k}}$ satisfies
the pointwise estimate
\[
|G_{\mathbf{k}}(z,z')|\leq\frac{1}{|\mathbf{k}|}e^{-|\mathbf{k}||z-z'|}\quad\mathbf{k}\neq\mathbf{0}.
\]
Now,
\[
\int_{I_{z}\times I_{z}}e^{-|\mathbf{k}||z-z'|}\indicator{A}(z)\indicator{A}(z')\,dzdz'\leq\int_{I_{z}}\left[\int_{\R}e^{-|\mathbf{k}||z-z'|}dz'\right]\indicator{A}(z)\,dz=\frac{2}{|\mathbf{k}|}|A|.
\]
Therefore,
\[
||G_{\mathbf{k}}||_{L^{1}(A\times A)}\lesssim\frac{|A|}{|\mathbf{k}|^{2}}.
\]
On the other hand, we have that
\[
||G_{\mathbf{k}}||_{L^{1}(A\times A)}\leq||G_{\mathbf{k}}||_{L^{\infty}}||1_{A\times A}||_{L^{1}}\leq\frac{|A|^{2}}{|\mathbf{k}|}.
\]
Combining these two bounds gives the first result.

Now we prove the second estimate. We need to show that
\[
\int G_{\mathbf{0}}(z,z')\indicator{A}(z)\indicator{A}(z')\,dzdz'\lesssim|A|\int_{A}z\wedge(1-z)\,dz.
\]
By symmetry, 
\begin{align*}
\int G_{\mathbf{0}}(z,z'))\indicator{A}(z)\indicator{A}(z')\,dzdz' & =2\int_{0}^{1}\int_{0}^{z'}z(1-z')1_{A}(z)1_{A}(z')\,dzdz'\\
 & \leq2\int_{0}^{1}\int_{0}^{z'}z'(1-z')1_{A}(z)1_{A}(z')\,dzdz' \\ 
 & \leq2|A|\int_{0}^{1}z(1-z)\,dz.
\end{align*}
Since 
\[
z(1-z)\leq z\wedge(1-z)\quad\forall\,z\in[0,1]
\]
we conclude the desired result.
\end{proof}

\section{Energy-constrained transport and convection roll designs\label{sec:energy-constraineddesign}}

In the previous section, we considered the general class of steady
wall-to-wall optimal transport problems \prettyref{eq:optimalw2w_generalintensity-1}
and produced their equivalent integral formulations \prettyref{eq:optimaldesgnpblm_general-1}.
In this section and the next, we use these formulations to study the
steady energy- and enstrophy-constrained problems. The subsequent
analyses are largely independent. Nevertheless, the reader may find
it helpful to study the energy-constrained problem first as its proof
is much shorter and its technical details much less burdensome.

Here we discuss energy-constrained transport. The main result of this
section is a proof of the lower bound from \prettyref{thm:mainbounds_w2w_eng}.
Recall from \prettyref{sec:OptlDesignProb_w2w} that the steady energy-constrained
optimal transport problem
\begin{equation}
\max_{\substack{\mathbf{u}(\mathbf{x})\\
\fint_{\Omega}|\mathbf{u}|^{2}=Pe^{2}\\
w|_{\partial\Omega}=0
}
}\,Nu(\mathbf{u})\label{eq:steady_w2w_eng}
\end{equation}
admits the integral formulation 
\begin{equation}
\min_{\substack{\mathbf{u}(\mathbf{x}),\xi(\mathbf{x})\\
\fint_{\Omega}w\xi=1\\
w|_{\partial\Omega}=\xi|_{\partial\Omega}=0
}
}\,\fint_{\Omega}\left|\nabla\Delta^{-1}\text{div}(\mathbf{u}\xi)\right|^{2}+\epsilon\fint_{\Omega}|\mathbf{u}|^{2}\cdot\fint_{\Omega}|\nabla\xi|^{2}\label{eq:optimaldesign_w2weng}
\end{equation}
with $\epsilon=Pe^{-2}$. The optimal values of \prettyref{eq:steady_w2w_eng}
and \prettyref{eq:optimaldesign_w2weng} are reciprocals and their
optimizers are related through symmetrization. 

Our goal now is to identify the scaling law of \prettyref{eq:optimaldesign_w2weng}
in the advective regime $\epsilon\ll1$. Combined with the results
of \prettyref{sec:OptlDesignProb_w2w}, this completes the proof of
the lower bound half of \prettyref{thm:mainbounds_w2w_eng}. 
\begin{prop}
\label{prop:lowerbound_engconstr_w2w} Let $\mathfrak{E}(\epsilon)$
denote the optimal value of \prettyref{eq:optimaldesign_w2weng}.
Then, 
\[
\mathfrak{E}(\epsilon)\sim\epsilon^{1/2}
\]
whenever $\epsilon\lesssim1\wedge l_{x}^{4}\wedge l_{y}^{4}$. 
\end{prop}
The \emph{a priori }lower bound $\mathfrak{E}(\epsilon)\gtrsim\epsilon^{1/2}$
is implied by the corresponding bound $Nu\lesssim Pe$ from \prettyref{thm:mainbounds_w2w_eng}.
The remainder of this section is regarding the upper bound $\mathfrak{E}(\epsilon)\lesssim\epsilon^{1/2}$.
We prove it by constructing the convection roll designs depicted in
\prettyref{fig:rollsrough}. Such designs can be parameterized using
two variables: the number of rolls and their wall-to-wall extent.
Carrying out the optimization from \prettyref{eq:optimaldesign_w2weng}
with respect to these variables yields the desired upper bound. The
condition that $\epsilon$ be small enough in the statement above
is required to fit what would be, in the absence of an overall horizontal
period, an optimal number of rolls inside the domain. Given the symmetry
between $x$ and $y$, we may suppose that $l_{x}\leq l_{y}$ in what
follows. 

\subsection{Convection roll designs}

The integral formulation \prettyref{eq:optimaldesign_w2weng} requires
designing a velocity field $\mathbf{u}$ and a test function $\xi$.
For the velocity field, we introduce a family of streamfunctions of
the form
\[
\psi(x,z)=\chi(z)\Psi(x)\quad\chi\in C_{c}^{\infty}(I_{z}),\ \Psi\in C^{\infty}(\mathbb{T}_{x}).
\]
Each such $\psi$ gives rise to a divergence-free velocity field by
\[
\mathbf{u}=(-\partial_{z}\psi,0,\partial_{x}\psi).
\]
These are two-dimensional flows as their $\hat{j}$-component vanishes
identically. Although we do not claim that optimizers must be of this
form, we will prove that such a construction suffices to capture the
optimal scaling law of \prettyref{eq:optimaldesign_w2weng}.

Next we must describe test functions $\xi$ well-suited to the velocity
fields. Recall the design principle \prettyref{eq:design principle},
which states that for a design to be competitive it must satisfy 
\begin{equation}
\overline{w\xi}\approx\fint_{\Omega}w\xi=1.\label{eq:designprinciple-recalled}
\end{equation}
This rules out taking, for instance, $\xi=\psi$ as it would result
in zero flux thru each slice $\{z=\text{const}.\}$. We can, however,
choose $\xi$ to depend only on $x$ and $z$ as does $\psi$. Then
by Parseval's identity we can rewrite the flux as 
\[
\overline{w\xi}=\sum_{\mathbf{k}}\hat{w}_{\mathbf{k}}(z)\hat{\xi}_{\mathbf{k}}^{*}(z)=\sum_{\mathbf{k}}(ik_{x}^{1/2}\hat{\psi}_{\mathbf{k}})(k_{x}^{1/2}\hat{\xi}_{\mathbf{k}})^{*}.
\]
Taking
\[
|\partial_{x}^{1/2}\xi|\sim|\partial_{x}^{1/2}\psi|
\]
allows to satisfy \prettyref{eq:designprinciple-recalled}.

Now we make the convection roll construction concrete. Given $\delta\in(0,1/2)$
and $l$ such that 
\[
\frac{1}{l}\in\frac{2\pi}{l_{x}}\mathbb{N},
\]
we define
\begin{align*}
\psi_{\delta,l}(x,z) & =\chi_{\delta}(z)\cdot l^{1/2}\Psi(\frac{x}{l})\\
\xi_{\delta,l}(x,z) & =\chi_{\delta}(z)\cdot l^{1/2}\Psi'(\frac{x}{l})
\end{align*}
where 
\[
\Psi(x)=c_{0}\cos x
\]
and $c_{0}$ is chosen so that $\overline{(\Psi')^{2}}=1$. Here,
the cut-off functions $\{\chi_{\delta}\}\subset C_{c}^{\infty}(I_{z})$
are required to satisfy 
\begin{itemize}
\item $\chi_{\delta}=1$ on $[\delta,1-\delta]$,
\item $|\chi_{\delta}|\lesssim1$ and $|\chi'_{\delta}|\lesssim\frac{1}{\delta}$,
\item $\fint_{I_{z}}\chi_{\delta}^{2}=1$.
\end{itemize}
The constants in these assumptions are independent of all parameters.
In what follows, we often neglect to record the subscripts $\delta$
and $l$ as the meaning is clear.

First, we check admissibility.
\begin{lem}
The convection roll construction described above is admissible for
\prettyref{eq:optimaldesign_w2weng}. 
\end{lem}
\begin{proof}
All conditions in admissibility are clear, except for the net-flux
constraint which we verify now. Given the above, we find that
\[
\overline{w\xi}=\overline{\partial_{x}\psi\xi}=\chi_{\delta}^{2}\overline{(\Psi')^{2}}=\chi_{\delta}^{2}
\]
for all $z$, so that
\[
\fint_{\Omega}w\xi=\fint_{I_{z}}\overline{w\xi}=\fint_{I_{z}}\chi_{\delta}^{2}=1
\]
as required.
\end{proof}
Next, we estimate the advection term from \prettyref{eq:optimaldesign_w2weng}.
\begin{lem}
The convection roll construction satisfies 
\[
\fint_{\Omega}|\nabla\Delta^{-1}\emph{div}(\mathbf{u}\xi)|^{2}=\int_{I_{z}}|\overline{w\xi}-1|^{2}\lesssim\delta.
\]
In particular, the quadratic form $\mathcal{Q}$ from \prettyref{lem:advection_Fourierests_w2w}
vanishes on it.
\end{lem}
\begin{proof}
We apply \prettyref{lem:advection_Fourierests_w2w}. Note that since
all functions entering into the construction are independent of $y$,
we may work with $k$ in place of $\mathbf{k}=(k_{x},0)$. We start
with the $k=0$ term from \prettyref{lem:advection_Fourierests_w2w}:
it satisfies the estimate
\[
\int_{I_{z}}|\overline{w\xi}-1|^{2}=\int_{I_{z}}|\chi^{2}-1|^{2}\lesssim\delta.
\]
Now we address the $k\neq0$ terms. We must compute the quadratic
form $\mathcal{Q}$ from \prettyref{lem:advection_Fourierests_w2w},
and to do so we must compute $\hat{J}_{k}$ for $k\neq0$ where $J=\mathbf{u}\cdot\nabla\xi$.
Note that by the form of the convection roll construction,
\begin{align*}
J & =\nabla^{\perp}\psi\cdot\nabla\xi=-\partial_{z}\psi\partial_{x}\xi+\partial_{x}\psi\partial_{z}\xi\\
 & =-\left(\chi_{\delta}'(z)\cdot l^{1/2}\Psi(\frac{x}{l})\right)\left(\chi_{\delta}(z)\cdot\frac{1}{l^{1/2}}\Psi''(\frac{x}{l})\right) \\ 
 & \qquad +\left(\chi_{\delta}(z)\frac{1}{l^{1/2}}\Psi'(\frac{x}{l})\right)\left(\chi_{\delta}'(z)\cdot l^{1/2}\Psi'(\frac{x}{l})\right)\\
 & =\frac{1}{2}(\chi_{\delta}^{2})'(z)\cdot\Theta(\frac{x}{l})
\end{align*}
where $\Theta=(\Psi')^{2}-\Psi\Psi''$. Since $\Theta=c_{0}^{2}\sin^{2}+c_{0}^{2}\cos^{2}=c_{0}^{2}$,
\[
J=\frac{1}{2}(\chi_{\delta}^{2})'(z)
\]
which is entirely a function of $z$. This shows that $\hat{J}_{k}=0$
for $k\neq0$. Hence, $\mathcal{Q}$ vanishes on the convection roll
construction. 
\end{proof}
Continuing, we estimate the higher order terms from \prettyref{eq:optimaldesign_w2weng}. 
\begin{lem}
The convection roll construction satisfies 
\[
\fint_{\Omega}|\nabla\xi|^{2}\vee\fint_{\Omega}|\mathbf{u}|^{2}\lesssim\frac{l}{\delta}+\frac{1}{l}.
\]
\end{lem}
\begin{proof}
Clearly, 
\[
\fint_{\Omega}|\nabla\xi|^{2}\sim\fint_{\Omega}|\nabla\psi|^{2}
\]
since $||\Psi^{(k)}||_{L^{2}(\mathbb{T}_{x})}\sim1$ for all $k$.
Noting that
\[
\fint_{\Omega}|\mathbf{u}|^{2}=\fint_{\Omega}|\nabla\psi|^{2}=\fint_{\Omega}(\chi'_{\delta})^{2}l\Psi^{2}+\chi_{\delta}\frac{1}{l}(\Psi')^{2}\lesssim\frac{l}{\delta}+\frac{1}{l}
\]
we conclude the result.
\end{proof}
Combining the above yields the following estimate on 
\[
E(\epsilon;\delta,l)=\fint_{\Omega}|\nabla\Delta^{-1}\text{div}(\mathbf{u}_{\delta,l}\xi_{\delta,l})|^{2}+\epsilon\fint_{\Omega}|\mathbf{u}_{\delta,l}|^{2}\cdot\fint_{\Omega}|\nabla\xi_{\delta,l}|^{2}.
\]

\begin{cor}
\label{cor:convection_roll_est}The convection roll construction satisfies
\[
E(\epsilon;\delta,l)\lesssim\delta+\epsilon\left(\frac{l}{\delta}+\frac{1}{l}\right)^{2}.
\]
\end{cor}

\subsection{Proof of \prettyref{prop:lowerbound_engconstr_w2w}}

Finally we choose $\delta$ and $l$ to prove the desired bound.

\begin{proof}[Proof of \prettyref{prop:lowerbound_engconstr_w2w}]
\prettyref{cor:convection_roll_est} holds for all admissible $\delta$
and $l$, i.e., so long as $\delta\in(0,\frac{1}{2})$ and $l^{-1}\in2\pi l_{x}^{-1}\mathbb{N}$.
To optimize the result, we take 
\[
l\sim\sqrt{\delta}\quad\text{and}\quad\delta\sim\epsilon^{1/2}
\]
which we can do so long as $\epsilon\lesssim1\wedge l_{x}^{4}$. Under
this condition, 
\[
\inf_{\delta,l}\,E(\epsilon;\delta,l)\lesssim\epsilon^{1/2}.
\]

\end{proof}

\section{Enstrophy-constrained transport and branched flow designs \label{sec:enstrophy_constrained_design}}

We now consider the steady enstrophy-constrained wall-to-wall
optimal transport problem in the framework of \prettyref{sec:OptlDesignProb_w2w}.
The main result of this section is a proof of the lower bound from
\prettyref{thm:mainbounds_w2w_ens}. As in the previous section on
energy-constrained transport, we exploit the fact that
the enstrophy-constrained problem
\begin{equation}
\max_{\substack{\mathbf{u}(\mathbf{x})\\
\fint_{\Omega}|\nabla\mathbf{u}|^{2}=Pe^{2}\\
\mathbf{u}|_{\partial\Omega}=\mathbf{0}
}
}\,Nu(\mathbf{u})\label{eq:steady_w2w_enstrophy}
\end{equation}
can be written in integral form as
\begin{equation}
\min_{\substack{\mathbf{u}(\mathbf{x}),\xi(\mathbf{x})\\
\fint_{\Omega}w\xi=1\\
\mathbf{u}|_{\partial\Omega}=\mathbf{0},\xi|_{\partial\Omega}=0
}
}\,\fint_{\Omega}\left|\nabla\Delta^{-1}\text{div}(\mathbf{u}\xi)\right|^{2}+\epsilon\fint_{\Omega}|\nabla\mathbf{u}|^{2}\cdot\fint_{\Omega}|\nabla\xi|^{2}\label{eq:optimaldesign_w2wenstr}
\end{equation}
where $\epsilon=Pe^{-2}$. This form of the problem suggests the possibility
of analyzing a certain multiple scales ansatz for $\mathbf{u}$ and
$\xi$ (we explain the intuition behind this further in \prettyref{sec:Implications-for-RBC}
and \prettyref{sec:energy-driven-pattern-formation}). As proved below,
such an ansatz turns out to capture the scaling of the optimal value
of \prettyref{eq:optimaldesign_w2wenstr} in $\epsilon$ up to possible
logarithmic corrections. The precise statement is as follows:
\begin{prop}
\label{prop:lowerbound_enstrconstr_w2w} Let $\mathfrak{E}(\epsilon)$
denote the optimal value of \prettyref{eq:optimaldesign_w2wenstr}.
Then, 
\[
\epsilon^{1/3}\lesssim\mathfrak{E}(\epsilon)\lesssim\epsilon^{1/3}\log^{4/3}\frac{1}{\epsilon}
\]
whenever $\epsilon\lesssim1$ and $\epsilon\log\frac{1}{\epsilon}\lesssim l_{x}^{6}\wedge l_{y}^{6}$.
\end{prop}
The \emph{a priori }lower bound $\mathfrak{E}(\epsilon)\gtrsim\epsilon^{1/3}$
is implied by the upper bound $Nu\lesssim Pe^{2/3}$ from \prettyref{thm:mainbounds_w2w_ens}.
(For a proof which is more self-contained, see the discussion surrounding
\prettyref{eq:Howard_LB}.) To prove the upper bound $\mathfrak{E}(\epsilon)\lesssim\epsilon^{1/3}\log^{4/3}\frac{1}{\epsilon}$
we must construct a suitable class of designs and estimate their heat
transport. The successful ones are as depicted in \prettyref{fig:branchingrough}.
In contrast with the convection roll designs considered previously,
such ``branching'' designs are evidently more complicated to analyze.
The main challenge of course lies with estimating the advection term.
Here, we make use of \prettyref{lem:advection_Fourierests_w2w} and
\prettyref{lem:Gestimates}. Note the requirement that $\epsilon$
be small enough is to ensure that our construction fits into the given
domain. As in the previous section, we need only consider the case
$l_{x}\leq l_{y}$ by symmetry. 

Combined with the results of \prettyref{sec:OptlDesignProb_w2w},
\prettyref{prop:lowerbound_enstrconstr_w2w} completes the proof of
the lower bound from \prettyref{thm:mainbounds_w2w_eng}. 

\subsection{The branching construction\label{sec:branchingconstruction}}

The integral formulation \prettyref{eq:optimaldesign_w2wenstr} requires
the construction of a divergence-free velocity field $\mathbf{u}$
and a test function $\xi$ (the latter of which plays the role of
temperature in this approach). For the velocity fields, we will use
a streamfunction $\psi(x,z)$ whose streamlines are as in \prettyref{fig:branchingrough}.
That figure can be thought of as consisting of many individual convection
roll systems which have been carefully fit together. In the bulk,
there are large anisotropic convection rolls at some horizontal lengthscale
$l_{\text{bulk}}$. At the walls, there are much smaller isotropic convection
rolls at some other lengthscale $l_{\text{bl}}\ll l_{\text{bulk}}$. Between the
bulk and the walls, streamlines branch and refine through several
transition layers, a single one of which is shown in \prettyref{fig:transitionlayer}.
As the construction is symmetric about $z=1/2$, we only need describe
it for $z\in[\frac{1}{2},1]$.

Counting upwards from the bulk, we understand by the $j$th transition
layer that part of the domain where $z\in[z_{j},z_{j+1}]$. The points
$\{z_{j}\}_{j=1}^{n}$ marking the edges of the layers satisfy
\begin{equation}
\frac{1}{2}<z_{\text{bulk}}=z_{1}<z_{2}<\cdots<z_{n}=z_{\text{bl}}<1.\label{eq:zk_inequalities}
\end{equation}
At the horizontal slice $\{z=z_{j}\}$ the velocity components fluctuate
at lengthscale $l_{j}$. These decrease monotonically according as
\begin{equation}
l_{\text{bulk}}=l_{1}>l_{2}>\cdots>l_{n}=l_{\text{bl}}.\label{eq:lk_inequalities}
\end{equation}
In what follows, we think of the parameters $\{z_{j}\}_{j=1}^{n}$
and $\{l_{j}\}_{j=1}^{n}$ as playing a distinguished role in specifying
the branching design.

Given such a streamfunction $\psi$ and its corresponding two-dimensional
velocity field
\[
\mathbf{u}=(-\partial_{z}\psi,0,\partial_{x}\psi),
\]
we must choose a ``temperature'' field $\xi$ well-suited to the
minimization \prettyref{eq:optimaldesign_w2wenstr}. Recall the design
principle \prettyref{eq:design principle} discussed in \prettyref{sec:OptlDesignProb_w2w},
which requires that 
\[
\overline{w\xi}\approx\fint_{\Omega}w\xi=1
\]
throughout the domain. For our purposes, it will suffice to set
\begin{equation}
\xi=w\label{eq:pointwise_slaving}
\end{equation}
and enforce that $\overline{w^{2}}\approx1$. Such considerations
significantly constrain the way that streamlines may branch. We note
that while \prettyref{eq:pointwise_slaving} may not necessarily hold
for optimal designs, it greatly simplifies the ensuing analysis. And,
as claimed in \prettyref{prop:lowerbound_enstrconstr_w2w} and proved
below, such a choice introduces at most a logarithmic error in our
estimates of enstrophy-constrained optimal transport.

\begin{figure}
\centering

\includegraphics[width=0.6\paperwidth,height=0.15\paperheight]{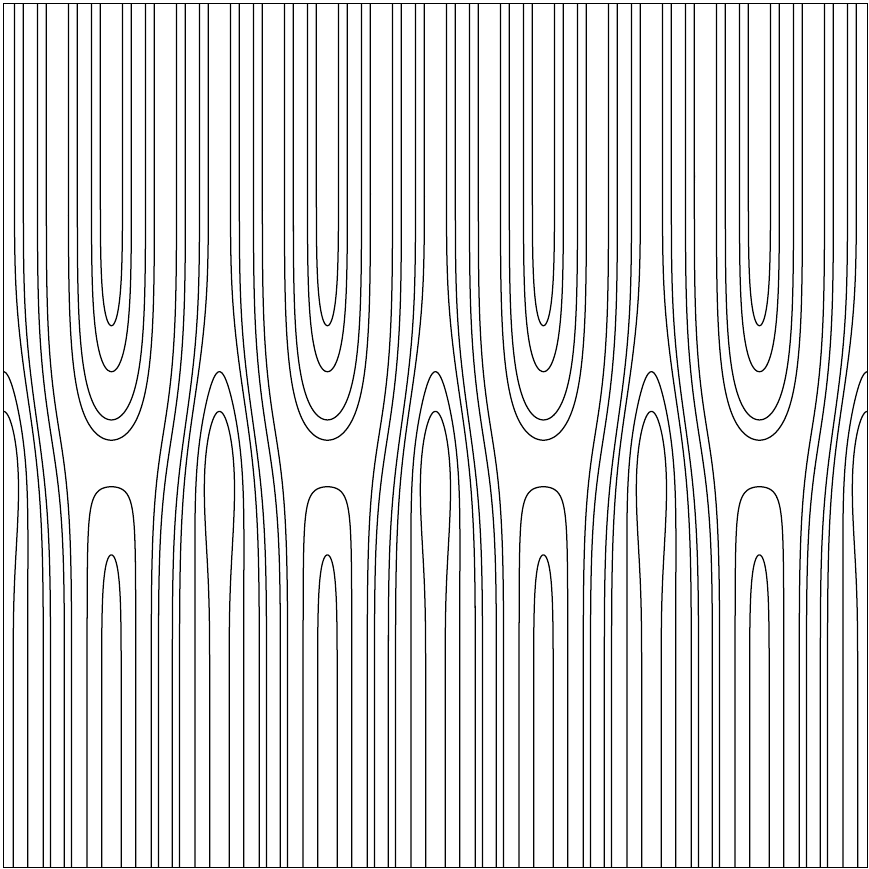}

\caption{Streamlines from a period doubling transition layer. \label{fig:transitionlayer}}
\end{figure}
%
%

We are now ready to give the precise functional form of our branching
construction. Let points $\{z_{j}\}_{j=1}^{n}$ satisfying \prettyref{eq:zk_inequalities}
and lengthscales $\{l_{j}\}_{j=1}^{n}$ satisfying 
\[
\frac{1}{l_{j}}\in\frac{2\pi}{l_{x}}\mathbb{N}
\]
and \prettyref{eq:lk_inequalities} be given. Let
\begin{equation}
\Psi(x)=c_{0}\cos x\label{eq:masterstreamfunction}
\end{equation}
where $c_{0}$ is chosen so that $\overline{(\Psi')^{2}}=1$. We define
$\{\psi_{j}\}_{j=1}^{n}$ by
\[
\psi_{j}(x)=l_{j}\Psi(\frac{x}{l_{j}}),
\]
and set
\[
\psi(x,z)=\psi_{1}(x)\quad\text{for }z\in[\frac{1}{2},z_{1}]
\]
which corresponds to the bulk. In the boundary layer, we set
\[
\psi(x,z)=g\left(\frac{z-z_{n}}{1-z_{n}}\right)\psi_{n}(x)\quad\text{for }z\in[z_{n},1].
\]
Here, $g\in C^{\infty}([0,1])$ is a fixed cutoff function satisfying
the matching conditions $g(0)=1$, $g(1)=0$, $g'(0)=g'(1)=0$, as
well as the integral condition
\begin{equation}
\int_{0}^{1}(g(t))^{2}\,dt=1.\label{eq:squareintegralisone_for_g}
\end{equation}
In the $j$th transition layer we take
\[
\psi(x,z)=f\left(\frac{z-z_{j}}{z_{j+1}-z_{j}}\right)\psi_{j}(x)+f\left(\frac{z_{j+1}-z}{z_{j+1}-z_{j}}\right)\psi_{j+1}(x)\quad\text{for }z\in[z_{j},z_{j+1}]
\]
where $f\in C^{\infty}([0,1])$ is a second fixed cutoff function.
We require it to satisfy the Pythagorean condition
\begin{equation}
(f(t))^{2}+(f(1-t))^{2}=1\quad\text{for }t\in[0,1]\label{eq:pythagorean_conditions}
\end{equation}
as well as the matching conditions $f(0)=1$, $f(1)=0$, and $f'(0)=f'(1)=0$.
One might choose, for instance, 
\[
f(t)=\sqrt{\frac{1}{2}-\frac{1}{2}\tanh\left(\tau\cdot\frac{t-\frac{1}{2}}{t^{2}(1-t)^{2}}\right)}\quad\text{for }0<t<1
\]
for some $\tau\in(0,\infty)$. \prettyref{fig:cutofffunction} shows
such a cutoff function. This completes our construction of a general
class of branching designs; we turn now to select an optimal one for
use in \prettyref{eq:optimaldesign_w2wenstr}.

\begin{figure}
\centering

\includegraphics[height=0.15\paperheight]{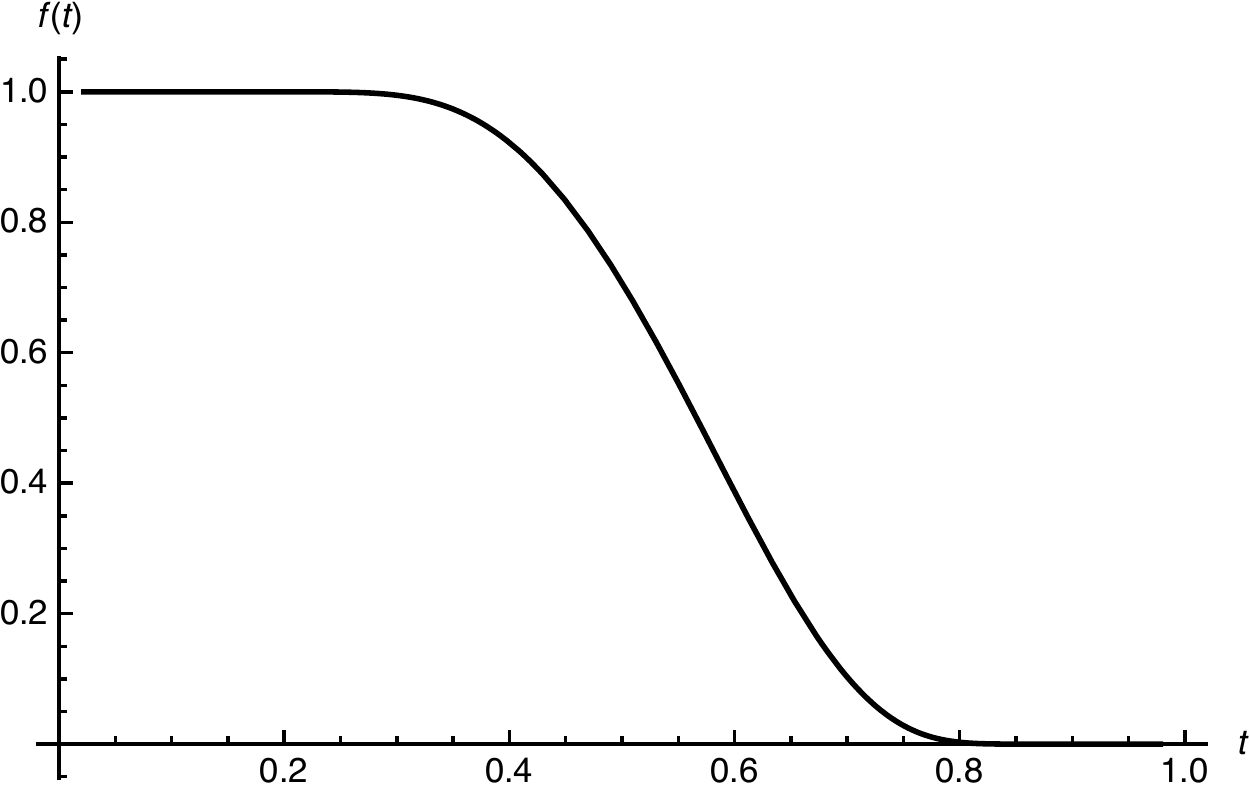}

\caption{The cutoff function $f(t)$ used in the branching construction. \label{fig:cutofffunction}}
\end{figure}
%
%

\subsection{Admissibility\label{sec:Admissibility} }

Our first task is to check the admissibility of this construction.
The Fourier series of $\psi$ can be written in the form 
\begin{equation}
\psi=\sum_{j=1}^{n}\chi_{j}(z)\psi_{j}(x),\label{eq:branchingstreamfunction_w2w}
\end{equation}
where the amplitude functions $\{\chi_{j}\}_{j=1}^{n}$ belong to
$H^{2}(I_{z})\cap L^{\infty}(I_{z})$. (Such regularity is guaranteed
by the definition of the cutoff functions $f$ and $g$.) The support
of the $j$th amplitude function satisfies
\[
\text{supp}\,\chi_{j}\subset[z_{j-1},z_{j+1}]\cup[1-z_{j+1},1-z_{j-1}]
\]
and the expansion is nearly diagonal in the sense that 
\[
\chi_{j}\chi_{j'}\not\equiv0\iff|j-j'|\leq1.
\]
Thus,
\begin{equation}
\text{supp}\,\chi_{j}\chi_{j+1}\subset[z_{j},z_{j+1}]\cup[1-z_{j+1},1-z_{j}]\quad\forall\,j\label{eq:cross_support}
\end{equation}
and the same is true for products of derivatives thereof. Note also
that by \prettyref{eq:pythagorean_conditions},
\begin{equation}
\sum_{j=1}^{n}\chi_{j}^{2}=1\quad z\in[1-z_{\text{bl}},z_{\text{bl}}]\label{eq:chisquaresum}
\end{equation}
and by \prettyref{eq:squareintegralisone_for_g},
\begin{equation}
\fint_{z_{\text{bl}}}^{1}\chi_{n}^{2}\,dz=1.\label{eq:blsquareintegral_is_one}
\end{equation}
 
\begin{lem}
The branching construction defined above is admissible for \prettyref{eq:optimaldesign_w2wenstr}. 
\end{lem}
\begin{proof}
That the boundary conditions for $\mathbf{u}$ and $\xi$ are met
is clear. Here, we check that the net-flux constraint $\fint_{\Omega}w\xi=1$
is satisfied. Since
\[
\partial_{x}\psi=\sum_{j=1}^{n}\chi_{j}\psi_{j}'
\]
and $\{\psi_{j}'\}_{j=1}^{n}$ forms an $L^{2}$-orthonormal set,
we can write that
\[
\fint_{\Omega}w\xi=\fint_{\Omega}(\partial_{x}\psi)^{2}=\int_{I_{z}}\sum_{j=1}^{n}\chi_{j}^{2}\,dz.
\]
By \prettyref{eq:chisquaresum}, the integrand is equal to one for
$z\in[1-z_{\text{bl}},z_{\text{bl}}]$. Since 
\[
\int_{z_{\text{bl}}}^{1}\chi_{n}^{2}\,dz=1-z_{\text{bl}}
\]
by \prettyref{eq:blsquareintegral_is_one}, we conclude that
\[
\fint_{\Omega}w\xi=2\left(\int_{\frac{1}{2}}^{z_{\text{bl}}}+\int_{z_{\text{bl}}}^{1}\right)\sum_{j=1}^{n}\chi_{j}^{2}\,dz=2\left(z_{\text{bl}}-\frac{1}{2}+1-z_{\text{bl}}\right)=1.
\]
\end{proof}
Next, we record some technical requirements that will greatly simplify
the identification of an optimal branching construction. These requirements
are compatible with the upper bound from \prettyref{prop:lowerbound_enstrconstr_w2w},
and we have not been able to improve upon the scaling of this result
by removing them. First, we require the transition layer thicknesses
$\{\delta_{k}\}_{k=1}^{n}$, which are defined by
\[
\delta_{k}=z_{k+1}-z_{k}\quad1\leq k\leq n-1\quad\text{and}\quad\delta_{\text{bl}}=\delta_{n}=1-z_{n},
\]
and the horizontal lengthscales $\{l_{k}\}_{k=1}^{n}$ to satisfy
the relations
\begin{equation}
\delta_{1}\gtrsim\delta_{2}\gtrsim\cdots\gtrsim\delta_{n}\label{eq:assumption2-0}
\end{equation}
and
\begin{equation}
l_{k}\lesssim\delta_{k}\quad1\leq k\leq n-1\qquad\text{and}\qquad\delta_{n}\sim l_{n}.\label{eq:assumption2-1}
\end{equation}
The latter guarantees that a certain anisotropy is present throughout
the construction which will simplify, amongst other things, the estimation
of the higher order terms from \prettyref{eq:optimaldesign_w2wenstr}.
Second, we require that 
\begin{equation}
\delta_{\text{bulk}}=z_{\text{bulk}}-\frac{1}{2}\sim1.\label{eq:assumption3}
\end{equation}
Note the constants implicit in \prettyref{eq:assumption2-0}-\prettyref{eq:assumption3}
are not allowed to depend on any parameters. Third, we require that
the refinement of lengthscale through each transition layer occur
by period doubling, i.e., 
\begin{equation}
l_{k+1}=\frac{1}{2}l_{k}\quad1\leq k\leq n-1.\label{eq:l_relations}
\end{equation}
This last requirement will serve to simplify the Fourier analysis
involved in estimating the non-local advection term.

\subsection{Estimating the efficiency of branching\label{sec:Energy-estimates}}

In this section we estimate each of the terms from \prettyref{eq:optimaldesign_w2wenstr}
for the branching construction. The requirements laid out in \prettyref{sec:branchingconstruction}
and \prettyref{sec:Admissibility} are understood to hold. The constants
appearing below are only allowed to depend on those implicit in \prettyref{eq:assumption2-0}-\prettyref{eq:assumption3},
and so do not depend on any parameters.

First, we deal with the advection term. By \prettyref{lem:advection_Fourierests_w2w},
\[
\fint_{\Omega}|\nabla\Delta^{-1}\text{div}(\mathbf{u}\xi)|^{2}=\fint_{I_{z}}|\overline{w\xi}-1|^{2}\,dz+\mathcal{Q}(\mathbf{u}\cdot\nabla\xi)
\]
where $\mathcal{Q}=\sum_{\mathbf{k}\neq\mathbf{0}}\mathcal{Q}_{\mathbf{k}}$
and 
\[
\mathcal{Q}_{\mathbf{k}}(f)=\fint_{I_{z}\times I_{z'}}G_{\mathbf{k}}(z,z')\hat{f}_{\mathbf{k}}(z)\hat{f}_{\mathbf{k}}^{*}(z')\,dzdz'.
\]
As the construction is two-dimensional, $k_{y}$ does not play a role.
For ease of reading, we denote $\mathbf{k}=(k_{x},0)$ simply by $k$
in what follows. 
\begin{lem}
\label{lem:diagonal_interactions} The branching construction satisfies
\[
\mathcal{Q}(\mathbf{u}\cdot\nabla\xi)=\sum_{i=1}^{n-1}\mathcal{Q}_{k_{i}^{\emph{sum}}}(\mathbf{u}\cdot\nabla\xi)+\mathcal{Q}_{k_{i}^{\emph{diff}}}(\mathbf{u}\cdot\nabla\xi)
\]
where
\[
k_{i}^{\emph{sum}}=\frac{1}{l_{i+1}}+\frac{1}{l_{i}}\quad\text{and}\quad k_{i}^{\emph{diff}}=\frac{1}{l_{i+1}}-\frac{1}{l_{i}}\quad\text{for }1\leq i\leq n-1.
\]
\end{lem}
\begin{proof}
Using the formula for the branching construction given in \prettyref{eq:branchingstreamfunction_w2w},
\begin{align*}
J & =\nabla^{\perp}\psi\cdot\nabla\partial_{x}\psi=\sum_{j,j'}\nabla^{\perp}(\chi_{i}\psi_{i})\cdot\nabla\partial_{x}(\chi_{j}\psi_{j})=\sum_{|j-j'|\leq1}\nabla^{\perp}(\chi_{j}\psi_{j})\cdot\nabla\partial_{x}(\chi_{j'}\psi_{j'})\\
 & =\sum_{j=1}^{n}\nabla^{\perp}(\chi_{j}\psi_{j})\cdot\nabla\partial_{x}(\chi_{j}\psi_{j}) \\ 
 &\qquad +\sum_{j=1}^{n-1}\nabla^{\perp}(\chi_{j}\psi_{j})\cdot\nabla\partial_{x}(\chi_{j+1}\psi_{j+1})+\nabla^{\perp}(\chi_{j+1}\psi_{j+1})\cdot\nabla\partial_{x}(\chi_{j}\psi_{j})\\
 & =J_{\text{self}}+J_{\text{nbr}}.
\end{align*}
Given our choice of fundamental streamfunction \prettyref{eq:masterstreamfunction},
these expressions can be made explicit and we do so now.

The general term in the first sum, $J_{\text{self}}$, satisfies 
\begin{align*}
\nabla^{\perp}(\chi_{j}\psi_{j})\cdot\nabla\partial_{x}(\chi_{j}\psi_{j}) & =(-\chi_{j}'\psi_{j},\chi_{j}\psi_{j}')\cdot(\chi_{j}\psi_{j}'',\chi_{j}'\psi_{j}')\\
 & =-\chi_{j}'\chi_{j}\psi_{j}\psi_{j}''+\chi_{j}\chi_{j}'\psi_{j}'\psi_{j}'=\frac{1}{2}(\chi_{j}^{2})'(z)\cdot\Theta(\frac{x}{l_{j}})
\end{align*}
where $\Theta=(\Psi')^{2}-\Psi\Psi''$. Using \prettyref{eq:masterstreamfunction},
we see that $\Theta=c_{0}^{2}\left(\cos^{2}+\sin^{2}\right)=c_{0}^{2}$
so that
\begin{equation}
J_{\text{self}}=c_{0}^{2}\sum_{i=1}^{n}\left(\chi_{i}^{2}\right)'.\label{eq:self-interactions}
\end{equation}
In particular, we find that $J_{\text{self}}$ is constant in the periodic
variable $x$, so that the Fourier coefficient $[J_{\text{self}}]_{k}^{\wedge}$
vanishes identically except for when $k=0$. For $k=0$, note that
\[
[J_{\text{self}}]_{0}^{\wedge}=J_{\text{self}}=\overline{J}=\frac{1}{2}\frac{d}{dz}\overline{(\partial_{x}\psi)^{2}}.
\]

Continuing, we see that the general term in the second sum, $J_{\text{nbr}}$,
satisfies
\begin{align*}
 & \nabla^{\perp}(\chi_{j}\psi_{j})\cdot\nabla\partial_{x}(\chi_{j+1}\psi_{j+1})+\nabla^{\perp}(\chi_{j+1}\psi_{j+1})\cdot\nabla\partial_{x}(\chi_{j}\psi_{j})\\
 & =-(\chi_{j}'\psi_{j},\chi_{j}\psi_{j}')\cdot(\chi_{j+1}\psi_{j+1}'',\chi_{j+1}'\psi_{j+1}')-(\chi_{j+1}'\psi_{j+1},\chi_{j+1}\psi_{j+1}')\cdot(\chi_{j}\psi_{j}'',\chi_{j}'\psi_{j}')\\
 & =-\chi_{j}'\psi_{j}\chi_{j+1}\psi_{j+1}''+\chi_{j}\psi_{j}'\chi_{j+1}'\psi_{j+1}'-\chi_{j+1}'\psi_{j+1}\chi_{j}\psi_{j}''+\chi_{j+1}\psi_{j+1}'\chi_{j}'\psi_{j}'\\
 & =\chi_{j}\chi_{j+1}'\Theta_{j,j+1}+\chi_{j}'\chi_{j+1}\Theta_{j+1,j}
\end{align*}
where
\begin{align*}
\Theta_{j,j+1} & =\psi_{j}'\psi_{j+1}'-\psi_{j}''\psi_{j+1}\\
\Theta_{j+1,j} & =\psi_{j+1}'\psi_{j}'-\psi_{j+1}''\psi_{j}.
\end{align*}
Given \prettyref{eq:masterstreamfunction}, we find that
\begin{align*}
\Theta_{j,j+1} & =\sin(\frac{x}{l_{j}})\sin(\frac{x}{l_{j+1}})+\frac{l_{j+1}}{l_{j}}\cos(\frac{x}{l_{j}})\cos(\frac{x}{l_{j+1}})\\
\Theta_{j+1,j} & =\sin(\frac{x}{l_{j}})\sin(\frac{x}{l_{j+1}})+\frac{l_{j}}{l_{j+1}}\cos(\frac{x}{l_{j}})\cos(\frac{x}{l_{j+1}}).
\end{align*}
Applying standard trigonometric identities, 
\begin{align*}
\Theta_{j,j+1} & =\frac{1}{2}(1+\frac{l_{j+1}}{l_{j}})\cos(k_{j}^{\text{diff}}x)+\frac{1}{2}(\frac{l_{j+1}}{l_{j}}-1)\cos(k_{j}^{\text{sum}}x)\\
\Theta_{j+1,j} & =\frac{1}{2}(1+\frac{l_{j}}{l_{j+1}})\cos(k_{j}^{\text{diff}}x)+\frac{1}{2}(\frac{l_{j}}{l_{j+1}}-1)\cos(k_{j}^{\text{sum}}x)
\end{align*}
where $k_{j}^{\text{diff}}$ and $k_{j}^{\text{sum}}$ are as in the statement of
the result. In sum, 
\begin{equation}
J_{\text{nbr}}=\sum_{j=1}^{n}\chi_{j}\chi_{j+1}'\Theta_{j,j+1}+\sum_{j=1}^{n}\chi_{j}'\chi_{j+1}\Theta_{j+1,j}.\label{eq:trans-interactions}
\end{equation}

From \prettyref{eq:self-interactions}, \prettyref{eq:trans-interactions},
and the decomposition $J=J_{\text{self}}+J_{\text{nbr}}$ it is clear which wave-numbers
are present in $\mathcal{Q}$. We see that $\hat{J}_{k}$ is not identically
zero if and only if 
\[
k\in\{0\}\cup\{k_{j}^{\text{diff}}:1\leq j\leq n-1\}\cup\{k_{j}^{\text{sum}}:1\leq j\leq n-1\}.
\]
Since $l_{j+1}\neq l_{j}$, we see that 
\[
0\notin\{k_{j}^{\text{diff}}:1\leq j\leq n-1\}\cup\{k_{j}^{\text{sum}}:1\leq j\leq n-1\}.
\]
For general choices of lengthscales $\{l_{j}\}$ these two sets of
wavenumbers may intersect; however, given our special choices of lengthscales
in \prettyref{eq:l_relations}, we find that
\begin{equation}
\{k_{j}^{\text{diff}}:1\leq j\leq n-1\}\cap\{k_{j}^{\text{sum}}:1\leq j\leq n-1\}=\emptyset.\label{eq:nointerference}
\end{equation}
(If $k_{i}^{\text{sum}}=k_{j}^{\text{diff}}$ then $2^{i+1}+2^{i}=2^{j+1}-2^{j}$
from which the contradiction $3\cdot2^{i}=2^{j}$ follows.) Therefore,
\[
\mathcal{Q}(\mathbf{u}\cdot\nabla\xi)=\sum_{k\neq0}\mathcal{Q}_{k}(\mathbf{u}\cdot\nabla\xi)=\sum_{j=1}^{n-1}\mathcal{Q}_{k_{j}^{\text{sum}}}(\mathbf{u}\cdot\nabla\xi)+\mathcal{Q}_{k_{j}^{\text{diff}}}(\mathbf{u}\cdot\nabla\xi).
\]
\end{proof}
Now we estimate each of the non-zero contributions to the advection
term picked out by \prettyref{lem:diagonal_interactions}.
\begin{lem}
The branching construction satisfies 
\[
\fint_{I_{z}}|\overline{w\xi}-1|^{2}\,dz\lesssim\delta_{n}
\]
and
\[
\mathcal{Q}_{k_{j}^{\emph{sum}}}(\mathbf{u}\cdot\nabla\xi)\vee\mathcal{Q}_{k_{j}^{\emph{diff}}}(\mathbf{u}\cdot\nabla\xi)\lesssim\frac{l_{j}^{2}}{\delta_{j}}\quad1\leq j\leq n-1.
\]
\end{lem}
\begin{proof}
We begin with the $k=0$ term. Since 
\[
\overline{w\xi}=\sum_{j=1}^{n}\chi_{j}^{2}
\]
we find that
\[
\int_{I_{z}}|\overline{w\xi}-1|^{2}\,dz\lesssim\delta_{n}.
\]

Next we wish to estimate $\mathcal{Q}_{k_{j}^{\text{sum}}}$ and $\mathcal{Q}_{k_{j}^{\text{diff}}}$.
Recall that
\begin{align*}
\mathcal{Q}_{k_{j}^{\text{sum}}} & =\int_{I_{z}\times I_{z}'}G_{k_{j}^{\text{sum}}}(z,z')\hat{J}_{k_{j}^{\text{sum}}}(z)\hat{J}_{k_{j}^{\text{sum}}}^{*}(z')\,dzdz'\\
\mathcal{Q}_{k_{j}^{\text{diff}}} & =\int_{I_{z}\times I_{z}'}G_{k_{j}^{\text{diff}}}(z,z')\hat{J}_{k_{j}^{\text{diff}}}(z)\hat{J}_{k_{j}^{\text{diff}}}^{*}(z')\,dzdz'.
\end{align*}
By \prettyref{eq:self-interactions} and \prettyref{eq:trans-interactions},
\[
\hat{J}_{k}=[J_{\text{nbr}}]_{k}^{\wedge}=\sum_{j=1}^{n}\chi_{j}\chi_{j+1}^{'}\widehat{\Theta_{j,j+1}}(k)+\chi_{j}^{'}\chi_{j+1}\widehat{\Theta_{j+1,j}}(k)
\]
for $k\neq0$. It follows from \prettyref{eq:nointerference} that
\begin{align*}
\hat{J}_{k_{j}^{\text{sum}}} & =\frac{1}{2}(\frac{l_{j+1}}{l_{j}}-1)\chi_{j}\chi_{j+1}'+\frac{1}{2}(\frac{l_{j}}{l_{j+1}}-1)\chi_{j}'\chi_{j+1}\\
\hat{J}_{k_{j}^{\text{diff}}} & =\frac{1}{2}(1+\frac{l_{j+1}}{l_{j}})\chi_{j}\chi_{j+1}'+\frac{1}{2}(1+\frac{l_{j}}{l_{j+1}})\chi_{j}'\chi_{j+1}.
\end{align*}
In particular, 
\[
\text{supp}\,\hat{J}_{k_{j}^{\text{sum}}}\cup\text{supp}\,\hat{J}_{k_{j}^{\text{sum}}}\subset\text{supp}\,\chi_{j}\chi_{j+1}'\cup\text{supp}\,\chi_{j}'\chi_{j+1} \subset[z_{j},z_{j+1}]\cup[1-z_{j+1},1-z_{j}]=I_{j}
\]
by \prettyref{eq:cross_support}. Thus,
\begin{align*}
\mathcal{Q}_{k_{j}^{\text{sum}}} & =\int_{I_{j}\times I_{j}}G_{k_{j}^{\text{sum}}}(z,z')\hat{J}_{k_{j}^{\text{sum}}}(z)\hat{J}_{k_{j}^{\text{sum}}}^{*}(z')\,dzdz'\\
\mathcal{Q}_{k_{i}^{\text{diff}}} & =\int_{I_{j}\times I_{j}}G_{k_{j}^{\text{diff}}}(z,z')\hat{J}_{k_{j}^{\text{diff}}}(z)\hat{J}_{k_{j}^{\text{diff}}}(z')\,dzdz'.
\end{align*}

Now we estimate these quadratic forms. By H\"older's inequality,
\[
\mathcal{Q}_{k_{j}^{\text{sum}}}\leq||G_{k_{j}^{\text{sum}}}||_{L^{1}(I_{j}\times I_{j})}||\hat{J}_{k_{j}^{\text{sum}}}||_{L^{\infty}}^{2}\quad\text{and}\quad\mathcal{Q}_{k_{j}^{\text{diff}}}\leq||G_{k_{j}^{\text{diff}}}||_{L^{1}(I_{j}\times I_{j})}||\hat{J}_{k_{j}^{\text{diff}}}||_{L^{\infty}}^{2}.
\]
Observe that 
\[
||\chi_{j}\chi_{j+1}'||_{L^{\infty}}\vee||\chi_{j}'\chi_{j+1}||_{L^{\infty}}\lesssim\frac{1}{\delta_{j}}
\]
as a result of \prettyref{eq:assumption2-0}. Combining this with
the first part of \prettyref{lem:Gestimates} applied with $A=I_{j}$,
we find that
\[
\mathcal{Q}_{k_{j}^{\text{sum}}}\lesssim\frac{\delta_{j}}{k_{j}^{\text{sum}}}(\delta_{j}\wedge\frac{1}{k_{j}^{\text{sum}}})\cdot(\frac{l_{j}}{l_{j+1}}\frac{1}{\delta_{j}})^{2}\quad\text{and}\quad\mathcal{Q}_{k_{j}^{\text{diff}}}\lesssim\frac{\delta_{j}}{k_{j}^{\text{diff}}}(\delta_{j}\wedge\frac{1}{k_{j}^{\text{diff}}})\cdot(\frac{l_{j}}{l_{j+1}}\frac{1}{\delta_{j}})^{2}.
\]
It follows from \prettyref{eq:l_relations} that 
\[
\frac{1}{l_{j}}\lesssim k_{j}^{\text{sum}}\wedge k_{j}^{\text{diff}},
\]
so we can simplify these estimates to 
\[
\mathcal{Q}_{k_{j}^{\text{sum}}}\vee\mathcal{Q}_{k_{j}^{\text{diff}}}\lesssim\delta_{j}l_{j}(\delta_{j}\wedge l_{j})\cdot(\frac{l_{j}}{l_{j+1}}\frac{1}{\delta_{j}})^{2}.
\]
Using the first part of \prettyref{eq:assumption2-1} followed by
\prettyref{eq:l_relations}, we conclude that
\[
\mathcal{Q}_{k_{j}^{\text{sum}}}\vee\mathcal{Q}_{k_{j}^{\text{diff}}}\lesssim\frac{l_{j}^{4}}{\delta_{j}l_{j+1}^{2}}\lesssim\frac{l_{j}^{2}}{\delta_{j}}.
\]
This completes the proof.
\end{proof}
We turn to estimate the higher order terms from \prettyref{eq:optimaldesign_w2wenstr}.
\begin{lem}
The branching construction satisfies
\[
\fint_{\Omega}|\nabla\mathbf{u}|^{2}\vee\fint_{\Omega}|\nabla\xi|^{2}\lesssim\frac{1}{l_{1}^{2}}+\sum_{j=1}^{n-1}\frac{\delta_{j}}{l_{j+1}^{2}}+\frac{l_{n}^{2}}{\delta_{n}^{3}}.
\]
\end{lem}
\begin{proof}
Note that 
\[
\partial_{xx}\psi=\sum_{j=1}^{n}\chi_{j}\psi_{j}'',\quad\partial_{xz}\psi=\sum_{j=1}^{n}\chi_{j}'\psi_{j}',\ \text{and}\quad\partial_{zz}\psi=\sum_{j=1}^{n}\chi_{j}''\psi_{j}.
\]
Therefore, by orthogonality,
\[
\fint_{\Omega}|\nabla\nabla\psi|^{2}=\sum_{j=1}^{n}\fint_{\Omega}|\chi_{j}\psi_{j}''|^{2}+2|\chi_{j}'\psi_{j}'|^{2}+|\chi_{j}''\psi_{j}|^{2}.
\]
For $j\neq1$, we see from an application of \prettyref{eq:assumption2-0}
that
\[
\fint_{\Omega}|\chi_{j}\psi_{j}''|^{2}=||\chi_{j}||_{L^{2}(I_{z})}^{2}\overline{(\psi_{j}'')^{2}}\lesssim\frac{\delta_{j}+\delta_{j-1}}{l_{j}^{2}}\lesssim\frac{\delta_{j-1}}{l_{j}^{2}}.
\]
For $j=1$, we have instead that
\[
\fint_{\Omega}|\chi_{1}\psi_{1}''|^{2}=||\chi_{1}||_{L^{2}(I_{z})}^{2}\overline{(\psi_{1}'')^{2}}\lesssim\frac{1}{l_{1}^{2}}.
\]

Similarly, we find that
\begin{align*}
\fint_{\Omega}|\chi_{j}'\psi_{j}'|^{2} & =||\chi_{j}'||_{L^{2}(I_{z})}^{2}\overline{(\psi_{j}')^{2}}\lesssim\frac{1}{\delta_{j-1}}+\frac{1}{\delta_{j}}\lesssim\frac{1}{\delta_{j}},\\
\fint_{\Omega}|\chi_{j}''\psi_{j}|^{2} & =||\chi_{j}''||_{L^{2}(I_{z})}^{2}\overline{(\psi_{j})^{2}}\lesssim(\frac{1}{\delta_{j-1}^{3}}+\frac{1}{\delta_{j}^{3}})l_{j}^{2}\lesssim\frac{l_{j}^{2}}{\delta_{j}^{3}}
\end{align*}
for all $j$. Therefore,
\[
\fint_{\Omega}|\nabla\nabla\psi|^{2}\lesssim\frac{1}{l_{1}^{2}}+\sum_{j=1}^{n-1}\left(\frac{\delta_{j}}{l_{j+1}^{2}}+\frac{1}{\delta_{j}}+\frac{l_{j}^{2}}{\delta_{j}^{3}}\right)+\frac{1}{\delta_{n}}+\frac{l_{n}^{2}}{\delta_{n}^{3}}.
\]
By \prettyref{eq:assumption2-1}, 
\[
\frac{\delta_{j}}{l_{j+1}^{2}}\gtrsim\frac{1}{\delta_{j}}\vee\frac{l_{j}^{2}}{\delta_{j}^{3}}\quad\text{and}\quad\frac{l_{n}^{2}}{\delta_{n}^{3}}\gtrsim\frac{1}{\delta_{n}}.
\]
 The result follows.
\end{proof}
We now assemble the previous estimates. Let
\[
E(\epsilon;\{z_{k}\},\{l_{k}\})=\fint_{\Omega}|\nabla\Delta^{-1}\text{div}(\mathbf{u}\xi)|^{2}+\epsilon\fint_{\Omega}|\nabla\mathbf{u}|^{2}\cdot\fint_{\Omega}|\nabla\xi|^{2}
\]
where $(\mathbf{u},\xi)$ are constructed from $\{z_{k}\}_{k=1}^{n}$
and $\{l_{k}\}_{k=1}^{n}$ as described in \prettyref{sec:branchingconstruction}.
It will be convenient in what follows to think of estimating $E$
in terms of some smoothly interpolated version of these parameters. 
\begin{cor}
\label{cor:branchingestimate-final}Let $\ell(z)$ be any smooth,
monotonic function defined on $[z_{\emph{bulk}},z_{\emph{bl}}]$ that satisfies
\[
\ell(z_{k})=l_{k}\quad1\leq k\leq n.
\]
Then, the branching construction corresponding to $\{z_{k}\}_{k=1}^{n}$
and $\{l_{k}\}_{k=1}^{n}$ satisfies 
\[
E(\epsilon;\{z_{k}\},\{l_{k}\})\lesssim l_{\emph{bl}}+\int_{z_{\emph{bulk}}}^{z_{\emph{bl}}}(\ell'(z))^{2}\,dz+\epsilon\left(\frac{1}{l_{\emph{bulk}}^{2}}+\int_{z_{\emph{bulk}}}^{z_{\emph{bl}}}\frac{1}{(\ell(z))^{2}}\,dz+\frac{1}{l_{\emph{bl}}}\right)^{2}.
\]
\end{cor}
\begin{proof}
Collecting the results above and using \prettyref{eq:l_relations},
we conclude that
\[
\fint_{\Omega}|\nabla\Delta^{-1}\text{div}(\mathbf{u}\xi)|^{2}\lesssim l_{\text{bl}}+\sum_{j=1}^{n-1}\frac{l_{j}^{2}}{\delta_{j}}\sim l_{\text{bl}}+\sum_{j=1}^{n-1}\left|\frac{l_{j+1}-l_{j}}{\delta_{j}}\right|^{2}\delta_{j}
\]
and

\[
\fint_{\Omega}|\nabla\mathbf{u}|^{2}\vee\fint_{\Omega}|\nabla\xi|^{2}\lesssim\frac{1}{l_{1}^{2}}+\sum_{j=1}^{n-1}\frac{\delta_{j}}{l_{j+1}^{2}}+\frac{l_{n}^{2}}{\delta_{n}^{3}}\sim\frac{\delta_{\text{bulk}}}{l_{\text{bulk}}^{2}}+\sum_{j=1}^{n-1}\frac{1}{l_{j}^{2}}\delta_{j}+\frac{1}{l_{\text{bl}}}.
\]
By Jensen's inequality and the definition of $\ell(z)$, 
\[
\left|\frac{l_{j+1}-l_{j}}{\delta_{j}}\right|^{2}=\left|\fint_{z_{j}}^{z_{j+1}}\ell'(z)\,dz\right|^{2}\leq\fint_{z_{j}}^{z_{j+1}}|\ell'|^{2}\,dz
\]
so that
\[
\fint_{\Omega}|\nabla\Delta^{-1}\text{div}(\mathbf{u}\xi)|^{2}\lesssim l_{\text{bl}}+\sum_{j=1}^{n-1}\int_{z_{j}}^{z_{j+1}}|\ell'|^{2}\,dz=l_{\text{bl}}+\int_{z_{\text{bulk}}}^{z_{\text{bl}}}(\ell')^{2}\,dz.
\]
Also, as $l_{j+1}\sim l_{j}$ by \prettyref{eq:l_relations}, 
\[
\frac{1}{l_{j}^{2}}\delta_{j}=\int_{z_{j}}^{z_{j+1}}\frac{1}{l_{j}^{2}}\,dz\sim\int_{z_{j}}^{z_{j+1}}\frac{1}{\ell^{2}}\,dz.
\]
Therefore, 
\[
\fint_{\Omega}|\nabla\mathbf{u}|^{2}\vee\fint_{\Omega}|\nabla\xi|^{2}\lesssim\frac{\delta_{\text{bulk}}}{l_{\text{bulk}}^{2}}+\sum_{j=1}^{n-1}\int_{z_{j}}^{z_{j+1}}\frac{1}{\ell^{2}}\,dz+\frac{1}{l_{\text{bl}}}=\frac{\delta_{\text{bulk}}}{l_{\text{bulk}}^{2}}+\int_{z_{\text{bulk}}}^{z_{\text{bl}}}\frac{1}{\ell^{2}}\,dz+\frac{1}{l_{\text{bl}}}.
\]
\end{proof}

\subsection{Proof of \prettyref{prop:lowerbound_enstrconstr_w2w}}

The result of the previous analysis is that the branching construction
from \prettyref{sec:branchingconstruction} satisfies the efficiency
estimate
\[
E(\epsilon;\{z_{k}\},\{l_{k}\})\lesssim l_{\text{bl}}+\int_{z_{\text{bulk}}}^{z_{\text{bl}}}(\ell')^{2}\,dz+\epsilon\left(\frac{1}{l_{\text{bulk}}^{2}}+\int_{z_{\text{bulk}}}^{z_{\text{bl}}}\frac{1}{(\ell)^{2}}\,dz+\frac{1}{l_{\text{bl}}}\right)^{2}
\]
where $\ell(z)$ is obtained from $\{l_{k}\}_{k=1}^{n}$ by smooth
and monotonic interpolation. Now to prove \prettyref{prop:lowerbound_enstrconstr_w2w},
we will optimize the righthand side in the free parameters $\ell(z)$,
$l_{\text{bulk}}$, and $l_{\text{bl}}$, and then back out admissible choices of
$\{z_{k}\}_{k=1}^{n}$ and $\{l_{k}\}_{k=1}^{n}$ from the result.
To ensure that the requirements from \prettyref{sec:branchingconstruction}-\prettyref{sec:Admissibility}
hold, we must carry out this optimization under the constraint that
\begin{equation}
0\leq\ell'(z)\lesssim1\quad z\in[z_{\text{bulk}},z_{\text{bl}}].\label{eq:monotonicity_of_l}
\end{equation}
That the minimizer of 
\begin{equation}
\min_{\substack{\ell(z)\\
l(z_{\text{bulk}})=l_{\text{bulk}}\\
\ell(z_{\text{bl}})=l_{\text{bl}}
}
}\,l_{\text{bl}}+\int_{z_{\text{bulk}}}^{z_{\text{bl}}}(\ell')^{2}\,dz+\epsilon\left(\frac{1}{l_{\text{bulk}}^{2}}+\int_{z_{\text{bulk}}}^{z_{\text{bl}}}\frac{1}{(\ell)^{2}}\,dz+\frac{1}{l_{\text{bl}}}\right)^{2}\label{eq:1d_lengthscale_minpblm}
\end{equation}
satisfies \prettyref{eq:monotonicity_of_l} will be verified later
on. 

First, let us determine the optimal form of $\ell(z)$. We consider
that $\epsilon\ll1$ throughout this preliminary discussion, which
should serve to motivate the choices made in the formal proof that
follows. Consider the contributions to \prettyref{eq:1d_lengthscale_minpblm}
coming from the transition layers where $z\in[z_{\text{bulk}},z_{\text{bl}}]$.
We can identify the scaling of their minimum value by balancing the
corresponding integrands. This yields
\begin{equation}
\ell'(z)\sim\epsilon^{1/2}\left(\int_{z_{\text{bulk}}}^{z_{\text{bl}}}\frac{1}{\ell^{2}}\,dz\right)^{1/2}\frac{1}{\ell(z)}.\label{eq:balance}
\end{equation}
It is natural to impose the boundary condition $\ell(1)=0$ to determine
$\ell$. We find that 
\[
\ell(z)\sim c(\epsilon)(1-z)^{1/2}
\]
where $c(\epsilon)$ must be determined by substitution into \prettyref{eq:balance}.
Thus,
\begin{align*}
c(\epsilon) & \sim\epsilon^{1/6}\left(\int_{z_{\text{bulk}}}^{z_{\text{bl}}}\frac{1}{1-z}\,dz\right)^{1/6}=\epsilon^{1/6}\log^{1/6}\left(\frac{1-z_{\text{bulk}}}{1-z_{\text{bl}}}\right).
\end{align*}
Since
\[
l_{\text{bl}}\sim c(\epsilon)(1-z_{\text{bl}})^{1/2}\quad\text{and}\quad l_{\text{bulk}}\sim c(\epsilon)(1-z_{\text{bulk}})^{1/2}
\]
we conclude that
\[
\frac{1-z_{\text{bulk}}}{1-z_{\text{bl}}}\sim\frac{l_{\text{bulk}}}{l_{\text{bl}}}.
\]
Anticipating that $l_{\text{bl}}\ll l_{\text{bulk}}$ for $\epsilon\ll1$, we conclude
that the optimal form of the smooth lengthscale function $\ell(z)$
is given by
\begin{equation}
\ell(z)\sim\epsilon^{1/6}\log^{1/6}\left(\frac{l_{\text{bulk}}}{l_{\text{bl}}}\right)\,(1-z)^{1/2},\quad z\in[z_{\text{bulk}},z_{\text{bl}}].\label{eq:powerlaw_for_l}
\end{equation}
Such an $\ell$ yields the estimates
\begin{align*}
\int_{z_{\text{bulk}}}^{z_{\text{bl}}}(\ell')^{2}\,dz & \sim\epsilon\left(\int_{z_{\text{bulk}}}^{z_{\text{bl}}}\frac{1}{\ell^{2}}\,dz\right)^{2}\sim\epsilon^{1/3}\log^{1/3}\left(\frac{l_{\text{bulk}}}{l_{\text{bl}}}\right)\int_{z_{\text{bulk}}}^{z_{\text{bl}}}\frac{1}{1-z}\,dz\\
 & =\epsilon^{1/3}\log^{1/3}\left(\frac{l_{\text{bulk}}}{l_{\text{bl}}}\right)\log\left(\frac{1-z_{\text{bulk}}}{1-z_{\text{bl}}}\right)\sim\epsilon^{1/3}\log^{4/3}\left(\frac{l_{\text{bulk}}}{l_{\text{bl}}}\right)
\end{align*}
for $\epsilon\ll1$. 

Next, we determine the optimal choices for $l_{\text{bulk}}$ and $l_{\text{bl}}$
in this asymptotic regime. Plugging \prettyref{eq:powerlaw_for_l}
back into \prettyref{eq:1d_lengthscale_minpblm} yields the resulting
minimization
\[
\min_{\substack{l(z_{\text{bulk}})=l_{\text{bulk}}\\
\ell(z_{\text{bl}})=l_{\text{bl}}
}
}\,l_{\text{bl}}+\epsilon^{1/3}\log^{4/3}\left(\frac{l_{\text{bulk}}}{l_{\text{bl}}}\right)+\epsilon\left(\frac{1}{l_{\text{bulk}}^{2}}+\frac{1}{l_{\text{bl}}}\right)^{2}.
\]
Critical point tests yield the optimal scalings
\begin{equation}
l_{\text{bulk}}\sim\epsilon^{1/6}\log^{1/6}\frac{1}{\epsilon}\quad\text{and}\quad l_{\text{bl}}\sim\epsilon^{1/3}\log^{1/3}\frac{1}{\epsilon}\label{eq:bl_bulk_conditions_for_l}
\end{equation}
for $\epsilon\ll1$. Note this is consistent with the hypothesis that
$l_{\text{bl}}\ll l_{\text{bulk}}$ in this regime. To summarize, the smooth lengthscale
function $\ell(z)$ picked out by our analysis of \prettyref{eq:1d_lengthscale_minpblm}
scales as
\begin{equation}
\ell(z)\sim\epsilon^{1/6}\log^{1/6}\frac{1}{\epsilon}\,(1-z)^{1/2},\quad z\in[z_{\text{bulk}},z_{\text{bl}}],\label{eq:powerlaw_for_l-1}
\end{equation}
where
\[
1-z_{\text{bulk}}\sim1\quad\text{and}\quad1-z_{\text{bl}}\sim\epsilon^{1/3}\log^{1/3}\frac{1}{\epsilon}.
\]

We are now ready to prove the upper bound from \prettyref{prop:lowerbound_enstrconstr_w2w}.

\begin{proof}[Proof of \prettyref{prop:lowerbound_enstrconstr_w2w}]
Our plan is to verify the existence of a branching construction, as
described in \prettyref{sec:branchingconstruction}, whose parameters
$\{z_{k}\}_{k=1}^{n}$ and $\{l_{k}\}_{k=1}^{n}$ are consistent with
the optimal smooth lengthscale function $\ell(z)$ from \prettyref{eq:powerlaw_for_l-1}.
Once we verify the requirements of \prettyref{sec:branchingconstruction}-\prettyref{sec:Admissibility}
hold, the desired bound $E(\epsilon;\{z_{k}\},\{l_{k}\})\lesssim\epsilon^{1/3}\log^{4/3}\frac{1}{\epsilon}$
follows as above. For the reader's convenience, we recall the requirements
that must be checked: these are \prettyref{eq:zk_inequalities}, \prettyref{eq:lk_inequalities},
and \prettyref{eq:assumption2-0}-\prettyref{eq:l_relations}.

We start by defining 
\begin{equation}
\ell(z)=\epsilon^{1/6}\log^{1/6}\frac{1}{\epsilon}\,(1-z)^{1/2}\quad z\in[\frac{1}{2},1]\label{eq:eqn_for_cts_l}
\end{equation}
in obvious analogy with \prettyref{eq:powerlaw_for_l-1}. To choose
the horizontal lengthscales $\{l_{k}\}_{k=1}^{n}$, we set 
\[
l_{\text{bulk}}=\frac{l_{x}}{2\pi}\frac{1}{k_{\text{bulk}}}
\]
where $k_{\text{bulk}}\in\mathbb{N}$ satisfies
\begin{equation}
k_{\text{bulk}}-1<\frac{l_{x}}{\pi}\frac{1}{\epsilon^{1/6}\log^{1/6}\frac{1}{\epsilon}}\leq k_{\text{bulk}},\label{eq:bulk_wavenumber}
\end{equation}
and take
\[
l_{k}=\frac{l_{\text{bulk}}}{2^{k-1}},\quad k=1,\dots,n.
\]
Note to ensure $k_{\text{bulk}}\geq1$ we must require that $\epsilon^{1/6}\log^{1/6}\frac{1}{\epsilon}\lesssim l_{x}$.
This condition is given in the statement of \prettyref{prop:lowerbound_enstrconstr_w2w}.
Note also that \prettyref{eq:lk_inequalities} and \prettyref{eq:l_relations}
hold. 

Now as \prettyref{eq:eqn_for_cts_l} is strictly decreasing, we may
define the points $\{z_{k}\}_{k=1}^{n}$ by 
\begin{equation}
\ell(z_{k})=l_{k},\quad k=1,\dots,n.\label{eq:zk_from_lk}
\end{equation}
This gives
\[
z_{k}=1-c_{1}\frac{1}{2^{2(k-1)}},\quad k=1,\dots,n
\]
where
\[
c_{1}=(\frac{l_{\text{bulk}}}{\epsilon^{1/6}\log^{1/6}\frac{1}{\epsilon}})^{2}=\frac{1}{4\pi^{2}}\frac{1}{\epsilon^{1/3}\log^{1/3}\frac{1}{\epsilon}}\frac{l_{x}^{2}}{k_{\text{bulk}}^{2}}.
\]
By \prettyref{eq:bulk_wavenumber}, $c_{1}\leq\frac{1}{4}$ so that
$z_{1}=1-c_{1}\geq\frac{3}{4}$ as required by \prettyref{eq:assumption3}.
Note \prettyref{eq:zk_inequalities} and \prettyref{eq:assumption2-0}
are satisfied as well.

Finally, we fix $n\in\mathbb{N}$ by enforcing \prettyref{eq:bl_bulk_conditions_for_l},
which states here that
\[
\frac{l_{\text{bulk}}}{2^{n-1}}\sim\epsilon^{1/3}\log^{1/3}\frac{1}{\epsilon}.
\]
To achieve this, let us define $n\in\mathbb{N}$ via the inequalities
\[
n-1<\log_{2}\left(\frac{2\pi}{\epsilon^{1/3}\log^{1/3}\frac{1}{\epsilon}}\frac{1}{k_{\text{bulk}}}\right)\leq n.
\]

Having chosen $\{z_{k}\}_{k=1}^{n}$ and $\{l_{k}\}_{k=1}^{n}$, we
may invoke the definitions from \prettyref{sec:branchingconstruction}
to produce a branching construction $(\mathbf{u},\xi)$. Note we have
checked each requirement from \prettyref{sec:branchingconstruction}-\prettyref{sec:Admissibility}
except for \prettyref{eq:assumption2-1}. That $l_{n}\sim\delta_{n}$
follows from \prettyref{eq:bl_bulk_conditions_for_l} and \prettyref{eq:powerlaw_for_l-1}.
Now we show that $l_{k}\lesssim\delta_{k}$ for all $k$. Since $\delta_{k}=z_{k+1}-z_{k}$
and $l_{k}\sim|l_{k+1}-l_{k}|$, this requires showing that
\[
1\lesssim|\frac{z_{k+1}-z_{k}}{l_{k+1}-l_{k}}|
\]
for all $k.$ Noting $z'(\ell)<0$, we only need to show that
\[
1\lesssim|z'(\ell)|\quad\ell\in[l_{\text{bl}},l_{\text{bulk}}].
\]
Differentiating \prettyref{eq:eqn_for_cts_l} implicitly, we find
that
\[
|z'(\ell)|=2\frac{(1-z)^{1/2}}{\epsilon^{1/6}\log^{1/6}\frac{1}{\epsilon}}\gtrsim\frac{\delta_{\text{bl}}^{1/2}}{\epsilon^{1/6}\log^{1/6}\frac{1}{\epsilon}}\sim\frac{l_{\text{bl}}^{1/2}}{\epsilon^{1/6}\log^{1/6}\frac{1}{\epsilon}}\sim1
\]
as required.

In sum, we have produced a branching construction $(\mathbf{u},\xi)$
consistent with the requirements of \prettyref{sec:branchingconstruction}-\prettyref{sec:Admissibility}
whose parameters $\{z_{k}\}_{k=1}^{n}$ and $\{l_{k}\}_{k=1}^{n}$
interpolate the desired smooth lengthscale function $\ell(z)$ from
\prettyref{eq:powerlaw_for_l-1}. The estimates proved in \prettyref{sec:Energy-estimates}
apply, and we may immediately conclude from \prettyref{cor:branchingestimate-final}
and the discussion surrounding \prettyref{eq:bl_bulk_conditions_for_l}
and \prettyref{eq:powerlaw_for_l-1} that 
\begin{align*}
E(\epsilon;\{z_{k}\},&\{l_{k}\})  =\fint_{\Omega}|\nabla\Delta^{-1}\text{div}(\mathbf{u}\xi)|^{2}+\epsilon\fint_{\Omega}|\nabla\mathbf{u}|^{2}\cdot\fint_{\Omega}|\nabla\xi|^{2}\\
 & \lesssim l_{\text{bl}}+\int_{z_{\text{bulk}}}^{z_{\text{bl}}}(\ell')^{2}\,dz+\epsilon\left(\frac{1}{l_{\text{bulk}}^{2}}+\int_{z_{\text{bulk}}}^{z_{\text{bl}}}\frac{1}{(\ell)^{2}}\,dz+\frac{1}{l_{\text{bl}}}\right)^{2}\\
 & \lesssim\epsilon^{1/3}\log^{1/3}\frac{1}{\epsilon}+\epsilon^{1/3}\log^{4/3}\left(\frac{\epsilon^{1/6}\log^{1/6}\frac{1}{\epsilon}}{\epsilon^{1/3}\log^{1/3}\frac{1}{\epsilon}}\right)+\epsilon\left(\frac{2}{\epsilon^{1/3}\log^{1/3}\frac{1}{\epsilon}}\right)^{2}\\
 & \lesssim\epsilon^{1/3}\log^{1/3}\frac{1}{\epsilon}+\epsilon^{1/3}\log^{4/3}\frac{1}{\epsilon}+\frac{\epsilon^{1/3}}{\log^{2/3}\frac{1}{\epsilon}}\lesssim\epsilon^{1/3}\log^{4/3}\frac{1}{\epsilon}
\end{align*}
for $\epsilon\lesssim1$. Thus, \prettyref{prop:lowerbound_enstrconstr_w2w}
is proved. \end{proof}

\section{Implications for the analysis of turbulent heat transport\label{sec:Implications-for-RBC}}

There is a long history, originating in the works of Malkus \cite{malkus1954heat}
and Howard \cite{howard1963heat}, of variational methods for the
analysis of turbulent heat transport, the primary focus of which has been on absolute or \emph{a priori} upper bounds. 
Consider the usual setup
of Rayleigh-B\'enard convection (RBC), wherein an incompressible
fluid layer is heated from below and cooled from above, and is subjected
to a constant downwards-pointing gravitational force. The temperature
field $T(\mathbf{x},t)$ undergoes transport by means of advection-diffusion,
\begin{equation}
\partial_{t}T+\mathbf{u}\cdot\nabla T=\Delta T.\label{eq:adv-diff-eqn}
\end{equation}
The advecting velocity $\mathbf{u}(\mathbf{x},t)$ is coupled back
to temperature field $T$ through a suitable momentum equation. This
could be, for instance, Darcy's law as it is for convection in a fluid
saturated porous layer. Here, we are concerned with convection in
a fluid layer for which, in the Bousinessq approximation, \prettyref{eq:adv-diff-eqn}
is supplemented with the buoyancy forced incompressible Navier-Stokes
equations
\begin{equation}
\partial_{t}\mathbf{u}+\mathbf{u}\cdot\nabla\mathbf{u}+\nabla p=Pr\Delta\mathbf{u}+PrRa\hat{\mathbf{k}}T\label{eq:momentum-eqn}
\end{equation}
and
\begin{equation}
\text{div}\,\mathbf{u}=0.\label{eq:div-free-1}
\end{equation}
The two non-dimensional parameters are the Prandtl number $Pr$, the
ratio of the fluid's kinematic viscosity to its thermal diffusivity,
and the Rayleigh number $Ra$, a ratio of the intensities of driving
to damping forces which is proportional here to the bulk buoyancy
force across the layer. Altogether, \prettyref{eq:adv-diff-eqn}-\prettyref{eq:div-free-1}
constitute the equations of Rayleigh-B\'enard convection in a fluid
layer \cite{Rayleigh1916}. For boundary conditions we continue to
assume that the temperature field is imposed at the top and bottom
of the layer by
\[
T|_{z=1}=0\quad\text{and}\quad T|_{z=0}=1,
\]
while the velocity field is taken to satisfy either the no-slip boundary conditions
\[
\mathbf{u}|_{\partial\Omega}=\mathbf{0}
\]
or the stress-free boundary conditions
\[
w|_{\partial\Omega}=0\quad\text{and}\quad\partial_{z}u|_{\partial\Omega}=\partial_{z}v|_{\partial\Omega}=0.
\]
All fields are assumed to be periodic in the $xy$-plane.

The rate of heat transport in RBC can be measured by the Nusselt number
$Nu$, which evidently depends on $Pr$ and $Ra$ in some unknown
and complicated way. (It can also depend on the initial data, as well
as on the aspect ratios of the fluid layer.) Determining this relationship
and/or establishing absolute bounds on it continues to be the subject
of numerous works across the physical and mathematical literatures.
To date, the best known upper bound holding uniformly in $Pr$ and for no-slip
velocity boundary conditions states that
\begin{equation}
Nu\lesssim Ra^{1/2}\label{eq:Ra1/2bd}
\end{equation}
for $Ra\gg1$ \cite{doering1996variational,howard1963heat,seis2015}.
This bound also holds for stress-free velocity boundary conditions
in the three-dimensional layer $\Omega=\mathbb{T}_{xy}^{2}\times I_{z}$,
but more is known in the two-dimensional case where $\Omega=\mathbb{T}_{x}\times I_{z}$:
in two dimensions with stress-free boundary conditions, one has that
$Nu\lesssim Ra^{5/12}$ uniformly in $Pr$ for $Ra\gg1$ \cite{whitehead2011ultimate}.
(In the formal limit where $Pr=\infty$ and \prettyref{eq:momentum-eqn}
is replaced with Stoke's equation, the situation is quite different
\cite{doering2001upper,doering2006bounds,nobili2017limitations,otto2011rayleigh,yan2004limits}.)
There is little to no evidence, however, that any of these finite
$Pr$ bounds are in fact sharp, i.e., that there exist solutions of the equations of motion
\prettyref{eq:adv-diff-eqn}-\prettyref{eq:div-free-1} satisfying
$Nu\sim Ra^{1/2}$ as $Ra\to\infty$ (or $Nu\sim Ra^{5/12}$ for stress-free
boundaries in two dimensions).

In light of all this, we note that the main fluid dynamical contribution of this paper is a proof that when the momentum equation \prettyref{eq:momentum-eqn} is replaced by the enstrophy-constraint
\begin{equation}
\left\langle |\nabla\mathbf{u}|^{2}\right\rangle =Ra(Nu-1) \label{eq:momentum-balance}
\end{equation}
which it implies, the upper bound $Nu\lesssim  Ra^{1/2}$ becomes asymptotically sharp up to logarithmic corrections.  That is, for all large enough $Ra$ there exist velocity and temperature fields satisfying \prettyref{eq:adv-diff-eqn}, \prettyref{eq:div-free-1}, and \prettyref{eq:momentum-balance} along with the requisite boundary conditions such that
\begin{equation}
\frac{Ra^{1/2}}{\log^{2}Ra}\lesssim Nu \lesssim Ra^{1/2}. \label{eq:log_Nu_bds}
\end{equation}
This follows from \prettyref{thm:mainbounds_w2w_ens} upon taking $Pe^2 = Ra(Nu-1)$. Therefore, either the well-known bound \prettyref{eq:Ra1/2bd} on RBC is asymptotically sharp as $Ra\to \infty$ and $Pr$ is fixed, or details from the momentum equation \prettyref{eq:momentum-eqn} beyond the balance \prettyref{eq:momentum-balance} are essential for determining the scaling law of maximal turbulent heat transport.

The remainder of this section places our analysis of wall-to-wall optimal transport into its proper fluid dynamical context. 
To keep the discussion at a reasonable length, we do not attempt to summarize the vast literature on the
subject but instead focus on two of the most well-known methods for
proving \emph{a priori} bounds on transport: the variational approach
of Howard, and the background method of Constantin-Doering. Our plan is to recall just enough about these methods to allow for comparison with the techniques developed in this paper. 
For Howard's approach see \prettyref{sec:Howardapproach}, while for the background method see \prettyref{sec:onbackgroundmethod}. \prettyref{sec:momentumeqn} concerns the role of the momentum equation.

Before we proceed, let us mention the existence of the recently developed ``auxiliary functional'' method for producing bounds on time-averaged quantities \cite{chernyshenko2014polynomial}.
While the background method may ultimately be derived by a particular
choice of auxiliary functional \textemdash the same is true for the
recently proposed method of Seis \cite{chernyshenko2017} \textemdash it
is not yet clear if there exists any auxiliary functional that yields
an improvement to scaling beyond $Nu\lesssim Ra^{1/2}$. Although
for ordinary differential equations the auxiliary functional method
always yields sharp bounds on long-time averages \cite{tobasco2018optimal},
it remains to be seen if such a situation holds for general PDEs.

\subsection{On the variational approach of Howard\label{sec:Howardapproach}}

\subsubsection{Howard's variational problem}

If RBC is to be taken as a predictive model for turbulent convection,
one naturally asks: which of its solutions are actually realizable by experiment? Setting aside dynamical stability
as a possible selection principle, Malkus introduced in \cite{malkus1954heat} the idea that perhaps amongst all possible solutions of the equations of motion, those that are realized maximize their heat transport overall. An operational approach to establishing upper bounds inspired by Malkus' idea is to search for a larger admissible set of velocity and temperature fields, which contains all solutions of RBC, amongst which the maximal transport can analytically
be determined. This is Howard's variational approach.

Following Howard \cite{howard1963heat}, we observe that if $\mathbf{u}$ and $T$ arise in RBC, they must satisfy two identities known as the ``power integrals''. To derive the first of these, dot the momentum equation
\prettyref{eq:momentum-eqn} into $\mathbf{u}$, integrate by parts
and average in space and time. Changing variables by $\theta=T-(1-z)$ yields
 the first of Howard's identities
\begin{equation}
Ra\left\langle w\theta\right\rangle =\left\langle |\nabla\mathbf{u}|^{2}\right\rangle .\label{eq:powerintegral1}
\end{equation}
(Note this is simply a restatement of \prettyref{eq:momentum-balance}
from above.) A similar manipulation involving the
temperature equation \prettyref{eq:adv-diff-eqn} yields the second
identity
\begin{equation}
\left\langle w\theta\right\rangle +\left\langle w\theta\right\rangle ^{2}-\left\langle |\overline{w\theta}|^{2}\right\rangle =\left\langle |\nabla\theta|^{2}\right\rangle .\label{eq:powerintegral2}
\end{equation}
Consider now the problem of maximizing $Nu$ amongst
all divergence-free vector fields $\mathbf{u}$ and scalar fields
$\theta$ that vanish at the walls and furthermore satisfy \prettyref{eq:powerintegral1}
and \prettyref{eq:powerintegral2}. Since the equations of motion
of RBC imply these constraints, the resulting maximum sets
an\emph{ }upper bound on $Nu$ for RBC. 

Setting aside matters of statistical stationarity \cite{howard1963heat},
one can give an equivalent formulation of the variational problem
described above which makes it tractable for analysis. Under certain
further assumptions on the solutions of RBC (the ``requirements of
homogeneity'' from \cite{howard1963heat}), Howard deduced that the minimization
\begin{equation}
\min_{\substack{\mathbf{u}(\mathbf{x}),\theta(\mathbf{x})\\
\mathbf{u}|_{\partial\Omega}=\mathbf{0},\theta|_{\partial\Omega}=0\\
\fint_{\Omega}w\theta=1
}
}\,\fint_{\Omega}|\overline{w\theta}-1|^{2}+\epsilon\fint_{\Omega}|\nabla\mathbf{u}|^{2}\fint_{\Omega}|\nabla\theta|^{2}\label{eq:Howard's_pblm}
\end{equation}
is equivalent to the maximization $\sup\,Nu$ described above, and
that its optimal value can be used to produce an \emph{a priori} bound
on RBC (the algebraic manipulations in the proof of this are like
those performed in \prettyref{sec:OptlDesignProb_w2w} in the derivation
of the integral formulation of steady wall-to-wall optimal transport). 

The minimization \prettyref{eq:Howard's_pblm} is known as Howard's
problem. It bears striking resemblance to our integral formulation
of steady enstrophy-constrained wall-to-wall transport
\begin{equation}
\min_{\substack{\mathbf{u}(\mathbf{x}),\xi(\mathbf{x})\\
\mathbf{u}|_{\partial\Omega}=\mathbf{0},\xi|_{\partial\Omega}=0\\
\fint_{\Omega}w\xi=1
}
}\,\fint_{\Omega}|\nabla\Delta^{-1}\text{div}(\mathbf{u}\xi)|^{2}+\epsilon\fint_{\Omega}|\nabla\mathbf{u}|^{2}\fint_{\Omega}|\nabla\xi|^{2},\label{eq:steady_optl_design_enstrophy}
\end{equation}
obtained in \prettyref{sec:OptlDesignProb_w2w}. To find the relationship
between \prettyref{eq:Howard's_pblm} and \prettyref{eq:steady_optl_design_enstrophy},
we apply \prettyref{lem:advection_Fourierests_w2w} along with the
net-flux constraint $\fint_{\Omega}w\xi=1$ and decompose the advection
term as
\[
\fint_{\Omega}|\nabla\Delta^{-1}\text{div}(\mathbf{u}\xi)|^{2}=\fint_{\Omega}|\overline{w\xi}-1|^{2}+\mathcal{Q}(\text{div}(\mathbf{u}\xi))
\]
where $\mathcal{Q}$ is the positive semi-definite quadratic form
defined in \prettyref{lem:advection_Fourierests_w2w}. 

This last equation reveals the precise distinction between Howard's problem \prettyref{eq:Howard's_pblm} and our integral formulation in \prettyref{eq:steady_optl_design_enstrophy}. Because $\mathcal{Q}$ is positive semi-definite, it is evident that the minimum in \prettyref{eq:Howard's_pblm} is not smaller than the minimum in \prettyref{eq:steady_optl_design_enstrophy}. As a result, Howard's upper bound on heat transport is not lower than ours. Though the improvement in scaling in our approach is limited by \prettyref{eq:log_Nu_bds} to at most a logarithmic correction, it remains to be seen whether such a correction holds as an absolute upper bound. We turn now to consider the difference between the optimizers of \prettyref{eq:Howard's_pblm} and \prettyref{eq:steady_optl_design_enstrophy}.  

\subsubsection{Busse's multi-$\alpha$ technique}

As shown by Howard and Busse \cite{busse1969howards,howard1963heat},
the optimal value of Howard's problem \prettyref{eq:Howard's_pblm}
scales as $\epsilon^{1/3}$ for $\epsilon\ll1$. Thus, Howard's approach
to bounds on RBC yields $Nu\lesssim Ra^{1/2}$ and no better. The
\emph{a priori }lower bound implicit in this result is due to Howard;
the upper bound was obtained by Busse as an application of his ``multi-$\alpha$''
technique, which seeks to produce asymptotically valid solutions of
the Euler-Lagrange equations of \prettyref{eq:Howard's_pblm} involving
multiple horizontal wave numbers. Busse's multi-$\alpha$ analysis turns out to share parallels with our construction of branching flows, which we would like to discuss now. 

We start by recalling Howard's lower bound:

\begin{equation}
\min_{\substack{\mathbf{u}(\mathbf{x}),\theta(\mathbf{x})\\
\mathbf{u}|_{\partial\Omega}=\mathbf{0},\theta|_{\partial\Omega}=0\\
\fint_{\Omega}w\theta=1
}
}\,\fint_{\Omega}|\overline{w\theta}-1|^{2}+\epsilon\fint_{\Omega}|\nabla\mathbf{u}|^{2}\fint_{\Omega}|\nabla\theta|^{2}\gtrsim\epsilon^{1/3}\label{eq:Howard_LB}
\end{equation}
for $\epsilon\ll1$. Let $(\mathbf{u},\theta)$ be admissible, which we can take to be smooth. Let
$\delta\in(0,\frac{1}{2})$ be such that 
\[
0\leq|\overline{w\theta}|\leq\frac{1}{2}\quad\text{for}\ z\in[0,\delta],\quad\text{and}\quad\overline{w\theta}(\delta)=\frac{1}{2}.
\]
(If there does not exist such a $\delta$, then $\fint_{\Omega}|\overline{w\theta}-1|^{2}\gtrsim1\gg\epsilon^{2/3}$.)
By its definition,
\[
\fint_{\Omega}|\overline{w\theta}-1|^{2}\geq\frac{1}{|I_{z}|}\int_{0}^{\delta}|\overline{w\theta}-1|^{2}\gtrsim\delta.
\]
\prettyref{lem:lingrowth_wtheta} states that
\[
\frac{1}{|\mathbb{T}_{xy}^{2}|}||\partial_{z}\theta||_{L^{2}(\Omega)}||\partial_{z}w||_{L^{2}(\Omega)}\gtrsim\frac{\left|\overline{w\theta}(z)\right|}{|z\wedge(1-z)|}\quad\forall\,z.
\]
Taking $z=\delta$ and squaring, we conclude that
\[
\fint_{\Omega}|\nabla\mathbf{u}|^{2}\fint_{\Omega}|\nabla\theta|^{2}\gtrsim\frac{1}{\delta^{2}}.
\]
Therefore, the optimal value in the lefthand side of \prettyref{eq:Howard_LB}
is bounded below by 
\[
\inf_{\delta\in(0,\frac{1}{2})}\left\{ \delta+\epsilon\frac{1}{\delta^{2}}\right\} \sim\epsilon^{1/3}
\]
for $\epsilon\ll1$, and \prettyref{eq:Howard_LB} is proved. 

Now we discuss Busse's upper bound: it asserts the existence
of admissible pairs $\{(\mathbf{u}_{\epsilon},\theta_{\epsilon})\}$
satisfying 
\begin{equation}
\fint_{\Omega}|\overline{w_{\epsilon}\theta_{\epsilon}}-1|^{2}+\epsilon\fint_{\Omega}|\nabla\mathbf{u}_{\epsilon}|^{2}\fint_{\Omega}|\nabla\theta_{\epsilon}|^{2}\lesssim\epsilon^{1/3}\label{eq:Busse_UB}
\end{equation}
for $\epsilon\ll1$. Busse's multi-$\alpha$ technique is analogous to our branching construction from \prettyref{sec:enstrophy_constrained_design}.
Arguing as in that section, we find that our branching construction
with lengthscale $\ell(z)$ satisfies the estimates
\[
\fint_{\Omega}|\overline{w\theta}-1|^{2}\lesssim l_{\text{bl}}\quad\text{and}\quad\fint_{\Omega}|\nabla\mathbf{u}|^{2}\fint_{\Omega}|\nabla\theta|^{2}\lesssim\left(\frac{1}{l_{\text{bulk}}^{2}}+\int_{z_{\text{bulk}}}^{z_{\text{bl}}}\frac{1}{\ell^{2}}\,dz+\frac{1}{l_{\text{bl}}}\right)^{2}
\]
so long as $0\leq\ell'(z)\lesssim1$. Since branching is admissible
for Howard's problem, we find its optimal value is bounded above by
\begin{equation}
\min_{\substack{\ell(z)\\
\ell(z_{\text{bulk}})=l_{\text{bulk}}\\
\ell(z_{\text{bl}})=l_{\text{bl}}\\
0\leq\ell'(z)\lesssim1
}
}\,l_{\text{bl}}+\epsilon\left(\frac{1}{l_{\text{bulk}}^{2}}+\int_{z_{\text{bulk}}}^{z_{\text{bl}}}\frac{1}{\ell^{2}}\,dz+\frac{1}{l_{\text{bl}}}\right)^{2}. \label{eq:1dpblm_Busse}
\end{equation}
Choosing
\begin{align*}
\ell(z) & \sim1-z\quad z\in[z_{\text{bulk}},z_{\text{bl}}],\\
l_{\text{bulk}} & \sim1,\ l_{\text{bl}}\sim\epsilon^{1/3}
\end{align*}
yields \prettyref{eq:Busse_UB}. Although Busse's construction
is usually described in terms of discrete wavenumbers $\{\alpha_{k}\}_{k=1}^{n}$
and points $\{z_{k}\}_{k=1}^{n}$, for $\epsilon\ll1$ these can be
seen to arise from interpolation of the continuous lengthscale $\ell(z)\sim1-z$,
similar to the presentation in \prettyref{sec:enstrophy_constrained_design}.

Coming back to wall-to-wall optimal transport, we can now discuss the difference between the optimizers of Howard's problem \prettyref{eq:Howard's_pblm} and our integral formulation in \prettyref{eq:steady_optl_design_enstrophy}. As the analysis in \prettyref{sec:enstrophy_constrained_design} indicates, adding $\mathcal{Q}$ to Howard's problem \prettyref{eq:Howard's_pblm} should change the preferred lengthscale for branching from Busse's linear law $\ell\sim1-z$ to our square root one $\ell\sim c(\epsilon)\sqrt{1-z}$. The estimates obtained there show that
\begin{equation}
\mathcal{Q}\sim\int_{z_{\text{bulk}}}^{z_{\text{bl}}}(\ell')^{2}dz\label{eq:role_of_Q}.
\end{equation}
Thus, the 1D problem \prettyref{eq:1dpblm_Busse} for selecting the lengthscale function $\ell$ turns into \prettyref{eq:1d_w2w} for wall-to-wall optimal transport. It remains to be seen whether the true optimizers of \prettyref{eq:steady_optl_design_enstrophy} exhibit branching with these preferred lengthscales. Presumably, developing such fine detailed knowledge of the minimizers would help resolve the question of logarithmic corrections to scaling.  

\subsection{On the background method\label{sec:onbackgroundmethod}}

\subsubsection{Background method for RBC}

In \cite{doering1996variational}, Constantin and one of the authors introduced
an alternate method to Howard's for establishing \emph{a priori} bounds
on RBC, which can be applied without any assumptions of statistical
stationarity or homogeneity. We recall the argument now, with the
goal of connecting it to the symmetrization method from
\prettyref{sec:aprioribds}. We follow the presentation in \cite{doering1995}.

Let $\mathbf{u}$ and $T$ arise from RBC and decompose the temperature
field into the sum of stationary ``background'' and fluctuating
parts,
\[
T(\mathbf{x},t)=\tau(z)+\theta(\mathbf{x},t)
\]
where $\tau(0)=1$ and $\tau(1)=0$. Then, 
\[
\frac{1}{2}\frac{d}{dt}\left(\fint_{\Omega}|\theta|^{2}+\frac{1}{PrRa}\fint_{\Omega}|\mathbf{u}|^{2}\right)+\frac{1}{2}\fint_{\Omega}|\nabla T|^{2}=\frac{1}{2}\int_{0}^{1}|\tau'|^{2}-H_{\tau}(\mathbf{u},\theta)
\]
where $H_{\tau}$ is the quadratic form
\[
H_{\tau}(\mathbf{u},\theta)=\fint_{\Omega}\frac{1}{Ra}|\nabla\mathbf{u}|^{2}+\frac{1}{2}|\nabla\theta|^{2}+w\theta(\tau'-1).
\]
Provided that $H_{\tau}\geq0$ for all divergence-free vector fields
$\mathbf{u}(\mathbf{x})$ and scalar fields $\theta(\mathbf{x})$
vanishing at $\partial\Omega$, we can drop the last term from the
dissipation equation and take a long-time average to find the inequality
\[
\left\langle |\nabla T|^{2}\right\rangle \leq\int_{0}^{1}|\tau'|^{2}.
\]
This proves the following variational bound:
\begin{equation}
Nu\leq\inf_{\substack{\tau(z)\\
\tau(0)=1,\tau(1)=0\\
H_{\tau}\geq0
}
}\,\int_{0}^{1}|\tau'|^{2}.\label{eq:backgroundmethod_RBC}
\end{equation}
Those background fields $\tau$ which satisfy $H_{\tau}\geq0$ are
known as \emph{spectrally stable}. 

As proved in \cite{doering1995}, there exist spectrally stable background
fields $\{\tau_{\delta}\}$ satisfying 
\[
\int_{0}^{1}|\tau_{\delta}'|^{2}\sim\frac{1}{\delta}
\]
for all $\delta\leq Ra^{-1/2}$. Minimizing the resulting bound $Nu\lesssim\frac{1}{\delta}$
over this range of $\delta$ proves that $Nu\lesssim Ra^{1/2}.$ We
note the remarkable similarity between the background fields constructed
in \cite{doering1995} and those constructed for the symmetrization
method in \prettyref{eq:testfunctions}, which we turn to discuss now.

\subsubsection{Background method for optimal transport}

As observed in \cite{SD}, one can obtain \emph{a priori} bounds on
optimal transport via a suitable modification of the background method.
Here, our goal is to show that the symmetrization method from \prettyref{sec:aprioribds},
when properly abstracted and optimized, yields an \emph{a priori} bound
on transport whose value is \emph{exactly the same} as that obtained
in \cite{SD}. This begs the question of whether better background fields might
be constructed to improve upon the scaling $Nu\lesssim Ra^{1/2}$ (albeit by at most a logarithmic amount). Numerical evidence looks to point in the opposite direction, as the optimal bounds found in \cite{plasting2003improved} scale $\sim Ra^{1/2}$. We are not aware of a proof demonstrating this at the present time.

The modified background method from \cite{SD} is as follows. Let
$T$ solve the advection-diffusion equation \prettyref{eq:adv-diff-eqn}.
Performing the background decomposition
\[
T(\mathbf{x},t)=\tau(z)+\theta(\mathbf{x},t)
\]
with $\tau(0)=1$ and $\tau(1)=0$ and introducing a Lagrange multiplier
$\lambda\in\mathbb{R}$, we find that
\begin{equation}
\frac{1}{2}\frac{d}{dt}\fint_{\Omega}|\theta|^{2}+\frac{1}{2}\fint_{\Omega}|\nabla T|^{2}=\frac{1}{2}\int_{0}^{1}|\tau'|^{2}+\frac{\lambda}{2}Pe^{2}-H_{\tau,\lambda}(\mathbf{u},\theta)\label{eq:energy-dissapation-w2w}
\end{equation}
where $H_{\tau,\lambda}$ is the quadratic form
\[
H_{\tau,\lambda}(\mathbf{u},\theta)=\fint_{\Omega}\frac{\lambda}{2}|\nabla\mathbf{u}|^{2}+\frac{1}{2}|\nabla\theta|^{2}+w\theta\tau'.
\]
If $H_{\tau,\lambda}\geq0$ for all divergence-free vector fields
$\mathbf{u}(\mathbf{x})$ and scalar fields $\theta(\mathbf{x})$
vanishing at $\partial\Omega$, the dissipation equation \prettyref{eq:energy-dissapation-w2w}
implies that
\[
\left\langle |\nabla T|^{2}\right\rangle \leq\int_{0}^{1}|\tau'|^{2}+\lambda Pe^{2}.
\]
Thus,
\begin{equation}
Nu\leq\inf_{\substack{\tau(z),\lambda\\
\tau(0)=1,\tau(1)=0\\
H_{\tau,\lambda}\geq0
}
}\left\{ \int_{0}^{1}|\tau'|^{2}+\lambda Pe^{2}\right\} \label{eq:backgroundmethod_w2w}
\end{equation}
In parallel with the background method discussed above, we 
refer to background fields $\tau$ satisfying $H_{\tau,\lambda}\geq0$
as being \emph{spectrally stable at Lagrange multiplier $\lambda$.} 

On the other hand, the symmetrization method from \prettyref{sec:aprioribds}
yields the bound
\begin{equation}
\sup_{\substack{\mathbf{u}(\mathbf{x},t)\\
\left\langle |\nabla\mathbf{u}|^{2}\right\rangle =Pe^{2}\\
\mathbf{u}|_{\partial\Omega}=0
}
}\,Nu(\mathbf{u})\leq\inf_{\substack{\eta(\mathbf{x})\\
\eta|_{z=1}=0,\eta|_{z=1}=1
}
}\left\{ \fint_{\Omega}|\nabla\eta|^{2}+Pe^{2}\sup_{\substack{\mathbf{u}(\mathbf{x})\\
\mathbf{u}|_{\partial\Omega}=\mathbf{0}\\
\fint_{\Omega}|\nabla\mathbf{u}|^{2}=1
}
}\,\fint_{\Omega}|\nabla\Delta^{-1}\text{div}(\mathbf{u}\eta)|^{2}\right\} \label{eq:symmmethod_w2w}
\end{equation}
when carried out optimally. As it turns out, these bounds are one and the same. 
\begin{lem}
Let $U_{\emph{bm}}(Pe)$ and $U_{\emph{symm}}(Pe)$ denote the optimal values appearing
on the righthand sides of \prettyref{eq:backgroundmethod_w2w} and
\prettyref{eq:symmmethod_w2w}, respectively. We have that $U_{\emph{bm}}=U_{\emph{symm}}$.
\end{lem}
\begin{rem}
As the following proof shows, the minimization in \prettyref{eq:symmmethod_w2w}
can be performed over $\eta$ depending $z$ alone without changing the resulting value.
\end{rem}
\begin{proof}
We prove this in two steps: first we show that $U_{\text{symm}}\leq U_{\text{bm}}$
and then we prove the reverse inequality. In both cases, we will use
the fact that
\begin{equation}
\int_{\Omega}|\nabla\Delta^{-1}\text{div}\,\mathbf{m}|^{2}=\sup_{\substack{\theta(\mathbf{x})\\
\theta|_{\partial\Omega}=0
}
}\,\int_{\Omega}2\mathbf{m}\cdot\nabla\theta-|\nabla\theta|^{2}\label{eq:duality_equality}
\end{equation}
for all $\mathbf{m}\in L^{2}(\Omega;\mathbb{R}^{3})$. 

We begin by showing that $U_{\text{symm}}\leq U_{\text{bm}}$. Taking $\mathbf{m}=\mathbf{u}\tau$
in \prettyref{eq:duality_equality}, we see that a background field
$\tau(z)$ satisfies $H_{\tau,\lambda}\geq0$ if and only if
\[
\int_{\Omega}|\nabla\Delta^{-1}\text{div}\,\mathbf{u}\tau|^{2}\leq\lambda\int_{\Omega}|\nabla\mathbf{u}|^{2}
\]
for all divergence-free $\mathbf{u}$ that vanish at $\partial\Omega$.
Therefore, 
\[
U_{\text{bm}}\geq\inf_{\substack{\tau(z),\lambda\\
\tau(0)=1,\tau(1)=0\\
H_{\tau,\lambda}\geq0
}
}\left\{ \int_{0}^{1}|\tau'|^{2}+Pe^{2}\sup_{\substack{\mathbf{u}(\mathbf{x})\\
\mathbf{u}|_{\partial\Omega}=\mathbf{0}\\
\fint_{\Omega}|\nabla\mathbf{u}|^{2}=1
}
}\,\fint_{\Omega}|\nabla\Delta^{-1}\text{div}(\mathbf{u}\tau)|^{2}\right\} \geq U_{\text{symm}}
\]
since enlarging the admissible set only decreases the resulting minimal
value.

Now we prove that $U_{\text{symm}}\geq U_{\text{bm}}$. Parameterizing the admissible
set from \prettyref{eq:symmmethod_w2w} via the level sets of 
\[
M(\eta)=\sup_{\substack{\mathbf{u}(\mathbf{x})\\
\mathbf{u}|_{\partial\Omega}=\mathbf{0}\\
\fint_{\Omega}|\nabla\mathbf{u}|^{2}=1
}
}\,\fint_{\Omega}|\nabla\Delta^{-1}\text{div}(\mathbf{u}\eta)|^{2},
\]
we can write that 
\[
U_{\text{symm}}=\inf_{\lambda}\inf_{\substack{\eta(\mathbf{x})\\
\eta|_{z=1}=0,\eta|_{z=1}=1\\
M(\eta)=\lambda
}
}\left\{ \fint_{\Omega}|\nabla\eta|^{2}+Pe^{2}\lambda\right\} .
\]
Extend the definition of $H_{\tau,\lambda}$ to functions of $\mathbf{x}$
by taking
\[
H_{\eta,\lambda}(\mathbf{u},\theta)=\fint_{\Omega}\frac{\lambda}{2}|\nabla\mathbf{u}|^{2}+\frac{1}{2}|\nabla\theta|^{2}+\mathbf{u}\theta\cdot\nabla\eta.
\]
By \prettyref{eq:duality_equality},
\[
M(\eta)=\lambda\iff\inf_{\substack{\mathbf{u}(\mathbf{x}),\theta(\mathbf{x})\\
\mathbf{u}|_{\partial\Omega}=\mathbf{0},\theta|_{\partial\Omega}=0
}
}\,H_{\eta,\lambda}(\mathbf{u},\theta)=0
\]
and the latter happens if and only if $H_{\eta ,\lambda }\geq 0$.
Using that
$\fint_{\Omega}|\nabla\eta|^{2}$ is convex in $\eta$ and that $\{\eta:H_{\eta,\lambda}\geq0\}$
is also convex, we can replace $\eta$ with its periodic average
$\tau=\overline{\eta}$ to deduce that
\[
U_{\text{symm}}\geq\inf_{\lambda}\inf_{\substack{\tau(z)\\
\tau|_{z=1}=0,\tau|_{z=1}=1\\
H_{\tau,\lambda}\geq0
}
}\left\{ \fint_{\Omega}|\tau'|^{2}+Pe^{2}\lambda\right\} =U_{\text{bm}}
\]
as desired.
\end{proof}

\subsection{On the realizability of optimal heat transport by buoyancy-driven convection\label{sec:momentumeqn}}

We return to the full system \prettyref{eq:adv-diff-eqn}-\prettyref{eq:div-free-1} now. One may wonder if buoyancy forces are capable of producing flows,
time-dependent or steady, that realize near-optimal heat transport.
The answer depends upon the way in which flow intensity is constrained.

First, we note that the energy-constrained wall-to-wall optimal transport problem corresponds
to RBC in a fluid saturated porous layer where the Navier-Stokes momentum
equation \prettyref{eq:momentum-eqn} is replaced by Darcy's law.
This implies the balance law $\left\langle |\mathbf{u}|^{2}\right\rangle =Ra(Nu-1)$
which, when combined with the result of \prettyref{thm:mainbounds_w2w_eng},
yields the optimal scaling $Nu\sim Ra$ in this setting. Direct numerical simulations
of time-dependent high-$Ra$ porous medium convection \cite{hewitt2012ultimate,otero2004high}
are consistent with this scaling, indicating that buoyancy forces
\emph{can} produce flows realizing optimal heat transport insofar
as scaling is concerned. On the other hand, asymptotic and numerical
investigations indicate that the best possible transport by steady
flows satisfies $Nu\sim Ra^{0.6}$ \cite{wen2015structure}.

Second, we observe that the enstrophy-constrained optimal transport problem corresponds to
Rayleigh's original model of buoyancy-driven convection in a fluid
layer \cite{Rayleigh1916}. There, steady convection also appears
to be strongly sub-optimal with the highest computationally observed
scaling being $Nu\sim Ra^{0.31}$ \cite{sondak2015optimal,waleffe2015heat}.
To date, there are no turbulent high-$Ra$ direct numerical simulations
indicating heat transport scaling much higher than $Nu\sim Ra^{1/3}$.

We close our discussion of fluid dynamical implications by commenting
on the certain sub-optimality of heat transport in Rayleigh's original model.
Rayleigh imposed \prettyref{eq:adv-diff-eqn}-\prettyref{eq:div-free-1}
in two-dimensions with stress-free velocity boundary conditions and
the usual Dirichlet temperature ones. Although RBC in a fluid layer
must obey the bound $Nu\lesssim Ra^{1/2}$ in any dimension and for
any boundary conditions, the result of \cite{whitehead2011ultimate}
is that in two-dimensions and with stress-free boundaries
$Nu\lesssim Ra^{5/12}$. Nevertheless, by combining the relevant balance law $\left\langle |\nabla\mathbf{u}|^{2}\right\rangle =Ra(Nu-1)$ implied by the Navier-Stokes momentum equation \prettyref{eq:momentum-eqn} with the result of \prettyref{thm:mainbounds_w2w_ens} and the remark immediately thereafter, we conclude that optimal heat transport in the setting of Rayleigh's model must satisfy $Nu\sim Ra^{1/2}$ (up to logarithmic corrections). 
Our analysis is consistent with all the requirements of Rayleigh's model except for the Navier-Stokes momentum equation \prettyref{eq:momentum-eqn}.
Thus, buoyancy-driven convection in two-dimensions between stress-free
boundaries must yield strongly sub-optimal rates of heat transport
as compared with what happens if \prettyref{eq:momentum-eqn} is not
imposed. This underscores the importance of using the momentum equation \textemdash rather
than only a balance law it implies \textemdash for determining the
asymptotic heat transport of turbulent RBC.


\section{Optimal transport as energy-driven pattern formation \label{sec:energy-driven-pattern-formation}}

There is a second scientific context, other than the fluid dynamical one, in which the methods behind our analysis of wall-to-wall optimal heat transport have played a fundamental role. This is the subject of ``energy-driven pattern formation'' in mathematical materials science  \cite{kohn2007energy}. 

Perhaps the key methodological contribution of this paper is the reformulation
of the general steady wall-to-wall optimal transport problem
\begin{equation}
\sup_{\substack{\mathbf{u}(\mathbf{x})\\
||\mathbf{u}||=Pe\\
+b.c.
}
}\,Nu(\mathbf{u})\label{eq:steady_w2w_pblm}
\end{equation}
in its integral form
\begin{equation}
\inf_{\substack{\mathbf{u}(\mathbf{x}),\xi(\mathbf{x})\\
\fint_{\Omega}w\xi=1\\
+b.c.
}
}\,\fint_{\Omega}|\nabla\Delta^{-1}\text{div}(\mathbf{u}\xi)|^{2}+\epsilon||\mathbf{u}||^{2}\fint_{\Omega}|\nabla\xi|^{2}.\label{eq:integral_formulation}
\end{equation}
This change of viewpoint, accomplished in \prettyref{sec:OptlDesignProb_w2w},
hinges on the fact that the Nusselt number of a steady velocity field
$\mathbf{u}$ can be written as the maximal value of a certain non-local
functional in $\xi$. The resulting problem \prettyref{eq:integral_formulation}
is equivalent to the original one \prettyref{eq:steady_w2w_pblm},
and optimizers correspond. In the examples of energy- and enstrophy-constrained
optimal transport considered in \prettyref{sec:energy-constraineddesign}
and \prettyref{sec:enstrophy_constrained_design}, where $||\cdot||$
is the (volume-averaged) $L^{2}$- or $\dot{H}^{1}$-norm, the integral
formulation \prettyref{eq:integral_formulation} plays a key role
in the construction of divergence-free velocity fields that achieve
nearly optimal transport. As that analysis shows, the complexity of
the successful construction \textemdash whether it can be described
using few lengthscales or many \textemdash depends strongly on the
choice of norm. 

Besides its practical use for the estimation of optimal transport,
\prettyref{eq:integral_formulation} shares striking similarities
with other non-convex and singularly perturbed variational problems
from mathematical materials science. The study of patterns selected
by energy minimization principles in this field is known as energy-driven pattern
formation. It is important to note that the wall-to-wall optimal transport problem is variational by definition.
Thus, our observation is not that there exists some variational formulation
for it, but rather that the specific formulation \prettyref{eq:integral_formulation}
reminds of various model problems from energy-driven pattern formation.
From this point of view, it is no surprise that the (nearly) optimal
patterns constructed in this paper for wall-to-wall transport \textemdash convection
rolls and branching flows \textemdash bear similarities with other
well-appreciated patterns from materials science including domain
branching in micromagnetics \cite{choksi1998bounds,choksi1999domain}
and wrinkling cascades in thin elastic sheets \cite{belgacem2000rigorous,jin2001energy,ortiz1994morphology}.
What \prettyref{eq:integral_formulation} offers is a functional analytic
framework in which to make such connections precise. 

We discuss below two model problems from energy-driven pattern formation
and their connections to wall-to-wall optimal transport. We leave
their general scientific introduction to the references therein, focusing
instead on the salient features of their analysis. This discussion
provides an alternate viewpoint on the role of branching patterns
in the variational analysis of transport, which complements the older
purely fluid dynamical arguments of Busse \cite{busse1969howards}.
We hope these remarks prove useful to the reader interested in our
approach.

\subsection{Magnetic domain branching in a uniaxial ferromagnet}

Our first example comes from micromagnetics and concerns the patterns
formed by magnetic domains in a uniaxial ferromagnet. The energetic
description is as follows. We take as the magnet the domain $\Omega=(-L,L)_{x}\times[0,1]_{y,z}^{2}$
where $x$ is the the preferred direction of magnetization and $L$
is the magnet's (non-dimensionalized) length. On $\Omega$ we define
a magnetization vector field $\mathbf{m}(\mathbf{x})=m_{1}\hat{i}+m_{2}\hat{j}+m_{3}\hat{k}$
which is required to be of unit size $|\mathbf{m}|=1$, and is extended
by zero to the rest of space $\mathbb{R}^{3}\backslash\Omega$. The
micromagnetic energy that results is
\begin{equation}
\int_{\text{all space}}|\nabla\Delta^{-1}\text{div}\,\mathbf{m}|^{2}+\int_{\text{magnet}}Q(1-m_{1}^{2})+\epsilon|\nabla\mathbf{m}|\label{eq:micromag_energy}
\end{equation}
where the divergence is understood in the distributional sense. Strictly
speaking, this is a ``sharp interface'' model in which the total
variation norm
\[
\int_{\Omega}|\nabla\mathbf{m}|=\sum_{i=1}^{3}\int_{\Omega}|\nabla m_{i}|=\sum_{i=1}^{3}\sup_{\substack{\mathbf{v}\in C_{c}^{1}(\Omega;\mathbb{R}^{n})\\
||\mathbf{v}||_{L^{\infty}(\Omega)}\leq1
}
}\,\int_{\Omega}m_{i}\text{div}\,\mathbf{v}
\]
features instead of the $\dot{H}^{1}$-norm (for more on this reduction
see \cite{choksi1998bounds}). The first term appearing in \prettyref{eq:micromag_energy}
is called the magnetostatic energy; it accounts for the cost of the
magnetic field induced by $\mathbf{m}$ in the ambient space. The
second term is the anisotropy energy and it arises from an underlying
crystalline anisotropy which prefers $\mathbf{m}$ to be $\pm\hat{i}$.
The third term is the interfacial energy. It permits $\mathbf{m}$
to be discontinuous, but limits the total area of any interfaces across
which $\mathbf{m}$ jumps. The parameters $Q$ and $\epsilon$ set
the relative strengths of these effects. The magnetostatic and interfacial
energies have direct analogs in the wall-to-wall problem \prettyref{eq:integral_formulation};
the anisotropy term does not. Note that, due to the constraint $|\mathbf{m}|=1$,
this functional is non-convex.

There are various designs for $\mathbf{m}$ one can entertain in minimizing
\prettyref{eq:micromag_energy}. One is the so-called Kittel structure,
in which $\mathbf{m}$ is independent of $x$ and $\pm\hat{i}$-valued
throughout the magnet, alternating between these at some to be determined
lengthscale $l$ in the $yz$-plane. This design costs no anisotropic
energy and the optimal $l$ is selected by minimizing its magnetostatic
and interfacial costs. Another important design is the Landau-Lifshitz
structure, in which $\mathbf{m}$ is independent of $x$ and $\pm\hat{i}$-valued
except for in a thin boundary layer near $x=\pm L$. There, it is
taken to be perpendicular to $\hat{i}$ in such a way as to eliminate
the magnetostatic energy completely, thus coupling the thickness of
the boundary layer to the lengthscale $l$ of oscillations in the
bulk. This is a sharp-interface version of the convection roll design
described in \prettyref{sec:energy-constraineddesign}. Finally, there
is the Privorotski\u{\i} construction, which plays the role of the branching
flows from \prettyref{sec:enstrophy_constrained_design}. It too involves
a very large number of distinct lengthscales which interpolate between
a preferred lengthscale in the bulk $l_{\text{bulk}}$ and a significantly
smaller one at the boundary $l_{\text{bl}}$. We refer the reader for more
details to \cite{choksi1998bounds,choksi1999domain} including a description
of the relevant regimes.

What can be proved regarding this non-convex, non-local minimization
problem? Following the reference \cite{choksi1999domain} we assume
that $\mathbf{m}(x,y,z)$ is periodic in $(y,z)$ and identify $[0,1]_{y,z}^{2}$
with $\mathbb{T}_{y,z}^{2}$. Then there exist positive constants
$C$ and $C'$ such that the minimum micromagnetic energy satisfies
\[
CQ^{1/3}\epsilon^{2/3}L^{1/3}\leq\text{minimum micromagnetic energy}\leq C'Q^{1/3}\epsilon^{2/3}L^{1/3}
\]
for all sufficiently large $Q$ and sufficiently small $\epsilon/L$.
The proof of this result requires two kinds of arguments. The upper
bound comes from estimating the cost of an optimal Privorotskii construction
(the conditions on $Q$, $\epsilon$, and $L$ ensure that the result
is significantly less than those obtained by the Kittel and Landau-Lifshitz
structures). The lower bound asserts that the Privorotskii construction
cannot be beat as far as scaling is concerned. The original proof
of it can be found in \cite{choksi1999domain}, but we note the existence
of a second more recent proof in \cite{cinti2016interpolation} which
utilizes the end-point Gagliardo-Nirenberg interpolation inequality
\[
||f||_{L^{4/3}(\mathbb{T}^{2})}\lesssim||\nabla f||_{L^{1}(\mathbb{T}^{2})}^{1/2}||f||_{H^{-1}(\mathbb{T}^{2})}^{1/2}
\]
holding for all mean-zero and periodic functions $f$. 

\subsection{Blistering patterns in thin elastic sheets}

Our second example comes from elasticity theory. Consider a thin elastic
sheet of (non-dimensional) thickness $h$ which is strongly bonded
to the top of a large rubber block, except for on some known sub-domain
$\Omega\subset\mathbb{R}^{2}$. Applying biaxial compression to the
block causes the sheet to blister in the unbonded domain. The result
is a complex pattern of wrinkles and folds whose details can be modeled
through the minimization of a certain non-convex and singularly perturbed
variational problem. As in \cite{belgacem2000rigorous,jin2001energy},
we consider minimization of the internal elastic energy under clamped
boundary conditions. In the F\"oppl-von Karman model, the elastic
energy (per unit thickness) is given by
\begin{equation}
\int_{\substack{\text{blistered}\\
\text{region}
}
}|e(\mathbf{v})+\frac{1}{2}\nabla\phi\otimes\nabla\phi|^{2}+h^{2}|\nabla\nabla\phi|^{2}\label{eq:FvK}
\end{equation}
where the ``in-plane'' displacement parallel to the top of the block
is $\mathbf{v}(\mathbf{x})$ and the ``out-of-plane'' displacement
perpendicular to it is $\phi(\mathbf{x})$. Here, $e(\mathbf{v})$
denotes the symmetric part of the in-plane displacement gradient $\nabla\mathbf{v}$.
Taken together, the in- and out-of-plane displacements yield the map
$(\mathbf{x},0)\mapsto(\mathbf{x}+\mathbf{v}(\mathbf{x}),\phi(\mathbf{x}))$
which describes the deformation of the blister. At the edge of the
blister $\partial\Omega$ we impose the clamped boundary conditions
\[
\mathbf{v}|_{\partial\Omega}=-\lambda\mathbf{x},\ \phi|_{\partial\Omega}=0,\ \text{and}\ \partial_{\nu}\phi|_{\partial\Omega}=0.
\]
The parameter $\lambda$ is positive and sets the amount of overall
compressive strain. The first term in the energy is called the membrane
term. It prefers the in-plane strain $e(\mathbf{v})+\frac{1}{2}\nabla\phi\otimes\nabla\phi$
to vanish. The second one is called the bending term, and it prefers
the out-of-plane displacement to vary on longer lengthscales or not
at all. The relative strength of these effects is determined by the
parameter $h$, which is understood to be small.

There are significant parallels between the elastic energy functional
\prettyref{eq:FvK} and the integral formulation of wall-to-wall transport
\prettyref{eq:integral_formulation}. Of course, the bending term
from \prettyref{eq:FvK} and the higher order terms from \prettyref{eq:integral_formulation}
act to regularize designs. More interestingly, we observe a similarity
between the membrane term from \prettyref{eq:FvK} and the advection
term and net-flux constraint from \prettyref{eq:integral_formulation}.
Let us introduce a streamfunction $\psi$ for the divergence-free
velocity field $\mathbf{u}$ (we work with a two-dimensional fluid
layer now) and rewrite the advection term as
\begin{equation}
\fint_{\text{fluid layer}}|\nabla\Delta^{-1}\text{div}(\mathbf{u}\xi)|^{2}=\fint_{\text{fluid layer}}|\nabla\Delta^{-1}J(\psi,\xi)|^{2}\label{eq:-1scaling}
\end{equation}
where $J(\psi,\xi)=\nabla^{\perp}\psi\cdot\nabla\xi$. Recall also
that the net-flux constraint requires 
\[
\fint_{\text{fluid layer}}w\xi=1.
\]
As pointed out in \prettyref{sec:An-integral-formulation} \textemdash see
the discussion surrounding \prettyref{eq:div-free} \textemdash for
smooth enough designs $(\mathbf{u},\xi)$ the advection term cannot
vanish while the net-flux constraint and boundary conditions $w|_{\partial\Omega}=\xi|_{\partial\Omega}=0$
hold. As \prettyref{eq:integral_formulation} makes clear, wall-to-wall
optimal transport is precisely about balancing these competing effects.
Regarding elasticity, we ask: what does it take for the membrane term
to nearly vanish? This can be answered with the aid of the lower bound
\[
\fint_{\substack{\text{blistered}\\
\text{region}
}
}|e(\mathbf{v})+\frac{1}{2}\nabla\phi\otimes\nabla\phi|^{2}\gtrsim\left|\fint_{\substack{\text{blistered}\\
\text{region}
}
}\frac{1}{2}\nabla\phi\otimes\nabla\phi-\lambda\text{Id}_{2\times2}\right|^{2} +\fint_{\substack{\text{blistered}\\
\text{region}
}
}|\nabla\nabla(\Delta\Delta)^{-1}\det\nabla\nabla\phi|^{2},\label{eq:-2scaling}
\]
which is sharp for certain domains and boundary conditions.\footnote{Equality can be achieved if a suitable Airy stress function solving
$\Delta\Delta\Psi=\det\nabla\nabla\phi$ with appropriate boundary
conditions exists.} For the in-plane strain to nearly vanish, the bulk average of $\frac{1}{2}\nabla\phi\otimes\nabla\phi$
must be nearly constant and equal to a known multiple of the identity.
At the same time, $\phi$ must nearly satisfy the degenerate Monge-Amp\`ere
equation $\det\nabla\nabla\phi=0$. It follows from the results of
\cite{pakzad2004sobolev} that these are incompatible constraints,
i.e., the membrane term cannot vanish while the bending term remains
finite. The situation is remarkably similar to that of wall-to-wall
optimal transport. 

The scaling law of the minimum energy for blistering is known. As
proved in \cite{jin2001energy,belgacem2000rigorous}, there exist
constants $C$ and $C'$ depending only on $\Omega$ so that 
\begin{equation}
Ch\leq\text{minimum elastic energy}\leq C'h\label{eq:min-energy-est_FvK}
\end{equation}
for small enough $h$. The upper bound comes from a branching construction
involving finer and finer oscillations in $\nabla\phi$ at a certain
lengthscale depending on the distance from the blister edge $\partial\Omega$.
As opposed to the corresponding result for the wall-to-wall problem,
there is no logarithmic correction to scaling in \prettyref{eq:min-energy-est_FvK}.
This can be explained with the help of \prettyref{eq:-1scaling} and
\prettyref{eq:-2scaling}: whereas the advection term has a $-1$
scaling in its quadratic nonlinearity $J(\psi,\xi)$, the membrane
term has a $-2$ scaling in $\det\nabla\nabla\phi$ and therefore
permits much stronger oscillations. As a result, branching can be
more easily accommodated in blistering than in optimal transport.
The lower bound from \prettyref{eq:min-energy-est_FvK} asserts that
branching indeed achieves the minimum energy up to a prefactor depending
only on the domain. Its proof reminds of the proof of Howard's lower
bound given after \prettyref{eq:Howard_LB}. For details we refer
the reader to \cite{jin2001energy} for the case where $\Omega$ is
a square with periodic boundary conditions at opposite sides, and
to \cite{belgacem2000rigorous} for the more general case of an arbitrary
domain $\Omega$ with suitably smooth boundary.

\bibliographystyle{plain}
\bibliography{wall2wallrefs}

\end{document}